\documentclass[reqno]{amsart}
\usepackage{amsmath,amssymb,color}
\usepackage[mathscr]{euscript}
                 
\newtheorem{theorem}{Theorem}[section]
\newtheorem{lemma}[theorem]{Lemma}
\newtheorem{proposition}[theorem]{Proposition}
\newtheorem{corollary}[theorem]{Corollary}
\newcommand{\Glimsup}{\mathop{\textrm{$\Gamma\!$--$\limsup$}}\displaylimits}
\newcommand{\Gliminf}{\mathop{\textrm{$\Gamma\!$--$\liminf$}}\displaylimits}

\numberwithin{equation}{section}

\newcommand{\mc}[1]{{\mathcal #1}}
\newcommand{\mf}[1]{{\mathfrak #1}}
\newcommand{\mb}[1]{{\mathbf #1}}
\newcommand{\bb}[1]{{\mathbb #1}}
\newcommand{\ms}[1]{{\mathscr #1}}

\newcommand{\<}{\langle}
\renewcommand{\>}{\rangle}

\renewcommand{\epsilon}{\varepsilon}
\renewcommand{\phi}{\varphi}

\renewcommand{\hat}{\widehat}
\renewcommand{\tilde}{\widetilde}

\newcommand{\upbar}[1]{\,\overline{\! #1}}

\title[Burgers equation in a bounded interval]{Action functional and
quasi-potential for the Burgers equation in a bounded interval}

\begin{document}

\begin{abstract}
  Consider the viscous Burgers equation $u_t + f(u)_x = \varepsilon\,
  u_{xx}$ on the interval $[0,1]$ with the inhomogeneous Dirichlet
  boundary conditions $u(t,0) = \rho_0$, $u(t,1) = \rho_1$.  The flux
  $f$ is the function $f(u)= u(1-u)$, $\varepsilon>0$ is the
  viscosity, and the boundary data satisfy $0<\rho_0<\rho_1<1$.  We
  examine the quasi-potential corresponding to an action functional,
  arising from non-equilibrium statistical mechanical models,
  associated to the above equation.  We provide a static variational
  formula for the quasi-potential and characterize the optimal paths
  for the dynamical problem.  In contrast with previous cases, for
  small enough viscosity, the variational problem defining the quasi
  potential admits more than one minimizer. This phenomenon is
  interpreted as a non-equilibrium phase transition and corresponds to
  points where the super-differential of the quasi-potential is not a
  singleton.
\end{abstract}

\author [L. Bertini] {Lorenzo Bertini}
\address{\noindent Lorenzo Bertini \hfill\break\indent 
Dipartimento di Matematica, Universit\`a di Roma `La Sapienza' 
\hfill\break\indent 
P.le Aldo Moro 2, 00185 Roma, Italy}
\email{bertini@mat.uniroma1.it}

\author[A. De Sole]{Alberto De Sole}
\address{\noindent Alberto De Sole \hfill\break\indent 
Dipartimento di Matematica, Universit\`a di Roma `La Sapienza' 
\hfill\break\indent 
P.le Aldo Moro 2, 00185 Roma, Italy}
\email{desole@mat.uniroma1.it}

\author[D. Gabrielli]{Davide Gabrielli}
\address{\noindent Davide Gabrielli \hfill\break\indent 
 Dipartimento di Matematica, Universit\`a dell'Aquila
\hfill\break\indent 
67100 Coppito, L'Aquila, Italy
}
\email{gabriell@univaq.it}

\author[G. Jona-Lasinio]{Giovanni Jona-Lasinio}
\address{\noindent Giovanni Jona-Lasinio
\hfill\break\indent 
 Dipartimento di Fisica and INFN, Universit\`a di Roma La Sapienza
\hfill\break\indent 
 P.le A.\ Moro 2, 00185 Roma, Italy
}
\email{gianni.jona@roma1.infn.it}

\author[C. Landim]{Claudio Landim} 
\address{Claudio Landim
  \hfill\break\indent IMPA \hfill\break\indent Estrada Dona Castorina
  110, \hfill\break\indent
J. Botanico, 22460 Rio de Janeiro, Brazil\hfill\break\indent
  {\normalfont and} \hfill\break\indent CNRS UMR 6085, Universit\'e de
  Rouen, \hfill\break\indent Avenue de l'Universit\'e, BP.12,
  Technop\^ole du Madril\-let, \hfill\break\indent
F76801 Saint-\'Etienne-du-Rouvray, France.} 
\email{landim@impa.br}

\maketitle

\thispagestyle{empty}

\section{Introduction}
\label{sec0}

We consider an infinite dimension version of the classical
Freidlin-Wentzell variational problem for the quasi-potential.  We
first briefly recall this topic in the context of diffusion processes
in $\bb R^n$ \cite{FW}. Let $b$ be a smooth vector field in $\bb R^n$
and consider the stochastic perturbation of the dynamical systems
$\dot x = b(x)$ given by
\begin{equation*}
\dot X_\gamma = b(X_\gamma) +\sqrt{2 \gamma} \, \dot w\;,
\end{equation*}
where $\dot w$ is a white noise. Under suitable assumptions on $b$,
the process $X_\gamma$ has a unique invariant measure $\mu_\gamma$.
The Freidlin-Wentzell theory provides a variational expression for the
asymptotics of this invariant measure in the weak noise limit
$\gamma\downarrow 0$. To each path $X:(-\infty,0]\to \bb R^n$
associate the \emph{action}
\begin{equation}
\label{action-fd}
I(X) \;=\; \frac 14 \int_{-\infty}^0 \big|\dot X(t) - b(X(t))\big|^2 
\,dt \;.
\end{equation}
In the sequel we assume that the vector field $b$ has a unique,
globally attractive, equilibrium point $\upbar x$.  The
quasi-potential $V:\bb R^n\to [0,+\infty)$ is defined by
\begin{equation}
\label{qpfw}
V(x) \;=\; \inf\,\big\{I(X)\,,\: X(0)=x\,,\: X(t)\to \upbar x 
\textrm{ as } t\to-\infty \big\}\;.
\end{equation}
Namely, $V(x)$ is the minimal action to reach $x$ starting from the
equilibrium point $\upbar x$.
For each Borel set $B\subset \bb R^n$ we then have, as
$\gamma\downarrow 0$,
\begin{equation*}
\mu_\gamma(B) \;\asymp\; \exp\big\{-\gamma^{-1} \inf_{x\in B} V(x) 
\big\}\;.
\end{equation*}
If the vector field $b$ is conservative, namely $b=-\nabla U$ for some
$U:\bb R^n\to \bb R$, then $V(x)=U(x)-U(\upbar x)$, i.e.\ the
quasi-potential coincides with the potential. In general, though,
there is no simple expression for the quasi-potential.

In this finite dimensional setting, the quasi-potential $V$ is
Lipschitz \cite{FW}, in particular it is a.e.\ differentiable. As
discussed in \cite{DD}, for some ``special'' points $x\in \bb R^n$,
the function $V$ might however have ``corners'', namely the
super-differential of $V$ might not be a singleton.  From a dynamical
point of view, for such points $x$ there would exist more than a
single minimizer for the variational problem \eqref{qpfw}.  We refer
to the examples considered in the physical literature \cite{GT,J,MS}
for a discussion on the physical interpretation of this lack of
uniqueness.  In this article we show that this phenomenon occurs for
an infinite dimensional dynamical system. As far as we know, this is
the first concrete example in which such a result is analytically
proven.

As shown in \cite{D,DD,Pert}, if $\bb H$ denotes the Hamiltonian
associated to the action \eqref{action-fd}, then the quasi-potential
$V$ is a viscosity solution, the correct PDE formulation of the
variational problem \eqref{qpfw} in presence of ``corners'', to the
Hamilton-Jacobi equation $\bb H(x,DV)=0,\,x\in\bb R^n$.

We examine in this article an infinite dimensional version of the
previous variational problem.  As the basic dynamical system, we
consider the following non-linear parabolic equation on the interval
$[0,1]$
\begin{equation}
\label{lln}
u_t + f(u)_x = \epsilon \big( D(u) u_x \big)_x
\end{equation}
with the inhomogeneous Dirichlet boundary condition $u(t,0)=\rho_0$,
$u(t,1)=\rho_1$. In the above equation, $f$ is the flux, $D>0$ the
diffusion coefficient and $\epsilon>0$ the viscosity.  To introduce
the associated action functional, add an external ``controlling''
field $E=E(t,x)$ to obtain the perturbed equation
\begin{equation}
\label{pbe}
u_t + f(u)_x +2\epsilon \big( \sigma(u) E\big)_x 
= \epsilon \big( D(u) u_x \big)_x \;,
\end{equation}
where $\sigma(u)\ge 0 $ is the mobility. Denote by $u^E$ the solution
of this equation. The action of a path $u:(-\infty,0]\times [0,1]\to
\bb R$ is given by
\begin{equation}
\label{f14}
I^\epsilon(u) \;=\; \inf
\epsilon \int_{-\infty}^0 \int_0^1 \sigma(u) \, E^2 \, dx\,dt\;,
\end{equation}
where the infimum is carried over all $E$ such that $u^E=u$.  The
quasi-potential is then introduced as in the finite dimensional
setting, namely $V_\epsilon$ is the functional on the set of functions
$\rho:[0,1]\to \bb R$ defined by
\begin{equation}
\label{qpgen}
V_\epsilon(\rho) = \inf\, \big\{ I^\epsilon(u)\,:\: u(0)=\rho\,,\:
u(t)\to\upbar\rho_\epsilon \textrm{ as } t\to -\infty \big\}\;,
\end{equation}
where $\upbar\rho_\epsilon$ is the unique stationary solution of
\eqref{lln}. 

Informally, as in the Freidlin-Wentzell theory, one can add a
stochastic perturbation to \eqref{lln},
\begin{equation}
  \label{lln-stoc}
  u^\gamma_t + f(u^\gamma)_x = \epsilon \big( D(u^\gamma) u^\gamma_x \big)_x
  +\big(\sqrt{2\,\epsilon\,\gamma\,\sigma(u^\gamma)}\:
  \dot{w}_\gamma\big)_x\;, 
\end{equation}
where $\dot w_\gamma$ is a white noise in time that becomes also white
in space as $\gamma\downarrow 0$.  Then, informally, the finite
dimensional theory carries over to the present setting.  In
particular, the quasi-potential $V_\epsilon$ in \eqref{qpgen}
describes the asymptotics for the invariant measure of the process
$u^\gamma$ as $\gamma\downarrow 0$.  We refer to \cite{mauro} for an
analysis of the large deviation properties of the stochastic PDE
\eqref{lln-stoc}, with periodic boundary conditions, in the joint
limit $\gamma\downarrow 0$ and  $\epsilon\downarrow 0$.

Our main motivation for the analysis of the action functional
\eqref{f14} and the quasi-potential \eqref{qpgen} comes, however, from
non-equilibrium statistical mechanics. For a class of interacting
particle systems, the so-called stochastic lattice gases, it has been
shown that equation \eqref{lln} describes the typical evolution of the
empirical density in the diffusive scaling limit \cite{KL}.  Moreover,
the functional $I^\epsilon$, as in the Freidlin-Wentzell theory, gives
the corresponding asymptotic probability of observing deviations from
the typical behavior \cite{blm1,KL,KOV}.
 
In the case of equilibrium models, the flux $f$ vanishes and the
boundary conditions are equal. For such models, which are analogous to
gradient vector fields in the finite dimensional situation, the
quasi-potential \eqref{qpgen} does not depend on the viscosity
$\epsilon$ and coincides with the thermodynamic free energy functional
of the underlying microscopic model, whose invariant measure has the
standard Gibbs form.  If the flux $f$ does not vanishes but the
boundary data are still equal, $\rho_0=\rho_1$, the quasi-potential
does not depend on $f$ and it is equal to the one of the corresponding
equilibrium model, see \cite{BM,bdgjl10} for the analogous result in
the case of periodic boundary conditions.  On the other hand, if the
boundary data are not equal, in general there is no simple expression
for the invariant measure of the microscopic dynamics, often called
stationary non-equilibrium state.  In order to analyze the behavior of
such invariant measure in the thermodynamic limit, in \cite{bdgjl2} we
introduced the dynamical/variational approach outlined above.  In
particular, the quasi-potential $V_\epsilon$ plays an analogous role
to the free energy for equilibrium systems.  This gives actually a
natural way to extend the notion of thermodynamic potential to
non-equilibrium systems.

As it has been shown by concrete examples
\cite{bdgjl2,bdgjl3,bgl1,BGL,d1,de}, for non-equilibrium models the
quasi-potential \eqref{qpgen} presents peculiar features.  While for
equilibrium models $V_\epsilon$ is always a convex local functional,
i.e.\ of the form $V_\epsilon(\rho) = \int_0^1 v(\rho(x))\, dx$ for
some convex real function $v$, for non-equilibrium models $V_\epsilon$
might be non-local and non-convex.  In terms of the underlying
microscopic model, the non-locality of the quasi-potential corresponds
to the presence of long-range correlations which are believed to be a
generic feature of stationary non-equilibrium states \cite{DKS}.

The main purpose of the present paper is to show, by a concrete
example, that for non-equilibrium models the quasi-potential might
have ``corners'', equivalently that the minimizer for the variational
problem \eqref{qpgen} is not unique. We shall analyze the model
defined by a constant diffusion coefficient, $D=1$, while the flux and
mobility are given by $f(u)=\sigma(u)= u(1-u)$, $u\in [0,1]$. Then the
parabolic equation \eqref{lln} becomes the viscous Burgers equation
and the action functional $I^\epsilon$ in \eqref{f14} can be obtained
as the large deviation rate functional of the so-called weakly
asymmetric simple exclusion process \cite{blm1,KOV}. As shown in
\cite{bg,f}, the quasi-potential \eqref{qpgen} is also the large
deviation rate function for the invariant measure.

When $\rho_0> \rho_1$, so that both the boundary conditions and the
flux $f$ ``push'' the density $u\in [0,1]$ to the right, the behavior
of the invariant measure has been discussed in \cite{de} by
combinatorial techniques. In particular, in \cite{de} a static
variational characterization for the rate function of the invariant
measure is derived in terms of a one-dimensional boundary value
problem. More recently, the same model is analyzed in \cite{bgl1}
where it is shown that the quasi-potential \eqref{qpgen} can be
written in terms of the variational expression derived in \cite{de}
and the optimal paths for \eqref{qpgen} are characterized. We
emphasize that in this case, i.e.\ for $\rho_0>\rho_1$, the
quasi-potential has no corners and the minimizer for \eqref{qpgen} is
unique.

In this paper we examine the same model but in the more interesting
situation in which $\rho_0<\rho_1$, so that there is an effective
competition between the flux and the boundary conditions. Our main
results are summarized as follows.  By analyzing the variational
problem \eqref{qpgen}, we establish a static variational
characterization of $V_\epsilon$ analogous to the one in
\cite{bgl1,de}. We emphasize that in the case here discussed there is
no uniqueness for the minimizer of \eqref{qpgen} and this introduces
few technical complications. We then discuss the variational
convergence of the quasi-potential $V_\epsilon$ in the inviscid limit
$\epsilon\downarrow 0$. In particular, we show that in this limit we
recover the functional derived in \cite{dls3} in the context of the
boundary driven asymmetric exclusion process.  By a perturbation
argument with respect to the limiting case $\epsilon=0$ we then show,
provided the viscosity $\epsilon$ is small enough, that there exist
functions $\rho$ such that the minimizer for \eqref{qpgen} is not
unique. We also show that in such points $\rho$ the super-differential
of $V_\epsilon$ is not a singleton.  
In the context of equilibrium statistical mechanics, the existence of
more than a single tangent functional to the quasi-potential, which in this
case coincides with the free energy functional, is due to the
occurrence of phase transitions.
We therefore interpret the fact that the super-differential of
$V_\epsilon$ is not a singleton as a non-equilibrium phase transition.
Finally, we discuss the connection of the quasi-potential $V_\epsilon$
to the Hamilton-Jacobi equation $\bb H_\epsilon(\rho,D V)=0$, where
$\bb H_\epsilon$ is the Hamiltonian associated to the action
\eqref{f14}.

\section{Notation and results}
\label{sec2}

\subsection*{Viscous Burgers equation}

Consider the viscous Burgers equation on the interval $[0,1]$ with
inhomogeneous Dirichlet boundary conditions at the endpoints namely,
\begin{equation} 
\label{eq:1}
\begin{cases}
u_t + f(u)_x =  \varepsilon u_{xx} \\
u(t,0) = \rho_0 \,,\quad  u(t,1) = \rho_1\;,
\end{cases}
\end{equation}
where $u=u(t,x)$ is a scalar function and hereafter we denote partial
derivatives with subscripts.  The \emph{flux} $f$ is the function
$f(u)= u(1-u)$, $\varepsilon>0$ is the viscosity, and the boundary
data, fixed throughout the paper, satisfy $0<\rho_0<\rho_1<1$.

Simple computations show that the unique stationary solution
$\upbar{\rho}_\varepsilon$ of the viscous Burgers equation
\eqref{eq:1} can be described as follows. Let $J_0=J_0(\rho_0,\rho_1):=
\min_{r\in[\rho_0,\rho_1]} f(r)$. For each $\varepsilon>0$ there
exists a unique $J_\varepsilon\in (-\infty, J_0)$ such that
\begin{equation*}
\int_{\rho_0}^{\rho_1} \frac{\varepsilon}{f(r)-J_\varepsilon}\, dr = 1\;.
\end{equation*}
The function $\upbar{\rho}_\varepsilon$ is then obtained by
integrating $f( \upbar{\rho}_\varepsilon ) - \varepsilon
(\upbar{\rho}_\varepsilon)_x = J_\varepsilon $
with the boundary condition $\upbar{\rho}_\varepsilon(0)=\rho_0$.  In
particular, the function $\upbar{\rho}_\varepsilon$ is strictly
increasing.  We remark that the constant $J_\varepsilon$ can be
interpreted as the current maintained by the stationary solution
$\upbar{\rho}_\varepsilon$. Let $\varphi_i := \log
[\rho_i/(1-\rho_i)]\in\bb R$, $i=0,1$, and set $\varepsilon_0 :=
1/(\varphi_1-\varphi_0)$. Clearly, $J_\varepsilon$ is increasing as 
$\varepsilon$ decreases and $J_{\varepsilon_0} = 0$; therefore
$0<J_{\varepsilon} < J_0$ for $0<\varepsilon < \varepsilon_0$ and
$J_{\varepsilon} < 0$ for $\varepsilon > \varepsilon_0$.  A simple
computation shows also that $\lim_{\varepsilon\downarrow 0}
J_\varepsilon = J_0$.

By standard arguments for parabolic equations, see e.g.\ the 
more sophisticated analysis in \cite{GK}, the stationary solution
$\upbar{\rho}_\varepsilon$ is globally attractive for the flow defined
by \eqref{eq:1}. More precisely, fix $\epsilon>0$, let $u(t;\rho)$,
$t\ge 0$, be the solution to \eqref{eq:1} with initial condition
$u(0,\cdot;\rho)=\rho(\cdot)$, and assume that $\rho: [0,1] \to \bb R$ is a 
continuous function satisfying the boundary conditions $\rho(0)=\rho_0$,
$\rho(1)=\rho_1$.  Then $u(t;\rho)$ converges in the $C^1$ topology to
$\upbar{\rho}_\varepsilon$ as $t\to\infty$.  Furthermore, this
convergence is uniform with respect to $\rho$ in a bounded set of
$C([0,1])$ and holds with an exponential rate.

\subsection*{The action functional}

To define rigorously the action functional informally introduced in
\eqref{f14}, we need to introduce some notation.  Given $T>0$, the
inner products in $L^2([0,1])$ and $L^2([-T,0]\times [0,1])$ are
denoted by $\<\cdot,\cdot\>$ and $\<\!\<\cdot,\cdot\>\!\>$
respectively.  We consider the space $L^\infty([0,1])$ equipped with
the weak* topology and let $M$ be the set 
\begin{equation*}
M\;=\; \big\{\rho\in L^\infty([0,1]):\, 0\le \rho\le 1\big\}
\end{equation*}
equipped with the relative topology. Then $M$ is a compact Polish
space, i.e.\ complete metrizable and separable.  Recall that, by
definition of the weak* topology, a sequence $\{\rho^n\} \subset M$
converges to $\rho$ in $M$ if and only if $\langle\rho^n, g\rangle \to
\langle \rho, g \rangle$ for any $g\in L^1([0,1])$. For $T>0$, we let
$C([-T,0];M)$ be the set of continuous paths $u:[-T,0] \to M$ equipped
with the topology of uniform convergence.

Let $C^\infty_{0}([-T,0]\times [0,1])$ be the space of smooth
functions $H:[-T,0]\times [0,1]\to \bb R$ satisfying $H(t,0)=H(t,1)=0$
for $t\in[-T,0]$ and $H(-T,x)=H(0,x)=0$ for $x\in[0,1]$.  Given 
$u \in C([-T,0];M)$, let $L^\varepsilon_u : C^\infty_{0}([-T,0]\times
[0,1]) \to \bb R$ be the linear functional defined by
\begin{equation}
\label{e:2.5}
\begin{split}
L^\varepsilon_u(H) &:=\;
-\; \<\!\<  u , H_t\>\!\>
\; -\; \<\!\< f(u),H_x\>\!\>
\; - \; \varepsilon \, \<\!\< u, H_{xx}\>\!\> \\
& + \; {\varepsilon} \int_{-T}^0 \big[ \rho_1  \, H_x(t,1)  
\;-\; \rho_0 \, H_x(t,0)\big]  \, dt \;. 
\end{split}
\end{equation}
If $u$ is a smooth function satisfying the boundary conditions
$u(t,0)=\rho_0$ and $u(t,1)=\rho_1$ for $t\in[-T,0]$, then
$L^\varepsilon_u(H) = \<\!\< u_t + f(u)_x - \varepsilon u_{xx}
, H \>\!\>$.  Moreover, the functional $L^\varepsilon_u$
vanishes if and only if $u$ is a weak solution to \eqref{eq:1}.

Let $\sigma:[0,1] \to [0,+\infty)$ be the \emph{mobility} of the
system; we assume it is the function defined by $\sigma(a) =a(1-a)$.
Given $u\in C([-T,0],M)$, let
\begin{equation*}
\<\!\< u_x , u_x \>\!\> \; : = \; \sup_{H}
\big\{ -\, 2 \<\!\< u , H_x \>\!\> -  \<\!\< H , H \>\!\>\big\}\;,
\end{equation*}
where the supremum is carried over all smooth functions
$H:[-T,0]\times [0,1] \to \bb R$ such that $H(t,0) = H(t,1)=0$, 
$t\in [-T, 0]$. 
The action functional $I^\varepsilon_{{[-T,0]}}: C([-T,0];M) \to
[0,+\infty]$ is then defined by
\begin{equation}
\label{e:2.6}
I^\varepsilon_{[-T,0]}(u):= 
\begin{cases}
\displaystyle{
\sup_H \Big\{ L^\varepsilon_u(H) \;-\; 
\varepsilon \, \<\!\< H_x,\sigma(u) \, H_x \>\!\> \Big\}
}
& \textrm{if } \<\!\< u_x, u_x\>\!\><+\infty 
\\
+\infty &  \textrm{otherwise }
\end{cases}
\end{equation}
where the supremum is carried over all functions $H$ in $C^\infty_{0}
([-T,0]\times [0,1])$.  Clearly, $I^\varepsilon_{[-T,0]}(u)$ vanishes
if and only if $u \in C([-T,0];M)$ admits a square integrable
derivative and is a weak solution to \eqref{eq:1}.  We refer to
\cite{blm1} for equivalent definitions of the action functional
$I^\varepsilon_{[-T,0]}$. We remark that in \cite{blm1} the action
functional is defined with the condition $\<\!\< u_x,
u_x\>\!\>< +\infty $ replaced by the stronger condition
$\<\!\< u_x, \sigma(u)^{-1} u_x\>\!\> < +\infty$. The
argument in \cite[Lemma~4.9]{blm1} shows however that if the supremum
on the right hand side of \eqref{e:2.6} is finite these conditions are
in fact equivalent.

In order to state the connection of the action functional to the
perturbed parabolic problem \eqref{pbe} we need few more definitions.
Denote by $C^1([0,1])$ the space of continuously differentiable
functions $h: [0,1] \to \bb R$ and let $C^1_0([0,1]):=\{h\in
C^1([0,1]):\,h(0)=h(1)=0\}$.  Given a positive bounded measurable
function $\gamma:[0,1]\to [0,+\infty)$, let $\mc H^1_0(\gamma)$ be the
Sobolev space induced by $C^1_0([0,1])$ endowed with the inner product
\begin{equation*}
\<h,g\>_{1,\gamma} \;=\; \int_0^1 h_x \, g_x \, \gamma \, dx\;.
\end{equation*}
To be precise, the induced space $\mc H^1_0(\gamma)$ is obtained by
identifying and completing elements $h\in C^1_0([0,1])$ with respect
to the seminorm $\<h,h\>_{1,\gamma}^{1/2}$. When $\gamma=1$
the space $\mc H^1_0(\gamma)$ is the standard Sobolev space on
$[0,1]$, in this case we denote it simply by $\mc H_0^1$. 
Note that in the above equation, as well as below, we drop from the
notation the explicit dependence on the integration variable when
there is no ambiguity.

Denote by $| \cdot |_{1,\gamma}$ the norm of $\mc H^1_0(\gamma)$ and
let $\mc H^{-1}_0(\gamma)$ be the dual space of $\mc H^1_0(\gamma)$. 
It is equipped with the dual norm $| \cdot |_{-1,\gamma}$ defined by
\begin{equation*}
| \ell |^2_{-1,\gamma} \;=\;
\sup \big\{ 2 \< \ell, h\> \;-\; | h|^2_{1,\gamma} 
\:,\; h \in \mc H^1_0(\gamma) \big\}\;,
\end{equation*}
where $\< \ell, h\>$ stands also for the value at $h$ of the linear
functional $\ell$.

Fix a path $u$ in $C([-T,0], M)$ and denote by $\mf H^1_0(\sigma(u))$
the Hilbert space induced by $C^\infty_{0} ([-T,0]\times [0,1])$
endowed with the inner product
$\langle\!\<\cdot,\cdot\>\!\rangle_{1,\sigma(u)}$ defined by
\begin{equation*}
\langle\!\< H, G \>\!\rangle_{1,\sigma(u)}
\;=\; \int_{-T}^0 \langle H , G \rangle_{1,\sigma(u(t))}\, dt
\end{equation*}
and let $\Vert\cdot\Vert_{1,\sigma(u)}$ be the associated norm.
Let $\mf H^{-1}_0 (\sigma(u))$ be the dual of $\mf H^{1}_0
(\sigma(u))$; it is a Hilbert space equipped with the norm $\Vert \cdot
\Vert_{-1, \sigma(u)}$ defined by
\begin{equation*}
\|L\|^2_{-1,\sigma(u)} \;=\; \sup_{H}
\Big\{ 2 \<\!\<L, H \>\!\> \;-\;  
\Vert H \Vert^2_{1, \sigma(u)} \Big \}\;,
\end{equation*}
where the supremum is carried over all functions $H\in C^\infty_{0}
([-T,0]\times [0,1])$, equivalently over all $H\in \mf H^1_0(\sigma(u))$,
and $\<\!\<L, H \>\!\>$ stands for the value of the linear functional
$L$ at $H$.  The next statement is proven in \cite{blm1}.
A functional $f : X \to (-\infty,+\infty]$ defined on a Polish space
$X$ is said to be coercive if all its sub-level sets are precompact:
for all $t\in \bb R$, $\{ x : f(x) \le t\}$ is precompact.

\begin{theorem}
\label{s01}
For each $\epsilon>0$ and $T>0$ the functional
$I^\varepsilon_{[-T,0]}: C([-T,0];M) \to [0,+\infty]$ is coercive and
lower semicontinuous, namely it has compact sub-level sets.  Moreover,
given $u$ such that $I^\varepsilon_{[-T,0]}(u)<+\infty$, there exists
a unique $H=H(u)$ in $\mf H^{1}_0(\sigma (u))$ such that $u$ is a weak
solution to
\begin{equation*}
\begin{cases}
u_t + f(u)_x = \varepsilon u_{xx} - 2 \varepsilon (\sigma (u) H_x)_x
\\
u(t,0) = \rho_0 \,,\quad  u(t,1) = \rho_1\;.
\end{cases}
\end{equation*}
In this case, the linear functional $L^\varepsilon_u$, as 
defined in \eqref{e:2.5}, 
extends to a linear functional on $\mf H^1_0(\sigma(u))$, that we
denote by $u_t+f(u)_x-\varepsilon u_{xx}$, and
\begin{equation*}
I^\varepsilon_{[-T,0]}(u) \; =\; \varepsilon \,
\big \Vert H  \big\Vert^2_{1, \sigma (u)}
\;=\; \frac 1{4\varepsilon} \,
\big \Vert u_t + f(u)_x - \varepsilon u_{xx} \big\Vert^2_{-1, \sigma (u)} 
\;.
\end{equation*}
\end{theorem}

\subsection*{The quasi-potential}

The \emph{quasi-potential} \cite{FW} associated to the family 
of action functionals $I^\varepsilon_{[-T,0]}$, $T>0$,  is the functional 
$V_\varepsilon\,: M \to \bb [0,+\infty]$ defined by
\begin{equation}
\label{qp-1}
V_\varepsilon (\rho) \;:=\; \inf_{T>0} \; \inf \big\{ I^\varepsilon_{[-T,0]} (u) :
u  \in C([-T,0];M) \,,\, u(-T)=\upbar{\rho}_\varepsilon \, ,\, u(0) =\rho
\big\}\;, 
\end{equation}
so that $V_\varepsilon(\rho)$ measures the minimal cost to reach the
function $\rho$ starting from the stationary solution
$\upbar{\rho}_\varepsilon$. In this sense, while
$I^\varepsilon_{[-T,0]}(u)$ measures how much a path $u$ is close to
solutions to \eqref{eq:1}, the quasi-potential $V_\varepsilon (\rho)$
measures how much $\rho$ is close to the stationary solution
$\upbar{\rho}_\varepsilon$.

The previous definition implies that $V_\varepsilon$ is a
\emph{Lyapunov functional} for the Burgers equation \eqref{eq:1}. This
is to say that if $u(t;\rho)$, $t\ge 0$, is the solution to \eqref{eq:1}
with initial datum $u(0,\cdot;\rho)=\rho(\cdot)\in M$ then for any
$t\ge 0$ we have $V_\varepsilon \big(u(t;\rho)\big) \le V_\varepsilon
(\rho)$.  Observe that for each $t\ge 0$ we have $u(t;\rho)\in M$ by
the maximum principle.  The previous claim is easily proven, recalling
that $I^\varepsilon_{[-T,0]}$ vanishes on weak solutions to
\eqref{eq:1}, by exhibiting a test path for the variational problem
\eqref{qp-1}.

It is convenient to formulate the variational problem \eqref{qp-1}
also on paths defined on the semi-infinite time interval
$(-\infty,0]$. To this end we introduce the set
\begin{equation}
\label{Ue}
\mc U(\upbar{\rho}_\varepsilon) := \big\{ u\in C((-\infty,0];M)\,:\:
\lim_{t\to-\infty} u(t)= \upbar{\rho}_\varepsilon \big\}
\end{equation}
equipped with the topology of the uniform convergence. 
The family of action functionals $I^\varepsilon_{[-T,0]}$, $T>0$,
naturally induces the lower semicontinuous functional $I^\varepsilon :
\, \mc U (\upbar{\rho}_\varepsilon) \to [0,+\infty]$ defined by
\begin{equation}
\label{f03}
I^\varepsilon (u) := \lim_{T\to\infty} 
I_{[-T,0]}^\varepsilon \big( u\! \restriction_{[-T,0]} \big)\;,  
\end{equation}
where $u\! \restriction_{[-T,0]}$ denotes the restriction of $u\in \mc
U(\upbar{\rho}_\varepsilon)$ to $C([-T,0];M)$. Note that the above
limit always exists, possibly equal to $+\infty$, in view of the
monotonicity of $I^\varepsilon_{[-T,0]}(u\! \restriction_{[-T,0]})$ in
$T>0$.
Let
\begin{equation}
\label{qp}
\hat V_\varepsilon (\rho) \;:=\; \inf \big\{ I^\varepsilon (u) :
u  \in \mc U (\upbar{\rho}_\varepsilon)  \,,\,  u(0) =\rho
\big\}
\end{equation}
and observe that the inequality $\hat V_\varepsilon\leq V_\varepsilon$
holds trivially.  In the context of diffusion processes in $\bb R^n$,
by the continuity of $\hat V_\varepsilon$, it is easy to show that
$V_\varepsilon = \hat V_\varepsilon$ \cite{FW}.  We shall prove that
this identity also holds in the present setting where, as shown below,
$\hat V_\varepsilon$ is not continuous but only lower semicontinuous.

\subsection*{Characterization of the quasi-potential}

The analysis of the quasi-potential in the case $\rho_0> \rho_1$ and
in the case $\rho_0 < \rho_1$, $\varepsilon \ge \varepsilon_0$ has
been considered in \cite{bgl1}.  Here we discuss the more interesting
case in which $\rho_0< \rho_1$ and $\varepsilon\in (0,\varepsilon_0)$.
The first main result of this article states that the quasi-potential
$V_\varepsilon$ can be expressed in terms of a static variational
problem.

Given $m>0$, denote by $\mc P_{m}([0,1])$ the set of positive Borel
measures on $[0,1]$ with total mass equal to $m$. Recall that
$\varphi_i = \log\big[\rho_i/(1-\rho_i)\big]$, $i=0,1$,
$\varphi_0<\varphi_1$, and set
\begin{equation}
\label{Fa}
\mc F := \big\{ \varphi \,:\: \varphi(x) = \varphi_0 + \mu ([0,x]) 
\textrm{ for some } \mu \in \mc P_{\varphi_1-\varphi_0}([0,1])
\big\}\;.
\end{equation}
Clearly, if $\varphi$ belongs to $\mc F$ then $\varphi$ is an
increasing c\`adl\`ag function satisfying $\varphi_0 \le \varphi(0)$,
$\varphi(1) = \varphi_1$.  We consider the set $\mc F$ equipped with
the topology inherited from the weak convergence of measures, namely 
a sequence $\{\varphi^n\} \subset \mc F$ converges to
$\varphi$ in $\mc F$ if and only if for any continuous function $g$ on
$[0,1]$ we have $\int \! g \, d\varphi^n \to 
\int\! g \, d\varphi $. Then $\mc F$ is a compact Polish space.
Moreover, $\varphi^n\to \varphi$ in $\mc F$ implies $\varphi^n\to
\varphi$ a.e.

Let $s:\bb R \to (-\infty,+\infty]$ be the convex function defined
by 
\begin{equation}
  \label{s=}
  s(a) :=
  \begin{cases}
    a\log a +(1-a)\log(1-a) & \textrm{ if $a\in[0,1]$ }  \\
    +\infty                 & \textrm{ otherwise}  
  \end{cases}
\end{equation}
and observe that for $a\in (0,1)$ we have $s''(a) \, \sigma(a) =1$.
Let $\mc G_\varepsilon : M \times \mc F \to (-\infty,+\infty] $ be the
functional defined by
\begin{equation}
\label{Grf}
\mc G_\varepsilon  (\rho,\varphi) =
\int_{0}^1 \Big[ s(\rho) + s(\varepsilon \varphi_x)  
+(1-\rho)\varphi -\log\big(1+e^{\varphi}\big)  \Big] \, dx\;,
\end{equation}
where we understand that $\mc G_\varepsilon (\rho,\varphi) = +\infty$ unless
the measure $d\phi$ is absolutely continuous with respect to the Lebesgue
measure and its density, denoted by $\varphi_x$, is bounded by
$\epsilon^{-1}$. By the convexity of $s$, the functional $\mc
G_\epsilon$ is lower semicontinuous.

We shall connect the quasi-potential $V_\epsilon(\rho)$ to the minimum
of the functional $\mc G_\epsilon(\rho,\cdot)$ over $\mc F$.  Fix
$\rho$ in $M$ and consider the Euler-Lagrange equation associated to
the functional $\mc G_\varepsilon (\rho,\cdot)$
\begin{equation}
\label{Deq}
\left\{
\begin{aligned}
& \frac{\varepsilon \varphi_{xx}}{\varphi_x (1-\varepsilon\varphi_x)}
\;-\; \frac 1{1+ e^{\varphi}}  + \rho \;=\; 0 \;, \\
& \varphi(0) = \varphi_0\;, \quad \varphi(1) = \varphi_1\; . 
\end{aligned}
\right.
\end{equation}
Since this equation is not really meaningful for $\varphi\in\mc F$, we
formulate it as a fixed point condition for a suitable operator.
Denote by $K_{\rho,\varepsilon} = K(\varepsilon, \varphi_0, \varphi_1,
\rho) \, : \mc F \to \mc F$ the integral operator
\begin{equation}
\label{g03}
(K_{\rho,\varepsilon} \, \varphi)(x) \;=\; 
\varphi_0 \;+\; \frac 1\varepsilon  \int_0^x \frac{ { 
A \, \exp\Big\{ \frac 1\varepsilon \int_0^y
\big[\big( 1 + e^{\varphi} \big)^{-1} - \rho \big] \, dz \Big\} }}
{ 1 +  A \, \exp\Big\{ \frac 1\varepsilon \int_0^y
\big[  \big( 1 + e^{\varphi} \big)^{-1} - \rho \big] \, dz
\Big\} }\; dy \;,
\end{equation}
where $A= A(\epsilon,\rho,\varphi)\in(0,\infty)$ is chosen so that
$(K_{\rho,\varepsilon} \varphi)(1) = \varphi_1$. We prove in
Section~\ref{sec3} that such a choice is always possible and unique. 
Moreover, if $\varphi$ is a fixed point of $K_{\rho,\varepsilon}$, then 
$\varphi$ is in $C^1([0,1])$, it has a Lipschitz derivative
$\varphi_x$ satisfying $0 < \epsilon \varphi_x < 1$, and $\varphi$
solves \eqref{Deq} a.e.
We adopt, in particular, the following terminology.  A function
$\varphi\in\mc F$ is said to be a solution of \eqref{Deq} if it
belongs to $\ms P_\varepsilon(\rho)$, the set of fixed points of
$K_{\rho,\varepsilon}$.

Recall that $\upbar{\rho}_\varepsilon$ is the stationary solution to
\eqref{eq:1}. Let $S^o_\varepsilon$, $S_\varepsilon : M \to \bb R$ be
the functionals defined by
\begin{equation}
\label{ssg}
S^o_\varepsilon (\rho) \; :=\; \inf
\:\big\{ \mc G_\varepsilon(\rho,\varphi)\,,\: \varphi\in \mc F\big\} \;,\quad
S_\varepsilon (\rho) \;=\; S^o_\varepsilon (\rho) \;-\;
S^o_\varepsilon (\upbar{\rho}_\varepsilon) 
\end{equation}
and note that the infimum is achieved by the lower semicontinuity of
$\mc G_\epsilon(\rho,\cdot)$ and by the compactness of $\mc F$.  Given
$\rho\in M$, we denote by ${\ms F}_\epsilon(\rho)\subset \mc F$ the
collection of minimizers for the previous variational problem, i.e.\
\begin{equation}
  \label{Fer}
  {\ms F}_\epsilon(\rho) :=  
  \mathrm{arg \,inf} \: 
  \big\{ \mc G_\epsilon (\rho,\varphi)\,,\: \varphi\in \mc F\big\}
\end{equation}
and observe that ${\ms F}_\epsilon(\rho)$ is a non-empty compact subset
of $\mc F$.

\begin{theorem}
\label{t:1-1}
Fix $\varphi_0<\varphi_1$, $\varepsilon\in (0,\varepsilon_0)$, and
$\rho \in M$. Then any minimizer of $ \mc G_\epsilon (\rho,\cdot)$
solves \eqref{Deq},
i.e.\
\begin{equation*}
  {\ms F}_\epsilon(\rho) \; \subset \; \ms P_\varepsilon(\rho)\;.
\end{equation*}
Moreover, the functional $S_\varepsilon$ is lower semicontinuous on
$M$. Finally, if the sequence $\{\rho^n\}\subset M$ converges to
$\rho$ \emph{strongly} in $L^1([0,1])$ then $S_\varepsilon(\rho^n)\to
S_\epsilon(\rho)$.
\end{theorem}

The connection between the quasi-potential and the functional
\eqref{ssg} is established by the following theorem, which is the main
result of this paper.

\begin{theorem}
\label{t:S=V}
For each $\rho_0<\rho_1$ and $\varepsilon \in (0,\varepsilon_0)$
we have $V_\varepsilon=\hat V_\epsilon = S_\varepsilon$.
\end{theorem}

We remark that while the identity $V_\varepsilon=\hat V_\epsilon$
holds under general conditions, the characterization of the
quasi-potential in terms of the static variational problem \eqref{ssg}
depends crucially on the specific form of the flux $f$ and the
mobility $\sigma$ namely, $f(a)=\sigma(a)=a(1-a)$, $a\in [0,1]$. 

In the proof of Theorem \ref{t:S=V} we actually describe some optimal
paths for the variational problem \eqref{qp}.  Fix $\rho$ in $M$, let
$\varphi\in{\ms F}_\epsilon(\rho)$, and denote by $F=F(t,x)$,
$(t,x)\in [0,+\infty)\times [0,1]$ the solution of the viscous Burgers
equation \eqref{eq:1} with initial condition $e^{\varphi}/
(1+e^{\varphi})$.  Set $\psi= s'(F)$ and define $v=v(t,x)$ by
\begin{equation*}
v \;=\; \frac 1{1+ e^{\psi}}  
\;-\; \frac{\varepsilon \psi_{xx}}
{\psi_x (1-\varepsilon\psi_x)} \;\cdot
\end{equation*}
We prove that an optimal path $u$ for the variational problem
\eqref{qp} is the path $v$ reversed in time, i.e.\ $u(t)=v(-t)$.  This
construction shows that to each $\varphi\in {\ms F}_\epsilon(\rho)$
there is associated a path $u\in \mc U(\upbar{\rho}_\epsilon)$ which
is a minimizer for the variational problem \eqref{qp}. If ${\ms
  F}_\epsilon(\rho)$ is not a singleton, to different elements in
${\ms F}_\epsilon(\rho)$ are associated different minimizers for
\eqref{qp} and there is no uniqueness of the minimizer for \eqref{qp}.
Unfortunately, we are not able to prove that any minimizer of
\eqref{qp} can be obtained from the previous construction; in
particular we cannot deduce uniqueness of the minimizer for \eqref{qp}
from the uniqueness of the minimizer for $\mc G_\epsilon(\rho,\cdot)$.
We refer however to the heuristic argument presented below, in the
context of the Hamiltonian formalism, which suggests that the minimizers
for \eqref{qp} are indeed in a one-to-one correspondence with ${\ms
  F}_\epsilon(\rho)$.

Theorem~\ref{t:S=V} implies that the minimum of $S_\varepsilon^o$ is 
achieved at $\upbar{\rho}_\epsilon$, equivalently that $S_\epsilon$ is 
a positive functional. 
This can also be shown by direct computations.  Indeed it is enough
to observe that, for a fixed $\varphi\in\mc F$, the strict convexity
of the map $\rho\mapsto \mc G_\varepsilon(\rho,\varphi)$ implies that
the infimum over $\rho$ is uniquely attained for $\varphi=s'(\rho)$.
A straightforward computation then shows that the functional $\varphi
\mapsto \mc G_\varepsilon({s'}^{-1}(\varphi),\varphi)$ has a unique
critical point, which is a global minimum, achieved at
$s'(\upbar{\rho}_\varepsilon)$.

\subsection*{Uniqueness / non uniqueness of optimal paths}

The connection between the quasi-potential $V_\epsilon$ and a
``trial'' functional like $\mc G_\epsilon$ has been established for
other action functionals arising as large deviation rate functional
for few microscopic stochastic dynamics in the diffusive scaling limit
\cite{bdgjl3,bgl1,d1,de}.  In contrast with all other cases, the
functional $\mc G_\varepsilon ( \rho, \cdot)$ is neither concave nor
convex and might have more than a single critical point.  Fix $\rho\in
M$.  The construction presented above actually shows that to each
critical point $\varphi$ for $\mc G_\epsilon(\rho,\cdot)$ there is
associated a path $u\in \mc U(\upbar{\rho}_\epsilon)$ which is a
critical point for the variational problem \eqref{qp}.  It is
therefore natural to investigate whether the sets $\ms
P_\epsilon(\rho)$ and ${\ms F}_\epsilon(\rho)$ are singletons.

In this direction, our first result shows that - under suitable
conditions - there exists a unique critical point for $\mc
G_\epsilon(\rho,\cdot)$.  

\begin{theorem}
\label{t:unmine}The following statements hold.
\begin{itemize}
\item[(i)] Fix $\varphi_0<\varphi_1$. There exists
  $\varepsilon_1\in(0,\varepsilon_0)$ such that for any
  $\varepsilon\in (\varepsilon_1,\varepsilon_0)$ the set $\ms
  P_\epsilon(\rho)$ is a singleton for any $\rho\in M$.
\item[(ii)] Fix $\epsilon>0$. There exists $\delta\in
  (0,\epsilon^{-1})$ such that for any $0<\varphi_1-\varphi_0 <\delta$
  the set $\ms P_\epsilon(\rho)$ is a singleton for any $\rho\in M$.
\item[(iii)] Fix $\varphi_0<\varphi_1$ and $\epsilon\in
  (0,\epsilon_0)$. If $\rho\in M$ is in $C^1([0,1])$ and strictly
  increasing then the set $\ms P_\epsilon(\rho)$ is a singleton.
\end{itemize}
\end{theorem}

We remark that while the first two results are based on a standard
perturbation argument and are quite natural from a statistical
mechanics point of view, the third one somehow depends on the global
structure of the functional $\mc G_\epsilon$.  Recalling that the
stationary solution $\upbar{\rho}_\epsilon$ is smooth and strictly
increasing, the third statement implies that the functional $\mc
G_\varepsilon (\rho,\cdot)$ has a unique critical point when $\rho$
lies in a $C^1$-neighborhood of $\upbar{\rho}_\epsilon$.

As discussed above, a most striking feature of the model here examined
is that there can be more than a single minimizer for $\mc
G_\varepsilon(\rho,\cdot)$.  The next result states that this
phenomenon does indeed occur.

\begin{theorem}
\label{t:degmine}
Fix $\varphi_0<\varphi_1$. There exists $\varepsilon_2\in
(0,\epsilon_0)$ such that the following statement holds.  For each
$\epsilon\in (0,\epsilon_2)$ there exist functions $\rho\in M$ such
that ${\ms F}_\epsilon(\rho)$ is not a singleton.
\end{theorem}

As we show in Proposition \ref{t:DL*}, at the points $\rho$ where $\ms
F_\epsilon(\rho)$ is not a singleton the quasi-potential admits more
than one G\^ateaux super-differential.  The proof of the above theorem
is based on a perturbation argument with respect to the limiting case
$\epsilon=0$ that we next discuss.

\subsection*{The inviscid limit}

It is well known, see e.g.\ \cite[Ch.~15]{Se}, that in the inviscid
limit $\varepsilon\downarrow 0$, the solution to the Cauchy problem
associated to \eqref{eq:1} converges to the entropy solution of the
Cauchy problem associated to the inviscid Burgers equation $u_t +
f(u)_x = 0$ with the Bardos-Leroux-N\'ed\'elec boundary conditions
\cite{bln}.  We also mention that in the case of the Burgers equation
on $\bb R$, the variational convergence as $\epsilon\downarrow 0$ of
the action functional $I^\epsilon_{[-T,0]}$ for a fixed $T>0$ is
discussed in \cite{BBMN}.  Referring to \cite{bgl1} for the case
$\rho_0>\rho_1$, we here discuss the variational convergence of the
quasi-potential $V_\epsilon$ as $\varepsilon\downarrow 0$.

In the inviscid limit $\varepsilon=0$, the stationary solutions
$\upbar\rho$ are easily described by considering the propagation of
shocks for the inviscid Burgers equation on $\bb R$.  If
$1-(\rho_1+\rho_0)>0$ an entropic shock from $\rho_0$ to $\rho_1$
travels to the right so that $\upbar\rho =\rho_0$, while $\upbar\rho
=\rho_1$ if $1-(\rho_0+\rho_1)<0$.  In the case $1-(\rho_0+\rho_1)=0$
there is a one parameter family of stationary entropic solutions which
corresponds to a stationary shock that can be placed anywhere in
$[0,1]$.  Equivalently, a stationary entropic solution $\upbar\rho$ satisfies
$f(\upbar\rho)=\min_{r\in[\rho_0,\rho_1]} f(r)$.  
For $1-(\rho_0+\rho_1)\neq 0$ it is not difficult to check that, as
$\varepsilon \downarrow 0$, the stationary solution
$\upbar\rho_\varepsilon$ converges strongly in $L^1([0,1])$ to the
constant function equal to $\upbar\rho$.  In the case
$1-(\rho_0+\rho_1)=0$, $\upbar\rho_\varepsilon$ converges to the
stationary solution of the inviscid Burgers equation with a shock
placed at $x=1/2$.

Recall \eqref{Grf} and let $\mc G : M \times \mc F \to \bb R$ be the
lower semicontinuous functional defined by
\begin{equation}
\label{Grf-in}
\mc G  (\rho,\varphi)  \;:=\; \int_{0}^1
\Big[ s(\rho) +(1-\rho)\varphi  -\log\big(1+e^{\varphi}\big)
\Big]\,dx\;. 
\end{equation}
Let also $S^o, S : M \to \bb R$ be the functionals defined by 
\begin{equation}
\label{s-in}
\qquad S^o (\rho) \;:=\; 
\inf \: \big\{\mc G(\rho,\varphi)\,,\:  \varphi \in \mc F\big\} 
\;, \qquad S (\rho) \;:=\;S^o (\rho) \;-\; S^o (\upbar{\rho})
\end{equation}
and observe the infimum is achieved by the lower semicontinuity of
$\mc G(\rho,\cdot)$ and the compactness of $\mc F$.  As discussed
before, for $\rho_0+\rho_1\neq 1$ there exists a unique stationary
entropic solution $\upbar{\rho}$ of the inviscid Burgers equation. On
the other hand, if $\rho_0+\rho_1=1$ there exists a one-parameter
family of stationary entropic solutions
$\{\upbar{\rho}_\alpha,\,\alpha\in [0,1]\}$; it is however simple to
check that $S^o(\upbar{\rho}_\alpha)$ is in fact independent of
$\alpha$. This shows the functional $S$ is well defined.

As we show in Proposition~\ref{t:gin=tgin} below, we may restrict the
infimum in \eqref{s-in} to functions $\varphi\in \mc F$ which are step
functions in the sense that $\varphi$ jumps from $\varphi_0$ to
$\varphi_1$ at a single point in $[0,1]$.  In view of this result,
simple computations show that $S$ coincides with the functional
derived in \cite{dls3} within the context of the boundary driven
asymmetric exclusion process.

In the inviscid limit $\epsilon\downarrow 0$ we expect the functional
$S_\epsilon$ to converge to $S$. From a variational point of view, the
natural notion of convergence is the so-called $\Gamma$-convergence
that we next recall, see e.g.\ \cite{Braides}.  Let $X$ be a Polish
space.  A sequence of functionals $F_n : X \to (-\infty,+\infty]$ is
said to \emph{$\Gamma$-converge} to a functional $F : X \to
(\infty,+\infty]$ if the following two conditions hold for each $x\in
X$. There exists a sequence $x_n\to x$ such that $\limsup_n F_n(x_n)
\le F(x)$ (\emph{$\Glimsup$ inequality}) and for any sequence $x_n\to
x$ we have $\liminf_n F_n(x_n) \ge F(x)$ (\emph{$\Gliminf$
  inequality}).

\begin{theorem}
\label{t:gconv}
Let $S_\varepsilon, S \,: M \to [0,\infty)$ be defined as in \eqref{ssg},
\eqref{s-in}, respectively. The family of functionals
$\{S_\varepsilon\}_{\varepsilon >0}$ $\Gamma$-converges to $S$ in $M$
as $\varepsilon\downarrow 0$. 
In particular, the functional $S: M \to [0,+\infty)$ is lower
semicontinuous.
\end{theorem}

Given $\rho\in M$, we also expect that the minimizers of $\mc
G_\epsilon(\rho,\cdot)$ converge, as $\epsilon\downarrow 0$, to a
minimizer of $\mc G(\rho,\cdot)$. The precise statement is the
following.

\begin{theorem}
\label{t:convmin}
Fix $\rho\in M$ and let $\varphi_\epsilon \in {\ms F}_\epsilon(\rho)$. 
If $\epsilon_n\downarrow 0$ is a sequence such that
$\varphi_{\epsilon_n} \to \varphi$ for some $\varphi \in \mc F$, then 
$\varphi$ is a minimizer for $\mc G (\rho,\cdot)$. 
\end{theorem}

\medskip
\noindent\emph{Notation warning.} 
Apart when we discuss, in Section~\ref{sec6}, the inviscid limit, the
parameter $\epsilon>0$ is kept fixed. We therefore drop from most of
the notation the explicit dependence of the functionals on $\epsilon$.

\section{Hamiltonian picture}
\label{s:hp}

As it will be clear in the following discussion, the variational
problem \eqref{qp} is naturally described within the Hamiltonian
formalism.  Accordingly, Theorem \ref{t:S=V} reflects a peculiar
geometric structure of the underlying phase space.  In this section we
present heuristically this picture, but we emphasize that the actual
proofs are logically independent from it.

The functional $I_{[-T,0]}$ in \eqref{e:2.6} 
is the action functional corresponding to the Lagrangian 
\begin{equation*}
\bb L (u,u_t) \;=\; \frac1{4\varepsilon} \, \big | u_t 
+ f(u)_x - \varepsilon u_{xx} \big |^2_{-1, \sigma (u)}\;.
\end{equation*}
The associated Hamiltonian, obtained by a Legendre transform, is given
by
\begin{equation}
\label{f15}
\bb H (\rho,h)  \; =\; 
\varepsilon \, \langle h_x, \sigma (\rho) h_x \rangle 
\; + \; 
\langle \varepsilon \rho_{xx} - f(\rho)_x , h \rangle 
\;,
\end{equation}
where $\rho:[0,1]\to [0,1]$ satisfies $\rho(0)=\rho_0$,
$\rho(1)=\rho_1$, and $h:[0,1]\to \bb R$ is the momentum, satisfying
the boundary conditions $h(0)=h(1)=0$.  The canonical equations
associated to the Hamiltonian $\bb H$ are
\begin{equation}
\label{f12}
\left\{
\begin{aligned}
u_t & \;=\; \varepsilon \, u_{xx} \;-\;  f(u)_x \;-\; 
2 \varepsilon \, (\sigma(u) H_x)_x  \\
H_t & \;=\; -  \varepsilon \, H_{xx}  \;-\; f'(u)  H_x  
\;-\;  \varepsilon \, \sigma'(u) \, H_x^2
\end{aligned}
\right.
\end{equation}
with the boundary conditions $u(t,0) =\rho_0$, $u(t,1) =\rho_1$ and
$H(t,0) = H(t,1) = 0$.  It is not clear whether the above equations do
define, even locally, a flow and the discussion will be here kept at
the informal level.

As follows from the exponential attractiveness of
$\upbar{\rho}_\varepsilon$ for the flow defined by \eqref{eq:1},
$(\upbar{\rho}_\varepsilon, 0)$ is a hyperbolic fixed point of the
Hamilton flow \eqref{f12}. Denote by $\ms M_\mathrm{s}$, $\ms
M_\mathrm{u}$ the associated stable and unstable manifolds.  Of
course, $\ms M_\mathrm{s},\,\ms M_u\subset\{(\rho,h):\,\bb H(\rho,h)
=\bb H(\upbar{\rho}_\varepsilon,0)=0\}$.  Each point $(\rho, 0)$ is
driven by the flow to $(\upbar{\rho} _\varepsilon, 0)$ as
$t\to+\infty$, and therefore $\ms M_\mathrm{s}\supset \{(\rho,h)\,:\,h=0\}$.

Recall the Poincar\'e-Cartan theorem, see e.g.\ \cite[\S~44]{arn},
which states that the integral of the symplectic one-form $\langle
h,d\rho\rangle$ along any closed path in the phase space is preserved
by the Hamiltonian flow.  This implies that the stable and unstable
manifolds are \emph{Lagrangian}, namely for any closed path 
which lies either in $\ms M_{s}$ or in $\ms M_\mathrm{u}$,
\begin{equation*}
\oint \< H,du\> \;=\; 0\;.
\end{equation*}
We can therefore define the \emph{pre-potential} $W_\varepsilon:\,\ms
M_\mathrm{u}\to\bb R$ by
\begin{equation}
\label{f23}
W_\varepsilon(\rho, h) \;=\; \int_\Gamma \< H,du\> \;,
\end{equation}
where the integral is carried over a path $\Gamma$ in $\ms
M_\mathrm{u}$ which connects $(\upbar{\rho}_\varepsilon, 0)$ to
$(\rho, h)$.  Recalling \eqref{qp}, the connection between the
quasi-potential and the pre-potential is given by
\begin{equation}
\label{f13}
\hat{V}_\varepsilon(\rho) \;=\; \inf \, 
\big\{ W_\varepsilon(\rho, h)\,,\: h\,: \: (\rho,h) \in \ms M_\mathrm{u} \big\}
\end{equation}
In a finite dimensional framework, this result is proven in
\cite{D,DD}.  For the reader's convenience, we sketch the basic
argument.

Denote by $U_\varepsilon(\rho)$ the right hand side of \eqref{f13}.
By means of compactness arguments, one shows the existence of a path
$u \in \mc U (\upbar{\rho}_\epsilon)$ satisfying $u(0) =\rho$ and such
that
\begin{equation*}
\hat{V}_\varepsilon(\rho)= I (u) 
\;=\; \int_{-\infty}^0 \bb L (u \,,\, u_t ) \,dt \;.
\end{equation*}
Since $u$ is an extremal path, it satisfies the Euler-Lagrange
equation, or, equivalently, the pair $(u,H)$, where $H = \delta \bb
L /\delta u_t$ stands for the conjugate momentum, solves
the canonical equations \eqref{f12}.  One then shows that the
trajectory $(u,H)$ lies in the unstable manifold $\ms M_\mathrm{u}$;
while $u(t)\to\upbar{\rho}_\varepsilon$ as $t\to -\infty$ follows from
the definition, some efforts are required to show that also $H(t)\to
0$ as $t\to -\infty$.  By Legendre duality, the inclusion $\ms
M_\mathrm{u}\subset\{\bb H =0\}$, and \eqref{f23}
\begin{equation*}
\begin{split}
& \hat V_\varepsilon(\rho) \;=\;
I (u) \; =\; \int_{-\infty}^0 
\bb L (u \,,\, u_t ) \,dt 
\; =\; 
\int_{-\infty}^0 \big[ \< H,u_t\>  
\;-\; \bb H (u \,,\, H ) \big] \,dt \\
& \phantom{ \hat V_\varepsilon(\rho)}
\; =\;  \int_{-\infty}^0 \< H,u_t\>  \,dt \; 
=\; W_\varepsilon(\rho,H(0)) 
\;\ge\; U_\varepsilon (\rho)\;.
\end{split}
\end{equation*}
The proof of the reverse inequality is simple. Fix $\rho$ and choose
$h$ which minimizes the right hand side of \eqref{f13}. Since
$(\rho,h)$ belongs to $\ms M_\mathrm{u}$, we need only to follow the
Hamiltonian flow \eqref{f12} to obtain a path $(u,H)$ such that $u(0)
= \rho$, $H(0)=h$ and $(u(t), H(t))\to (\upbar{\rho}_\varepsilon,0)$
as $t\to-\infty$.  The previous computations now give
\begin{equation*}
\hat V_\varepsilon(\rho)\leq I (u) 
\;=\; W_\varepsilon(\rho,h) \;=\; U_\varepsilon(\rho)\,. 
\end{equation*}
The above argument actually shows that any minimizer $u$ for the
variational problem \eqref{qp} is a solution to the canonical
equations \eqref{f12} with $u(0)=\rho$ and $H(0)=h$ where $h$ is a
minimizer for the right hand side of \eqref{f13}.

In a neighborhood of the fixed point $(\upbar{\rho}_\varepsilon,0)$,
the unstable manifold $\ms M_\mathrm{u}$ can be written as a graph,
namely it has the form $\ms M_\mathrm{u} = \{(\rho,h)\, :
\,h=m_\mathrm{u}(\rho)\}$ for some map $m_\mathrm{u}$.  In this
case, the infimum on the right hand side of \eqref{f13} is trivial and
$\hat V_\varepsilon(\rho)= W_\varepsilon(\rho,m_\mathrm{u}(\rho))$.
In general, though, this is not true globally. Given $\rho\in M$, to
each $h$ such that $(\rho,h)\in\ms M_\mathrm{u}$, there corresponds a
critical point for the variational problem \eqref{qp}.  It may happen,
for special $\rho$, that the variational problem on the right hand
side of \eqref{f13} admits more than a single minimizer.  In this case
there is also more than one minimizer for the variational problem
\eqref{qp}.  In particular, as will be clearer in the following,
Theorem~\ref{t:unmine} implies that, for either $\varepsilon$ close to
$\varepsilon_0$ or $\varphi_1-\varphi_0$ small, the manifold $\ms
M_\mathrm{u}$ is globally a graph.  On the other hand, when 
$\ms P (\rho)$ is not a singleton for some $\rho\in M$ the manifold
$\ms M_\mathrm{u}$ is not globally a graph. Finally,
Theorem~\ref{t:degmine} implies that for $\varepsilon$ small enough
there exist functions $\rho$ such that the minimizer for the right hand side
of \eqref{f13} is not unique.

In view of \eqref{f13}, to prove heuristically Theorem \ref{t:S=V} we
need to replace $W_\epsilon$ by $\mc G_\epsilon$ on the right hand
side of \eqref{f13}. It is convenient to perform the symplectic change
of variables $(\rho,h) \mapsto (\varphi,\pi)$ given by
\begin{equation}
\label{scv}
\begin{cases}
\varphi = s'(\rho) -h\;, \\
\pi = \rho\;.
\end{cases}
\end{equation}
In the new variables $(\varphi,\pi)$ the Hamiltonian reads 
\begin{equation}
\label{tHam}
{\widetilde{\bb H}} (\varphi,\pi)
\;=\;  \bb H (\pi, s'(\pi) -\varphi) 
\;=\; \epsilon \langle \varphi_x, \sigma(\pi)\varphi_x\rangle 
\;-\; \langle \epsilon \pi_{x} + \sigma(\pi) ,\varphi_x  \rangle
\;+\; \rho_1 \,-\, \rho_0\;,
\end{equation}
where we used that $s'(\pi) - \varphi$ vanishes at the boundary,
$\sigma(\pi) s''(\pi) =1$, $f=\sigma$ and $\rho(0)=\rho_0$,
$\rho(1)=\rho_1$.  The corresponding canonical equations are
\begin{equation}
\label{nhf}
\left\{
\begin{aligned}
\Phi_t & \;=\; \frac {\delta {\widetilde{\bb H}}}{\delta\Pi}
\;=\;
\varepsilon \, \Phi_{xx}  \;-\; \sigma'(\Pi)  \,
\Phi_x  \, (1-\varepsilon\Phi_x)\;,
\\
\Pi_t & \;=\; - \frac {\delta {\widetilde{\bb H}}}{\delta\Phi}
\;=\;
- \varepsilon \, \Pi_{xx} \;-\;  \sigma (\Pi)_x \;+\; 
2 \varepsilon \, \big( \sigma(\Pi) \Phi_x \big)_x \;.
\end{aligned}
\right.
\end{equation}

In the new variables the fixed point $(\upbar{\rho}_\epsilon,0)$ reads
$(s'(\upbar{\rho}_\epsilon),\upbar{\rho}_\epsilon)$. The associated
stable manifold is $\{ (\varphi,\pi):\, \varphi=s'(\pi)\}$.  Let
\begin{equation}  
\label{Sig}
\Sigma = \Big\{ (\varphi,\pi) \, :\: 
\pi = \frac{1}{1+e^\varphi} - \frac{\epsilon \varphi_{xx}}
{\varphi_x(1-\epsilon\varphi_x)} \Big\}\;.
\end{equation}
By using that $\sigma(\pi)=\pi(1-\pi)$, 
a long and tedious computation that we omit shows that the set $\Sigma$
is invariant under the flow \eqref{nhf}. 
More precisely, pick a point $(\varphi,\pi)\in \Sigma$ and let $\Phi$ be
the solution to  
\begin{equation*}
  \Phi_t = -\epsilon \Phi_{xx} + \frac{1-e^\Phi}{1+e^\Phi}
  \Phi_x(1-\epsilon\Phi_x) 
\end{equation*}
with initial condition $\Phi(0)=\varphi$. Set now 
\begin{equation*}
  \Pi = \frac{1}{1+e^\Phi} 
  - \frac{\epsilon \Phi_{xx}}{\Phi_x(1-\epsilon\Phi_x)}
\end{equation*}
and observe that $\Pi(0)=\pi$ since $(\varphi,\pi)\in \Sigma$.  Then
$(\Phi,\Pi)$ is a solution to the canonical equations \eqref{nhf}.
Moreover, as we show in Lemmata \ref{s10} and \ref{s11}, under the
Hamiltonian flow any point in $\Sigma$ converges to the fixed point
$(s'(\upbar{\rho}_\epsilon),\upbar{\rho}_\epsilon))$ as $t\to
-\infty$.  This implies $\Sigma$ is the unstable manifold.  The
previous arguments really only show that $\{(\varphi,\pi):\,
\varphi=s'(\pi)\} \subset \ms M_\mathrm{s}$ and $\Sigma\subset\ms
M_\mathrm{u}$. On the other hand, as the tangent spaces to
$\{(\varphi,\pi):\,\varphi=s'(\pi)\}$ and $\Sigma$ at
$(s'(\upbar{\rho}_\epsilon),\upbar{\rho}_\epsilon)$ span the whole
space, we informally claim that the previous inclusions are
equalities.

At this point, the informal deduction of Theorem~\ref{t:S=V} will be
completed by the computation of the pre-potential. This is easily achieved
in the new variables $(\varphi,\pi)$. 
We start by the generating function of the symplectic transformation
\eqref{scv}. Let 
\begin{equation*}
  F(\rho,\varphi) = \int_0^1 \big[ s(\rho) - \rho\, \varphi \big] \, dx
\end{equation*}
be the so-called \emph{free generating function} of \eqref{scv} (see
e.g.\ \cite[\S~48]{arn}), so that
\begin{equation*}
h = \frac {\delta F}{\delta\rho} 
\,,\qquad \pi = - \frac {\delta F}{\delta\varphi} \;\cdot
\end{equation*}
Equivalently,
\begin{equation*}
\langle h , d\rho\rangle - \langle\pi, d\varphi \rangle= d F\;.
\end{equation*}
Hence, for any path $\Gamma=\{\gamma(t), \, t\in[0,1]\}$ in the phase
space 
\begin{equation*}
\int_\Gamma \langle H, d\Pi \rangle 
=  \int_\Gamma  \langle\Pi, d\Phi\rangle 
+ F(\gamma(1)) - F(\gamma(0)) \;.
\end{equation*}
Assume now that $\Gamma\subset \Sigma$. By \eqref{Sig} and since
\begin{equation*}
\frac{\epsilon \varphi_{xx}}{\varphi_x(1-\epsilon\varphi_x) }  
= \epsilon\, \big[s'(\epsilon\varphi_x)\big]_x\; ,
\end{equation*}
we have that
\begin{equation*}
\int_\Gamma  \langle\Pi, d\Phi\rangle = 
\int_0^1\big[ \Phi(t) - \log\big(1+e^{\Phi(t)}\big) 
+ s\big(\epsilon\Phi_x(t)\big)\big]\,dx
\Big|_{t=0}^{t=1}\;.
\end{equation*}
Therefore, in view of \eqref{Grf}, the previous identities  imply that 
\begin{equation*}
\int_\Gamma \langle H, du \rangle 
\;=\; \int_\Gamma \langle H, d\Pi \rangle 
\;=\; \mc G_\epsilon \big(\Pi (1),\Phi(1)\big) 
\;-\; \mc G_\epsilon \big(\Pi (0),\Phi(0)\big)\;.
\end{equation*}
Hence, by \eqref{f23}, $W_\epsilon(\rho,h) = \mc G_\epsilon (\rho,
\varphi) - \mc G_\epsilon (\upbar{\rho}_\epsilon,
s'(\upbar{\rho}_\epsilon))$, where $h$ and $\varphi$ are
related by \eqref{scv}. As stated above, 
$S^o_\varepsilon(\upbar{\rho}_\epsilon) =  \mc G_\epsilon (\upbar{\rho}_\epsilon,
s'(\upbar{\rho}_\epsilon))$ and therefore, in view of \eqref{f13},
\begin{equation*}
\hat{V}_\varepsilon (\rho) \;=\; \inf \, 
\big\{ \mc G_\epsilon (\rho, \varphi) \,,\: \varphi \,: \:
(\varphi, \rho) \in \Sigma  \big\} \; -\; S^o_\varepsilon
(\upbar{\rho}_\epsilon)\; . 
\end{equation*}
Since, by Theorem~\ref{t:1-1}, $\Sigma$ is the set of critical points
of the functional $\mc G_\epsilon(\pi, \cdot)$, we can drop in the
previous formula the condition that $(\varphi, \rho)$ belongs to the
unstable manifold $\Sigma$ as this condition will be automatically
satisfied by any minimizer. This concludes the heuristic proof of
Theorem~\ref{t:S=V}.

\section{The static variational problem}
\label{sec3}

In this section we analyze the variational problem \eqref{ssg}. In
particular, we prove here Theorems~\ref{t:1-1} and  \ref{t:unmine}.

\subsection*{Critical points of $\mc G _\epsilon(\rho,\cdot)$}

Fix $\rho\in M$. Recall the definition \eqref{g03} of the operator
$K_{\rho}:\mc F\to\mc F$ and that we denote by $\ms P(\rho)\subset \mc
F$ the set of fixed points of $K_\rho$. We claim that there exists a
unique $A=A(\varphi)\in(0,\infty)$, depending also on
$\varphi_0,\varphi_1,\epsilon,\rho$, such that $(K_{\rho}\,
\varphi)(1) = \varphi_1$.  Note indeed that the integral on the right
hand side of \eqref{g03} evaluated for $x=1$ is strictly increasing in
$A$, equals $0$ when $A$ vanishes, and increases to $1$ as
$A\uparrow\infty$.  Since $\varepsilon (\varphi_1-\varphi_0) =
\epsilon/\epsilon_0 \in (0,1)$ by assumption, the claim follows.  By
using that $0\le \rho\le 1 $ and $\varphi_0\le \varphi \le \varphi_1$,
it is also straightforward to check that there exist reals $0<b_0 <
b_1<\infty$, $b_i = b_i(\varphi_0, \varphi_1, \varepsilon)$, such that
for any $\rho\in M$ and $\varphi\in \mc F$ we have $b_0 \le A \le
b_1$.

Let $M^o$ be the subset of $M$ given by the smooth functions bounded
away from zero and one and satisfying the boundary conditions
$\rho(0)=\rho_0$, $\rho(1)=\rho_1$:
\begin{equation}
\label{Mo}
M^o := \big\{ \rho\in C^2([0,1])\,:\: 0<\rho<1\,,\: 
\rho(0)=\rho_0\,,\:\rho(1)=\rho_1\big\}\;.
\end{equation}
Recall the notation $\mc H_0^1(\gamma)$ introduced before
Theorem~\ref{s01} and note that for $\rho\in M^o$ the norm in $\mc
H_0^1(\sigma(\rho))$ is in fact equivalent to the norm in the standard
Sobolev space $\mc H_0^1$.  Let the Hamiltonian $\bb H : M^o
\times \mc H_0^1\to\bb R$ be the functional defined in \eqref{f15}:
\begin{equation}
\label{He}
\bb H (\rho,h) := 
\varepsilon \langle h_x, \sigma (\rho) h_x \rangle 
\; - \; \langle \varepsilon \rho_{x} - f(\rho), h_x \rangle \;.
\end{equation}

\begin{proposition}
\label{t:fp}
Fix $\varphi_0<\varphi_1$, $\epsilon\in (0,\epsilon_0)$, and
$\rho\in M$.  
\begin{itemize}
\item[(i)]The set $\ms P(\rho)$ is not empty.

\item[(ii)]If $\varphi\in \ms P(\rho)$ then $\varphi\in C^1([0,1])$.
  Moreover, there exist $\delta\in (0,1/2)$ and $C \in (0,\infty)$
  independent of $\rho$ such that any $\varphi\in\ms P(\rho)$
  satisfies $\delta \le \epsilon \varphi_x \le 1-\delta$, and
\begin{equation*}
\big| \varphi_x(x) -\varphi_x (y)\big| \le C |x-y| 
\qquad \;\forall \: x,y\in [0,1]
\end{equation*}
Finally, if $\varphi\in \ms P(\rho)$ then it solves \eqref{Deq} a.e.
  
\item[(iii)]If $\rho\in M^o$ then any $\varphi\in \ms P(\rho)$ belongs
  to $C^2([0,1])$ and satisfies $\bb H \big(\rho, s'(\rho) -
  \varphi\big) =0$.
\end{itemize}
\end{proposition}

\begin{proof}\hfill\break
\noindent
\emph{(i)} It is simple to check that $K_\rho:\mc F\to\mc F$ is
    continuous. By the convexity and compactness of $\mc F$, 
    the statement follows from Schauder's fixed point theorem. 

\noindent
\emph{(ii)} Let $\varphi\in \ms P(\rho)$.  The statement $\varphi\in
C^1([0,1])$ follows immediately from the definition of $K_\rho$.  By
using that $0<b_0 \le A \le b_1<\infty$, it is simple to check there exists
$\delta>0$ such that $\delta \le \epsilon\varphi_x\le 1-\delta$ as
well as the Lipschitz bound on $\varphi_x$.  In view of these bounds,
we can rewrite the equation $K_\rho\, \varphi = \varphi$ as
\begin{equation*}
\log \frac{\varepsilon \varphi_x(x)}{1- \varepsilon \varphi_x(x)} 
\;=\; \log A \;+\; \frac 1\varepsilon \int_0^x 
\Big[  \big( 1 + e^{\varphi (y)} \big)^{-1} - \rho(y) \Big] \, dy\;. 
\end{equation*}
Since $\varphi_x$ is a.e.\ differentiable, the above identity implies
that $\varphi$ satisfies the differential equation in \eqref{Deq} a.e.
    
\noindent
\emph{(iii)} The first statement is trivial. To prove the second,
observe that if $\rho\in M^o$ then $s'(\rho) - \varphi$ vanishes at
the boundary.  Recalling that $f = \sigma$, an integration by parts
shows that $\bb H \big( \rho, s'(\rho) - \varphi \big) = 0 $ is
equivalent to
\begin{equation*}
\big\langle \rho_x ,  (1-\varepsilon \varphi_x) \big\rangle 
\;-\; \big\langle \sigma(\rho) ,  \varphi_x  
(1 - \varepsilon \varphi_x) \big \rangle \;=\; 0\; .
\end{equation*}
To eliminate from this equation any derivative of $\rho$, we need to
integrate by parts the first term. To avoid boundary terms, we add and
subtract $e^\varphi /[1+e^{\varphi}]$ from $\rho$ and then integrate
by parts. After these steps the previous equation becomes
\begin{equation*}
\Big\langle \rho - \frac{e^\varphi}{1 + e^\varphi} \,,\,  
\varepsilon \varphi_{xx} \;+\;
\varphi_x \,(1-\varepsilon\varphi_x) \, 
\Big( \rho - \frac{1}{1 + e^\varphi} \Big) \Big\rangle 
\;=\; 0\;,
\end{equation*}
where we used that $\sigma(a)=a(1-a)$ which implies that $\sigma(b) -
\sigma(a) = (b-a)(1-a-b)$.  To conclude the proof it is now enough to
recall that, in view of item (ii), $\varphi$ solves \eqref{Deq}.
\end{proof}

\begin{proof}[Proof of Theorem~\ref{t:unmine}]
    \hfill\break 
\noindent
\emph{(i)} Recall that $\epsilon_0(\varphi_1-\varphi_0)=1$. Let
$\tilde{\mc F} :=\{\varphi\in C^1([0,1]):\, \varphi(0)=\varphi_0,\,
\varphi(1)=\varphi_1,\, 0\le \epsilon \varphi_x\le 1 \}$, and observe
that $\tilde{\mc F}\subset \mc F$.  Fix $\rho\in M$ and consider the
integro-differential operator $\mc K_\rho^{(1)}$ on $\tilde{\mc F}$
defined by
\begin{equation*}
\big(\mc K_\rho^{(1)} \varphi\big)\, (x) 
:= \varphi_0 +\frac x\epsilon 
- \Big( \frac 1\epsilon -\frac 1{\epsilon_0} \Big)
\frac{ 
\displaystyle
\int_{0}^x 
\exp\Big\{ \int_{0}^y \mc R^{(1)} (\rho,\varphi;z)\,dz \Big\} \,dy }
{\displaystyle
\int_{0}^1 
\exp\Big\{ \int_0^y \mc R^{(1)} (\rho,\varphi;z)\,dz
\Big\}\,dy } \;,
\end{equation*}
where
\begin{equation*}
\mc R^{(1)} (\rho,\varphi;x) :=   \epsilon^{-1}
\Big[ \rho(x) - \frac{1}{1+e^{\varphi(x)}}\Big] \varphi_x(x) \;,
\end{equation*}
which is informally obtained multiplying \eqref{Deq} by $\varphi_x$
and integrating the resulting equation.  It is simple to check that if
$\epsilon_0-\epsilon$ is small enough, then $\mc K_\rho^{(1)} :
\tilde{\mc F}\to \tilde{\mc F}$.

Let $\varphi\in \ms P(\rho)$. By Proposition~\ref{t:fp}, $\varphi$
solves \eqref{Deq} a.e.\ and therefore it is also a fixed point of
$\mc K_\rho^{(1)}$.  Consider now the set $\tilde{\mc F}$ endowed with
the distance $\mathrm{d} (\varphi,\psi):= \sup_x \, | \varphi_x(x)
-\psi_x(x)|$.  It is simple to show that, provided
$\epsilon_0-\epsilon$ is small enough, the operator $\mc K_\rho^{(1)}$
is a contraction with respect to this distance.  Namely, there exists
$\alpha\in (0,1)$ such that $\mathrm{d}\big(\mc K_\rho^{(1)} \varphi,
\mc K_\rho^{(1)} \psi) \le \alpha \, \mathrm{d}(\varphi,\psi)$.  This
yields uniqueness of the fixed point of $\mc K_\rho^{(1)}$ and
concludes the proof.

\noindent
\emph{(ii)} The proof is achieved by the same argument of the previous
item by considering the integro-differential operator ${\mc
  K}_\rho^{(2)}$ on $\tilde{\mc F}$ defined by
\begin{equation*}
\big( {\mc K}_\rho^{(2)} \varphi\big)\, (x) 
:=   \varphi_0 + \big(\varphi_1-\varphi_0\big)\:
\frac{ 
\displaystyle
\int_{0}^x \exp\Big\{ 
\int_{0}^y {\mc R}^{(2)} (\rho,\varphi;z) \,dz \Big\}\,dy }
{\displaystyle \int_{0}^1 
\exp\Big\{ \int_{0}^y {\mc R}^{(2)} (\rho,\varphi;z)\,dz
\Big\} \,dy }\;,
\end{equation*}
where
\begin{equation*}
{\mc R}^{(2)}(\rho,\varphi;x) :=  \epsilon^{-1}
\Big[ \frac{1}{1+e^{\varphi(x)}} -\rho(x) \Big] 
\big[ 1-\epsilon\varphi_x(x) \big]\;,
\end{equation*}
which is informally obtained multiplying \eqref{Deq} by
$1-\epsilon\varphi_x$ and integrating the resulting equation.

\noindent
\emph{(iii)} Fix a strictly increasing function $\rho\in M\cap
C^1([0,1])$ and let $\varphi\in \ms P(\rho)$.  We shall show that the
quadratic form given by the second variation of the functional $\mc
G_\varepsilon (\rho, \cdot)$ evaluated at $\varphi$ is uniformly
elliptic. This implies uniqueness of the critical point.

The second variation of $\mc G_\varepsilon (\rho , \cdot)$ evaluated
at $\varphi$ is the quadratic form
\begin{equation}
\label{f08}
\Big\langle h \,,\, \frac {\delta^2}{\delta \varphi^2} 
\mc G_\varepsilon(\rho, \varphi) \, h \Big\rangle 
\;=\; \int_{0}^{1}  \Big[ \frac {\varepsilon \, h_x^2} {\varphi_x 
(1-\varepsilon \varphi_x)} 
\;-\; \frac {e^{\varphi} \, h^2 }{\big( 1+ e^{\varphi}\big)^2}
\Big]\, dx
\end{equation}
defined on functions $h$ in $\mc H^1_0$.  Let $\psi :=\varphi_x$.
Since $\rho$ and $\varphi$ are smooth, by differentiating \eqref{Deq}
we deduce that
\begin{equation*}
\Big( \frac{\varepsilon \psi_x}{\varphi_x (1- \varepsilon \varphi_x)} 
\Big)_{\! x} \;+\; \frac {e^{\varphi}\, \psi }{\big(1+ e^{\varphi}\big)^2}  
\;=\; - \; \rho_x\; .
\end{equation*}
Performing the change of variables $h = \psi \, g$ in \eqref{f08},
which is legal because $\psi$ is smooth and strictly positive, a two
lines computation based on the previous identity shows that
\begin{equation*}
\Big\langle (\psi \,g)  , \frac {\delta^2}{\delta \varphi^2} 
\mc G_\varepsilon(\rho, \varphi) \, (\psi\, g) \Big\rangle
= \int_{0}^1\Big[ 
\frac {\varepsilon \psi^2}{\varphi_x [1- \varepsilon \varphi_x]} 
\, g_x^2    \;+\; \psi  \, \rho_x \, g^2 \Big] \, dx \; .
\end{equation*}
By item (ii) in Proposition~\ref{t:fp} and the hypothesis $\rho_x >
0$, we deduce that there exists a constant $c>0$ depending on $\rho$,
but independent of $\varphi\in \ms P(\rho)$, such that
\begin{equation*}
\Big\langle h , \frac {\delta^2}{\delta \varphi^2} 
\mc G_\varepsilon(\rho, \varphi) \, h \Big\rangle
\;\ge \; c \, \langle h,h\rangle
\end{equation*}
which concludes the proof.
\end{proof}

\subsection*{Minimizers of $\mc G _\epsilon(\rho,\cdot)$} 

We here analyze the minimizers of the functional $\mc G_\epsilon(\rho,
\cdot)$.

\begin{lemma}
\label{s02}
Let $\rho\in M$ be smooth. Then any minimizer of $\mc G_\epsilon
(\rho,\cdot)$ is a fixed point of $K_\rho$, i.e.\   
\begin{equation*}
{\ms F} (\rho) \;\subset\; \ms P (\rho)
\end{equation*}
\end{lemma}

\begin{proof}
To show that any minimizer of the variational problem \eqref{ssg} is a
fixed point of $K_{\rho}$ we use a dynamical argument.  Let
$\varphi\in \mc F$ be such that $0\le \epsilon\varphi_x\le 1$ and
consider the evolution equation
\begin{equation}
\label{gfG}
\begin{cases}
v_t = \varepsilon v_{xx} + v_x \, (1- \varepsilon v_x) \,
\big[ \rho - (1+e^v)^{-1} \big] \;, \\ 
v(t,0) =\varphi_0\;, \quad v(t,1) =\varphi_1\;, \\
v(0,\cdot )=\varphi (\cdot)\;.
\end{cases}
\end{equation}
Since $\rho$ is smooth, by classical results on uniformly parabolic
equations, the solution $v$ is smooth in $(0,\infty) \times [0,1]$.

The functional $\mc G_\varepsilon(\rho, \cdot)$ is a Lyapunov
functional for the evolution \eqref{gfG}. Indeed, for $t>0$ we have
\begin{equation*}
\frac {d}{dt} \mc G_\varepsilon(\rho, v(t)) \;=\; - \int_{0}^1 \,  
v_x \, (1- \varepsilon v_x)\, \Big[ 
\frac{\varepsilon v_{xx}}{v_x \, (1- \varepsilon v_x)} \;+\;
\rho \;-\; \frac 1{1+e^v} \Big]^2\, dx\;,
\end{equation*}
which shows that $\mc G_\varepsilon(\rho, v(t)) \le \mc
G_\varepsilon(\rho, \varphi) $ provided $\varphi$ is smooth.  By a
standard approximation argument, we then extend this inequality to any
$\varphi\in \mc F$.

Let $F=\rho - (1+e^v)^{-1}$ and $w=v_x$. Since $v_t(t,0) = v_t(t,1)
=0$, $t>0$, a simple computation shows that $w$ solves
\begin{equation*}
\left\{
\begin{aligned}
& w_t = \varepsilon w_{xx} + \big[w \, (1-\varepsilon w) \, F \big]_x  \\
& \varepsilon w_x(t,i) + w(t,i) \, 
[1-\varepsilon w(t,i)] \, F(t,i) = 0\;, \; i=0,1\;, \\
& w(0, \cdot) = \varphi_x (\cdot)\;.
\end{aligned}
\right.
\end{equation*}
Since $F_x$ is bounded on compact subsets of $(0,\infty)\times[0,1]$
and $\varphi_x$ is neither identically equal to
$0$ nor to $\varepsilon^{-1}$, Theorem 3.7 in
\cite{pw} and the remark (ii) following it, imply that $0< \varepsilon
w(t,x) <1$ for any $(t,x) \in (0,\infty)\times [0,1]$.  This proves
that $0< \varepsilon v_x <1$ in $(0,\infty)\times [0,1]$.

Let now $\varphi \in {\ms F}(\rho)$, i.e.\ $\varphi$ is a minimizer for
$\mc G_\varepsilon(\rho, \cdot)$. We deduce $\mc G_\varepsilon(\rho,
v(t))=\mc G_\varepsilon(\rho, \varphi)$, namely that for each $t\ge 0$
the function $v(t)$ is a minimizer for $\mc G_\epsilon(\rho,\cdot)$.
Since, for $t>0$, the function $v(t)$ is smooth and $0< \epsilon v_x
(t) < 1$, it satisfies the Euler-Lagrange equation \eqref{Deq}.  In
particular, $v(t)$ is a fixed point of $K_\rho$, that is $v(t) \in \ms
P (\rho)$.  Since $v(t)$ converges to $\varphi$ strongly in
$L^1([0,1])$ as $t\downarrow 0$, by taking the limit $t\downarrow 0$
in the equation $K_\rho v(t) = v(t)$, we conclude that $\varphi\in \ms
P(\rho)$.
\end{proof}

Since $\upbar\rho_\epsilon$ is smooth and strictly increasing, item
(iii) in Theorem \ref{t:unmine} implies that $\ms
P(\upbar\rho_\epsilon)$ is a singleton.  A simple computation shows
that $s'(\upbar{\rho}_\varepsilon)$ solves the differential equation
\eqref{Deq} for $\rho = \upbar{\rho}_\varepsilon$. By the previous
lemma, $s'(\upbar{\rho}_\varepsilon)$ is the unique minimizer of
$\mc G_\epsilon (\upbar{\rho}_\varepsilon ,\cdot)$.  Hence,
\begin{equation}
\label{f09}
S^o_\varepsilon(\upbar{\rho}_\varepsilon) 
\;=\; \inf_{\varphi \in \mc F}
\mc G_\varepsilon(\upbar{\rho}_\varepsilon,\varphi) \;=\; 
\mc G_\varepsilon(\upbar{\rho}_\varepsilon, s'(\upbar{\rho}_\varepsilon)) \;.  
\end{equation}

The previous lemma proves the first statement in Theorem~\ref{t:1-1}
for smooth functions $\rho$. The proof of the general result is based
in the following estimate.

\begin{lemma}
\label{t:int}  
Fix $\varphi_0<\varphi_1$ and $\epsilon\in (0,\epsilon_0)$.  There
exists $\delta\in (0,1/2)$ such that for any $\rho\in M$ and any $\varphi \in
{\ms F}(\rho)$ we have $ \delta\le \epsilon \varphi_x \le 1-\delta$
a.e.
\end{lemma}

The most direct approach to prove this lemma would be by
contradiction.  Assuming the existence of a minimizer $\varphi$ with
derivative not bounded away from $0$ and $\epsilon^{-1}$, we would
need to construct a suitable function $\psi$ with derivative bounded
away from $0$ and $\epsilon^{-1}$ such that $\mc G_\epsilon(\rho,\psi)
< \mc G_\epsilon(\rho,\varphi)$.  Our attempts in this direction have
not however been successful and we shall prove Lemma~\ref{t:int} by an
indirect route based on a geometric characterization of the
minimizers, results from convex analysis, and Lemma~\ref{s02} which
implies the statement for smooth $\rho$.  Postponing this proof, we
show Theorem~\ref{t:1-1}.

\begin{proof}[Proof of Theorem~\ref{t:1-1}: the inclusion $\ms
  F(\rho)\subset \ms P(\rho)$]
Let $\rho\in M$ and $\varphi \in {\ms F} (\rho)$.  In view of
Lemma~\ref{t:int}, we easily deduce that $\varphi$ satisfies the
Euler-Lagrange equation \eqref{Deq} weakly, namely that for any
$\lambda\in C^\infty_0([0,1])$
\begin{equation*}
\big\langle \epsilon s'(\epsilon \varphi_x), \lambda_x \big\rangle  
+\Big\langle -\rho + \frac 1{1+e^\varphi} , \lambda\Big\rangle = 0\; .
\end{equation*}
We deduce there exists a constant $C$ such that for a.e.\ $x\in [0,1]$ 
\begin{equation*}
\epsilon s'\big(\epsilon\varphi_x(x) \big)
\;=\; C\; +\; \int_0^x \Big[ \frac 1{1+e^\varphi} -\rho\Big] \, dy\;.  
\end{equation*}
Recalling \eqref{g03} and that $A=A(\varphi)$ is chosen so that
$(K_{\rho}\varphi )(1)=\varphi_1$, it is straightforward to check that
$K_{\rho}\varphi =\varphi$.
\end{proof}

To conclude the proof of Theorem~\ref{t:1-1} it remains to prove the
continuity properties of the functional $S_\epsilon$. 
To this end, we consider the space $L^1([0,1])$ endowed with the weak
topology and, recalling \eqref{s=}, we let 
$\Lambda : L^1([0,1]) \to (-\infty,+\infty]$ be the lower semicontinuous
functional defined by
\begin{equation}
\label{dLe}
\Lambda (\varphi)  := 
\begin{cases}
\displaystyle 
\int_0^1 \big[ s(\epsilon\varphi_x) + \varphi 
- \log( 1+ e^\varphi) \big]\, dx 
& \textrm{if $\varphi\in\mc F$} \;, \\
+\infty
& \textrm{otherwise}\;,
\end{cases}
\end{equation}
where we understand that $\Lambda (\varphi)=+\infty$ unless
$\varphi$ is absolutely continuous, and the density of the measure
$d\phi$, denoted by $\varphi_x$, satisfies $0\le \epsilon \varphi_x
\le 1$ a.e.\ Note that the set $\{\varphi :\,\Lambda
(\varphi)<+\infty\}$ is compact in $L^1([0,1])$.  Recalling that we
consider $L^\infty([0,1])$ endowed with the weak* topology, the
Legendre transform of $\Lambda $ is defined as the function
$\Lambda^* : L^\infty([0,1])\to \bb R$ given by
\begin{equation}
\label{dL*e}
\Lambda^* (\rho)
:= \sup_{\varphi}\big\{\langle \rho,\varphi\rangle
-\Lambda (\varphi)\big\}\,,
\end{equation}
where the supremum is carried over all functions $\varphi$ in
$L^1([0,1])$. Recalling \eqref{s=}, we also let $\ms S :
L^\infty([0,1])\to (-\infty,+\infty]$ be the functional defined by
\begin{equation}
\label{dSs}
\ms S (\rho) := \int_0^1 s(\rho) \, dx\;.
\end{equation}
In view of \eqref{ssg}, the previous definitions imply
\begin{equation}
\label{S=S-L}
S^o_\epsilon  = \ms S  -  \Lambda^*\;,
\end{equation}
where we understand that $S^o_\epsilon$, defined in \eqref{ssg}, 
has been extended to a functional on $L^\infty([0,1])$ by setting 
$S^o_\epsilon(\rho) =+\infty$ for $\rho\not\in M$.

\begin{lemma}
\label{t:cont}
Fix $\varphi_0<\varphi_1$, $\epsilon\in (0,\epsilon_0)$ and a not
empty closed subset $K$ of $\mc F$. Let $\Lambda^*_K: M \to \bb R$ be
the functional defined by
\begin{equation*}
\Lambda^*_K (\rho) \;:=\:
\sup\: \big\{ \langle \rho, \varphi\rangle - \Lambda (\varphi)
\,,\: \varphi\in K  \big\} \;.
\end{equation*}
Then $\Lambda^*_K$ is continuous.
\end{lemma}

\begin{proof}
The continuity of the map $M\ni \rho \mapsto \langle \rho,
\varphi\rangle\in \bb R$ for a given $\varphi \in \mc F$ implies
immediately the lower semicontinuity of $ \Lambda^*_K$.  To prove the
upper semicontinuity, fix a sequence $\{\rho^n\} \subset M$ converging
to $\rho$.  Since $\Lambda^*_K (\rho^n) < \infty$, $n\ge 1$, there
exists a sequence $\{\varphi^n\}\subset K$ such that
\begin{equation*}
\lim_{n\to\infty} \Lambda^*_K (\rho^n) \;=\; 
\lim_{n\to\infty} \big\{ \langle \rho^n, \varphi^n\rangle - 
\Lambda (\varphi^n) \big\}\;.
\end{equation*}
Since $\mc F$ is compact, by taking a subsequence, if necessary, we
may assume that $\{\varphi^n\}$ converges in $\mc F$ to some
$\varphi\in K$.  As $\rho^n\to \rho $ in $M$ and $\varphi^n\to
\varphi$ in $\mc F$ imply $\langle \rho^n,\varphi^n\rangle \to \langle
\rho,\varphi \rangle$ and $\liminf_n \Lambda (\varphi^n) \ge \Lambda
(\varphi)$, we deduce
\begin{equation*}
\limsup_n  \, \Lambda^*_K  (\rho^n) \; \le \;  
\langle \rho, \varphi\rangle - \Lambda (\varphi)
\; \le \; \Lambda^*_K  (\rho)\;,
\end{equation*}
which is the desired upper semicontinuity.
\end{proof}

\begin{proof}[Proof of Theorem~\ref{t:1-1}: conclusion]
  Recall \eqref{dSs} and observe that the convexity of $s$ immediately
  implies the lower semicontinuity of $\ms S$ on $M$.  On the other
  hand, $\ms S : M\to \bb R$ is clearly continuous with respect to the
  strong topology of $L^1([0,1])$.  In view of the decomposition
  \eqref{S=S-L}, the proof of Theorem~\ref{t:1-1} is then concluded by
  applying Lemma~\ref{t:cont} with $K=\mc F$.
\end{proof}

In the sequel we shall need the following density result.  Recall the
definition of the set $M^o$ introduced in \eqref{Mo}.

\begin{lemma}
\label{s16} 
Fix $\rho$ in $M$. There exist a function $\varphi$ in $\ms F (\rho)$,
a sequence $\{\rho^n\}\subset M^o$, and a sequence
$\{\varphi^n\}\subset \mc F$, $\varphi^n \in \ms F (\rho^n)$, such
that $\rho^n\to \rho$ strongly in $L^1([0,1])$, $\varphi^n\to \varphi$
in the $C^1$ topology, and $\mc G_\varepsilon (\rho^n,\varphi^n) \to
\mc G_\epsilon (\rho,\varphi)$.
\end{lemma}

\begin{proof}
Fix $\rho$ in $M$ and pick a sequence $\{\rho^n\}\subset M^o$
converging to $\rho$ strongly in $L^1([0,1])$. Since $\mc F$ is
compact and $\mc G_\varepsilon$ is lower semicontinuous, there exists
a sequence $\{\varphi^n\}\subset \mc F$ such that $\varphi^n\in {\ms
  F}(\rho^n)$. By Lemma~\ref{s02} we also have $\varphi^n\in \ms
P(\rho^n)$. Hence, by item (ii) in Proposition~\ref{t:fp} and
Ascoli-Arzel\`a theorem, $\{\varphi^n\}$ is a precompact sequence in
$C^1([0,1])$. In particular, by taking if necessary a subsequence,
there exists a function $\varphi$ in $\mc F\cap C^1([0,1])$ such that
$\varphi^n\to\varphi$ in $C^1([0,1])$.  The last statement and the
choice of $\{\rho^n\}$ imply $\mc G_\varepsilon (\rho^n,\varphi^n) \to
\mc G_\epsilon (\rho,\varphi)$.

It remains to show that $\varphi\in {\ms F}(\rho)$.  By taking the
limit $n\to\infty$ in the equation $K_{\rho^n}\varphi^n =\varphi^n$,
we readily deduce that $\varphi\in\ms P(\rho)$.  To show that
$\varphi$ belongs to ${\ms F}(\rho)$ we proceed as follows.  Since
$\varphi^n\in{\ms F}(\rho^n)$, Lemma~\ref{t:cont} and the continuity
of $\ms S$ with respect to the strong $L^1([0,1])$ topology imply $\mc
G_\varepsilon (\rho^n,\varphi^n) =S^o_\epsilon(\rho^n) \to
S^o_\epsilon(\rho)$.  Therefore, as $\mc G_\varepsilon
(\rho^n,\varphi^n) \to \mc G_\epsilon (\rho,\varphi)$, we deduce that
$S^o_\epsilon(\rho)= \mc G_\epsilon (\rho,\varphi)$, i.e.\ $\varphi\in
{\ms F}(\rho)$.
\end{proof}

\subsection*{Convexity considerations} 

Let $X$ and $X^*$ be vector spaces in duality with respect to
$\langle\cdot,\cdot\rangle$. We consider $X$ and $X^*$ respectively
endowed with the weak and weak* topology.  Recall that $f: X\to
(-\infty,+\infty]$ is \emph{G\^ateaux differentiable} at $x$ if there
exists $\ell\in X^*$ such that
\begin{equation*}
\lim_{\lambda \downarrow 0}\frac{1}{\lambda}
\big[f(x+\lambda v)-f(x)-\lambda \langle\ell,v\rangle\big]=0
\qquad \text{ for any } v\in X\;.
\end{equation*}
In such a case we denote $\ell$ by $D_\mathrm{G} f(x)$.  In general,
we define the \emph{G\^ateaux sub-differential} $D^-_\mathrm{G}f(x)$
and \emph{G\^ateaux super-differential} $D^+_\mathrm{G}f(x)$ of $f$ at
the point $x$ as the following, possibly empty, convex subsets of
$X^*$
\begin{equation}
\label{gsd}
\begin{aligned}
& D^+_\mathrm{G}f(x)
:=
\Big\{\ell\in X^*\,:\: \limsup_{\lambda\downarrow0}\frac{1}{\lambda}
\big[f(x+\lambda v)-f(x)-\lambda \langle\ell,v\rangle\big]\leq 0
\;\text{ for any } v\in X\Big\} \\
& D^-_\mathrm{G}f(x)
:=
\Big\{\ell\in X^*\,:\: \liminf_{\lambda\downarrow 0}\frac{1}{\lambda}
\big[f(x+\lambda v)-f(x)-\lambda \langle\ell,v\rangle\big]\geq 0
\;\text{ for any } v\in X\Big\}
\end{aligned}
\end{equation}
in which we understand $D^\pm_\mathrm{G}f (x) =\varnothing$ if
$f(x)=+\infty$. 

In the context of convex analysis, the \emph{sub-differential} of a
\emph{convex} function $f:X\to (-\infty,+\infty]$ at a point $x\in X$,
here denoted by $\partial f(x)$, is the set of linear functionals
$\ell\in X^*$ such that
\begin{equation*}
f(y) - f(x) \;\ge\; \langle \ell, y-x\rangle 
\quad \text{ for any } y\in X
\end{equation*}
in which we understand that $\partial f(x) =\varnothing$ if
$f(x)=+\infty$.  It is well-know, see e.g.\ \cite[Prop.~1.5.3]{ET},
that if $f$ is convex and G\^ateaux differentiable at $x$ then
$\partial f(x) =\{ D_\mathrm{G} f(x) \}$.  It is also simple to check
that for any convex function $f:X\to (-\infty,+\infty]$ the equality
$D^-_\mathrm{G} f (x) = \partial f(x)$ holds for all $x\in X$.

Recall from \eqref{Fer} that the set ${\ms F} (\rho)\subset \mc F$,
$\rho\in M$, represents the minimizers of $\mc G_\epsilon(\rho,
\cdot)$.  In view of definition \eqref{dLe}, the set ${\ms F} (\rho)$
coincides with the maximizers for the variational problem on the right
hand side of \eqref{dL*e}.  We here consider such variational problem
for $\rho\in L^\infty([0,1])$ and still denote by ${\ms F} (\rho)$ the
set of maximizers, i.e.\ ${\ms F} (\rho)= \mathrm{arg \,sup}\:\big\{
\langle\rho,\varphi\rangle - \Lambda(\varphi) \:,\, \varphi \in
L^1([0,1])\big\}$.  For each $\rho\in L^\infty([0,1])$, the set ${\ms
  F}(\rho)$ is a not empty compact subset of $\mc F$.  Given a set $A$
we denote by $\mathrm{co}\: A$ its convex hull.

\begin{proposition}
\label{t:DL*}
The functional $\Lambda^*: L^\infty([0,1]) \to \bb R$ defined in
\eqref{dL*e} is lower semicontinuous and convex. Moreover, for each
$\rho\in L^\infty([0,1])$ we have
\begin{equation*}
\begin{split}
& D^-_\mathrm{G} \Lambda^* (\rho) 
\; =\;  \partial \Lambda^* (\rho)  
\; =\;  \mathrm{co\:} {\ms F} (\rho)\;,
\\
&
D^+_\mathrm{G}\Lambda^* (\rho) \; =
\begin{cases}
\{\varphi\} 
& \textrm{if ${\ms F}(\rho)=\{\varphi\}$ for some $\varphi\in\mc
  F$}\;,   
\\
\varnothing 
& \textrm{otherwise}\;.
\end{cases}
\end{split}
\end{equation*}
In particular, $\Lambda^*$ is G\^ateaux differentiable at $\rho$ if
and only if ${\ms F}(\rho)$ is a singleton.
\end{proposition}

To prove this statement we need the following elementary result from
convex analysis.  We say that $f: X \to (-\infty,+\infty]$ has
\emph{super-linear growth} iff for each affine function on $X$, i.e.\
a map $X\ni x \mapsto \langle\ell, x\rangle +\alpha \in\bb R$ for some
$\ell\in X^*$ and $\alpha\in \bb R$, there exists a compact
$K=K_{\ell,\alpha} \subset X$ such that $f(x)\ge\langle\ell, x\rangle
+\alpha$ for any $x\not\in K$.  Observe that if $f$ has super-linear
growth, then $f$ is coercive.

\begin{lemma}
\label{t:ca} 
Let $f: X \to (-\infty,+\infty]$ be lower semicontinuous with
super-linear growth.  Denote by $f^{**}$ the convex envelope of $f$,
i.e.\ the largest \emph{convex} and lower semicontinuous function
below $f$.  Then
\begin{equation*}
\mathrm{co\:} \, \mathrm{arg \,inf} \{ f(x)\,,\: x\in X \}
\;=\; \mathrm{arg \,inf} \{ f^{**}(x)\,,\: x\in X \}\;.
\end{equation*}
\end{lemma}

\begin{proof}
The inclusion $ \mathrm{co\:} \mathrm{arg \,inf} \{ f(x),\,x\in X \}
\subset \mathrm{arg \,inf} \{ f^{**}(x),\, x\in X\}$ is trivial; to
prove the converse we shall argue by contradiction and assume, with no
loss of generality, that $\inf f = 0$.  Recall that the
\emph{epigraph} of $f$ is the subset of $X\times \bb R$ given by
$\mathrm{epi\:}f := \{ (x,t) :\, f(x) \le t\}$.  The convex envelope
$f^{**}$ is then characterized by the identity $\mathrm{co\:}
\mathrm{epi\:} f = \mathrm{epi\:}f^{**}$, see e.g.\
\cite[Proposition~1.3.2]{ET}.

Assume, by contradiction, that there exists $\upbar{x}\in \mathrm{arg
  \,inf} \{ f^{**}(x),\,x\in X \}$ such that $\upbar{x}\not\in
\mathrm{co\:} \: \mathrm{arg \,inf} \{ f(x),\,x\in X \}=: C$.  Since
$f$ is coercive and lower semicontinuous, its sub-level sets, and
therefore $C$, are compact. By the Hahn-Banach theorem there exist
$\alpha\in \bb R$ and $\ell \in X^*$ such that $C \subset \{ x:\,
\langle \ell, x\rangle > \alpha\}$ and $\upbar{x} \in \{ x:\, \langle
\ell, x\rangle < \alpha\}$.  Since $C$ is compact, we can find an open
neighborhood $A$ of $C$ such that $A \subset \{ x:\, \langle \ell,
x\rangle > \alpha\}$. Since $f$ has super-linear growth, we can also
find a compact $K$ such that for any $x\not\in K$ we have $f(x)\ge -
\langle\ell,x\rangle +\alpha$. Let now
\begin{equation*}
m:= \inf \big\{ f(x)\,,\: x\not\in A\big\} >0 \,,
\quad\qquad 
M := \sup\big\{ |\langle\ell,x\rangle|\,,\; x \in K \big\} <
\infty\,,
\end{equation*}
and set $\lambda := \max\big\{1\,,\, (M+\alpha)/m\big\} \in
[1,\infty)$. Consider the following half-spaces in $X\times \bb R$
\begin{equation*}
    \Pi_+ := 
    \big\{ (x,t)\,:\: 
    \langle\ell,x\rangle + \lambda\, t \ge \alpha 
    \big\} 
    \quad\;
    \Pi_- := 
    \big\{ (x,t)\,:\: 
    \langle\ell,x\rangle + \lambda\, t \le \alpha 
    \big\}    
\end{equation*}
and observe that $(\upbar{x},0)$ belongs to the interior of $\Pi_-$. 
It is also easy to check that $\mathrm{epi\:} f \subset \Pi_+$.
Therefore, as $\mathrm{co\:} \mathrm{epi\:} f $ is the intersection of
all the half-spaces containing $\mathrm{epi\:} f $, we deduce that 
$(\upbar{x},0)\not\in \mathrm{co\:} \mathrm{epi\:} f $ which yields 
the desired contradiction.
\end{proof}

\begin{proof}[Proof of Proposition~\ref{t:DL*}]
The lower semicontinuity and convexity of $\Lambda^*$ follows
trivially from its definition.  
Given $\rho\in L^\infty([0,1]$, set
\begin{equation*}
\widetilde {\ms F} (\rho) := \mathrm{arg \, sup} \;
\big\{ \langle\rho,\varphi \rangle - \Lambda^{**} (\varphi) 
\:,\, \varphi \in L^1([0,1]) \big\}
\end{equation*}
where, we recall, $\Lambda^{**}$ denotes the convex envelope of
$\Lambda$.  As $\{\varphi\in L^1([0,1])\,:\: \Lambda
(\varphi)<\infty\}$ is compact, $\Lambda$ has super-linear growth and
Lemma~\ref{t:ca} yields $\widetilde {\ms F}(\rho)=\mathrm{co\:} {\ms
  F}(\rho)$.  On the other hand, by e.g.\ \cite[Prop.~1.4.1 and
Cor.~1.4.1,]{ET},
\begin{equation*}
\Lambda^*(\rho)\;=\;\Lambda^{***}(\rho)
\;:=\; \sup_{\varphi} \big\{ \langle\rho,\varphi \rangle -
\Lambda^{**}(\varphi) \big\} \;.
\end{equation*}
Since we are now dealing with Legendre duality between convex
functions, by e.g.\ \cite[Prop.~1.5.1 and Cor.~1.5.2]{ET}, $\varphi\in
\partial \Lambda^*(\rho)$ iff $\rho\in \partial \Lambda^{**}
(\varphi)$ iff $\varphi\in\widetilde{\ms F} (\rho)$.  This concludes
the proof of the equality $\partial \Lambda^* (\rho) = \mathrm{co\:}
{\ms F} (\rho)$.
  
We claim that if ${\ms F}(\rho)=\{\varphi\}$ for some $\varphi\in\mc
F$ then $\varphi\in D^+_\mathrm{G} \Lambda^* (\rho)$.  This statement
completes the proof of the proposition.  Indeed, if $D^+_\mathrm{G}
\Lambda^* (\rho) \neq \varnothing $ and $D^-_\mathrm{G} \Lambda^*
(\rho) \neq \varnothing$ then $\Lambda^*(\rho) $ is necessary
G\^ateaux differentiable at $\rho$.

To prove the claim, given $v\in L^\infty([0,1])$ and $\lambda>0$, pick
$\phi_\lambda \in {\ms F}(\rho+\lambda v)$. By the very definition of
$\Lambda^*(\rho)$,
\begin{equation*}
\Lambda^*(\rho + \lambda v) - \Lambda^*(\rho) 
-\lambda \langle \varphi,v\rangle 
\;\le\; \lambda \langle \varphi_\lambda -\varphi, v\rangle\;.
\end{equation*}
The proof will therefore be completed once we show that any element in
${\ms F}(\rho+\lambda v)$ converges, as $\lambda\downarrow 0$, weakly
in $L^1([0,1])$ to $\varphi$. Since we assumed ${\ms F}(\rho)$ to be a
singleton, this is a straightforward consequence of the lower
semicontinuity of $\Lambda$ and the compactness of $\mc F$.
\end{proof}

The proof of Lemma~\ref{t:int} is basically achieved by
Lemma~\ref{s02}, Proposition~\ref{t:DL*}, and the following general
result in convex analysis, which is proven in \cite{LR}.  Let $B$ be a
\emph{separable} Banach space and $f: B \to \bb R$ a continuous convex
function. Given $x\in B$ and $\ell\in \partial f(x)$ we say that
$\ell$ is \emph{approximable by unique tangent functionals} iff there
exists a sequence $\{x_n\}\subset B$ converging (in norm) to $x$ such
that $\partial f(x_n) =\{\ell_n\}$ for some $\ell_n\in B^*$ and
$\ell_n\to \ell$ in the weak* topology of $B^*$. The collection of
elements in $\partial f(x)$ approximable by unique tangent functionals
is denoted by $\partial_{\mathrm{app}} f (x)$.

\begin{theorem}
\label{t:lr}
Let $B$ be a separable Banach space and $f:B \to \bb R$ be a
continuous convex function. Then for each $x\in B$
\begin{equation*}
\partial f(x) = \upbar{\mathrm{co}} \; \partial_{\mathrm{app}} f
(x)\;, 
\end{equation*}
where $\upbar{\mathrm{co}}$ denotes the weak* closure of the convex
hull.
\end{theorem}

As the previous theorem requires to work in a separable Banach space,
we extend the functional $\Lambda^*: L^\infty([0,1])\to \bb R$, as
defined in \eqref{dL*e}, to a functional on the space $L^1([0,1])$
endowed with the strong topology.  To avoid ambiguities we shall
denote the extended functional by $\hat \Lambda$.  In other words, we
let $\hat\Lambda : \big(L^1([0,1]),\textrm{strong}\big) \to \bb R$ be
the functional defined by
\begin{equation}
\label{dhLe}
\hat\Lambda (\rho) := 
\sup_\varphi \big\{ \langle \rho, \varphi\rangle - 
\Lambda (\varphi) \big\}
\end{equation}
where the supremum is carried over all $\varphi$ in $L^1([0,1])$ such
that $\Lambda (\varphi)< \infty$.

Clearly, $\hat\Lambda$ is convex and, by the argument used in
Lemma~\ref{t:cont}, continuous with respect to the strong topology of
$L^1([0,1]$.  We claim that for each $\rho\in L^\infty([0,1])$ we have
\begin{equation}
\label{f18}
\partial \hat\Lambda (\rho) \;=\;
\partial \Lambda ^* (\rho)\; . 
\end{equation}
By the previous identity we mean that any element in $\partial
\Lambda^*(\rho)$, which a priori belongs only to $L^1([0,1])$, belongs
also to $L^\infty([0,1])$ and it is an element in $\partial
\hat\Lambda (\rho)$. Indeed, on the one hand, it follows from the
definition of sub-differentials that $\partial \hat\Lambda (\rho)
\subset \partial \Lambda^* (\rho)$ for each $\rho\in L^\infty([0,1])$.
The reverse inclusion follows from the definition of
sub-differentials, the fact, proven in Proposition~\ref{t:DL*}, that
$\partial \Lambda^* (\rho) = \mathrm{co \:} {\ms F} (\rho)$ and the
continuity of $\hat\Lambda$ with respect to the strong topology of
$L^1([0,1])$.

\begin{proof}[Proof of Lemma~\ref{t:int}]
We shall consider the integral operator $K_\rho:\mc F\to \mc F$, as
defined in \eqref{g03}, for $\rho\in L^1([0,1])$ instead of
$L^\infty([0,1])$. To avoid ambiguities, denote by $\hat{\ms P}(\rho)$
the set of $\varphi\in \mc F$ which are fixed points of $K_{\rho}$,
$\rho\in L^1([0,1])$. By the proof of Proposition~\ref{t:fp},
$\hat{\ms P}(\rho)$ is not empty and any $\varphi \in \hat{\ms
  P}(\rho)$ belongs to $C^1([0,1])$.  Furthermore, there exists
$\delta >0$ depending on $\varphi_0<\varphi_1$, $\epsilon\in
(0,\epsilon_0)$, and $|\rho|_{L^1}$ such that any $\varphi \in
\hat{\ms P}(\rho)$ satisfies the bound $\delta \le \epsilon \varphi_x
\le 1-\delta$.  We stress that $\delta$ depends on $\rho$ only via
$|\rho|_{L^1}$.

We claim that if $\rho\in L^1([0,1])$ is such that $\partial \hat
\Lambda (\rho)= \{\varphi\}$ for some $\varphi \in L^\infty([0,1])$,
then $\varphi \in \hat{\ms P} (\rho)$.  Postponing the proof of this
claim, we first conclude the proof of the lemma.  Fix $\rho\in
L^\infty([0,1])$ and $\varphi\in \partial_{\mathrm{app}} \hat
\Lambda(\rho)$. By definition, there exists a sequence
$\{\rho^n\}\subset L^1([0,1])$ converging to $\rho$ strongly in
$L^1([0,1])$ such that $\partial \hat \Lambda(\rho^n)=\{\varphi^n\}$
and $\varphi^n\to\varphi$ weak* in $L^\infty([0,1])$.  In view of the
previous claim $\varphi^n\in C^1([0,1])$ and there exists
$\delta_n>0$, depending only on $|\rho^n|_{L^1}$, such that $\delta_n
\le \epsilon \varphi^n_x \le 1 -\delta_n$. Since $\rho^n$ converges to
$\rho$ in $L^1([0,1])$, $\delta = \min \{\delta_n : n\ge 1\} >0$.  By
duality, it is readily shown that the map $\psi\mapsto
|\psi_x|_{L^\infty}$ is lower semicontinuous with respect to the weak*
convergence in $L^\infty([0,1])$.  Therefore $|\epsilon
\varphi_x|_{L^\infty} \le 1-\delta$ and, by the same argument,
$|1-\epsilon\varphi_x|_{L^\infty} \le 1-\delta$, that is, $\delta\le
\epsilon\varphi_x\le 1-\delta$ a.e.

Lemma~\ref{t:int} thus holds for $\varphi$ in $\partial_{\mathrm{app}}
\hat \Lambda(\rho)$. Fix now $\rho \in L^\infty([0,1])$,
$|\rho|_{L^\infty} \le 1$, and $\varphi\in \ms F(\rho)$. By
Proposition \ref{t:DL*} and by \eqref{f18}, $\varphi$ belongs to
$\partial\Lambda^*(\rho) = \partial\hat\Lambda(\rho)$.  Hence, by
Theorem~\ref{t:lr} and by the first part of the proof, there exists a
sequence $\{\psi^n\}$, $\psi^n$ convex combinations of elements in
$\partial_{\mathrm{app}} \hat \Lambda(\rho)$, such that $\delta\le
\epsilon\psi^n_x\le 1-\delta$ a.e. for some $\delta>0$ and $\psi^n$
converges in the weak* topology of $L^\infty([0,1])$ to $\varphi$.  To
conclude, it remains to recall the lower semicontinuity of
$\psi\mapsto |\psi_x|_{L^\infty}$ with respect to the weak*
convergence in $L^\infty([0,1])$.

We turn now to the proof of the claim.  Pick a sequence of smooth
functions $\{\rho^n\}$ converging to $\rho$ strongly in $L^1([0,1])$
and choose $\varphi^n\in{\ms F}(\rho^n)$.  By Proposition~\ref{t:DL*}
and \eqref{f18}, $\varphi^n \in \partial \Lambda^* (\rho^n)= \partial
\hat \Lambda (\rho^n)$.  On the other hand, by Lemma~\ref{s02}, as
$\rho^n$ is smooth, $\varphi^n\in \hat{\ms P}(\rho^n)= \ms P(\rho^n)$.
In particular $\{\varphi^n\}\subset \mc F$. By taking, if necessary, a
subsequence, the compactness of $\mc F$ now yields the existence of
$\psi\in \mc F$ such that $\varphi^{n}\to \psi$ in $\mc F$.  We can
thus take the limit $n\to\infty$ in the equation $\varphi^n =
K_{\rho^n} \varphi^n$ and conclude that $\psi\in \hat{\ms P} (\rho)$.
Finally, by using the definition of sub-differentials, we easily
deduce that $\psi \in
\partial \hat \Lambda (\rho)$.  Since we assumed $\partial \hat
\Lambda (\rho)= \{\varphi\}$, we conclude $\psi=\varphi$.
\end{proof}

\section{The quasi-potential}
\label{sec4}

Relying on the properties of the static variational problem presented
in the previous section, we prove here the main result namely, the
identity between the quasi potential $V_\epsilon$ and the functional
$S_\epsilon$. The proof is organized as follows.  We first prove the
equality $\hat V_\epsilon = V_\epsilon$.  We then show that the
algorithm presented below the statement of Theorem~\ref{t:S=V}
provides a legal test path for the variational problem \eqref{qp}. By
elementary computations, basically the ones described in the context
of Hamiltonian formalism, we deduce the inequality $\hat V_\epsilon
\le S_\epsilon$.  Finally, by using the variational definition
\eqref{e:2.6} of the action functional, we prove the inequality
$V_\epsilon\ge S_\epsilon$ by exhibiting a suitable test function $H$.
In these arguments, all the computations will be performed for smooth
paths and we will use density results to extend the bounds to
arbitrary paths.

\subsection*{The identity $\hat V_\epsilon = V_\epsilon $.}

We prove in this subsection that the variational problems \eqref{qp-1}
and \eqref{qp} are equivalent.  We emphasize that this statement does
not depend on the specific form of the flux $f$ and the mobility
$\sigma$.  The argument relies on two results.  The first one asserts
that if the action of a path $u$ in $\mc U(\upbar{\rho}_\varepsilon)$
is finite, then there exists a sequence $t_n\to -\infty$ such that
$u(t_n)-\upbar{\rho}_\varepsilon$ converges to $0$ in $\mc H^1_0$,
recall the notation for the Sobolev spaces introduced before the
statement of Theorem \ref{s01}.  The second one asserts that if a
function $\rho\in M$ is such that $\rho-\upbar{\rho}_\varepsilon$ is
small in $\mc H^1_0$, then there exists a path in a time interval of
length one which connects $\upbar{\rho}_\varepsilon$ to $\rho$ and has
small action.  In this subsection we drop the subscript $\varepsilon$
also in the stationary solution $\upbar{\rho}_\varepsilon$ and denote
its derivative by $\upbar{\rho}_x$.

\begin{lemma}
\label{s04}
Fix $T>0$ and a path $u$ in $C([-T,0], M)$ such that
$I_{[-T,0]}(u) < \infty$. Then,
\begin{equation*}
\begin{split}
& \frac \varepsilon 8 \, \int_{-T}^0 \big| u(t)- \upbar{\rho}
\big|_{\mc H_0^1}^2 \,dt \;+\; \frac 14 \, 
   \big| u(0) - \upbar{\rho}\big|_{L^2}^2  \\
&\quad\le\;  \frac 14 \, \big| u(-T) - \upbar{\rho} \big|_{L^2}^2
\;+\; I_{[-T,0]}(u) \;+\; 
\frac 1{2\varepsilon}\, \int_{-T}^0 
\big| u(t) - \upbar{\rho}|_{L^2}^2  \,dt \;.
\end{split}
\end{equation*}
\end{lemma}

\begin{proof}
The proof of this lemma is similar to the one of Lemma 4.9 in
\cite{blm1} or the one of Lemma 4.2 in \cite{flm1}. We thus just
sketch the argument.

Fix $T>0$ and recall the definition of the linear functional $L_u$
introduced in \eqref{e:2.5}. By definition of the action functional
$I_{[-T,0]}$, for any function $H$ in
$C^\infty_{0}([-T,0]\times[0,1])$,  
\begin{equation*}
L_u (H) \;-\; 
\varepsilon\, \big\langle\!\big\< H_x , \sigma(u) H_{x}\big
\>\!\big\rangle \;\le\; I_{[-T,0]}(u)\;.
\end{equation*}
Take $H=(1/2)(u-\upbar{\rho})$. This function is not smooth and does
not satisfy the boundary conditions at $-T$ and $0$, but can be
approximated by such smooth functions in the norms needed for our
purposes. This is presented with all details in the proofs of Lemma
4.9 in \cite{blm1} and Lemma 4.2 in \cite{flm1}. For instance, to
match the boundary conditions at $-T$ and $0$, we multiply $H$ by a
time dependent function which vanishes at $-T$ and $0$ and which
is close to the indicator of the interval $[0,T]$.

Integrating by parts and adding to the previous expression 
$\<\!\< f(\upbar{\rho})_x -\varepsilon \upbar{\rho}_{xx},
H\>\!\>$, which vanishes, the previous
equation implies
\begin{equation*}
\begin{split}
& \frac 14 \, \big| u(0) - \upbar{\rho} \big|_{L^2}^2 \;+\;
\frac{\varepsilon}2 
\<\!\<  u_x - \upbar{\rho}_x \,,\,  u_x - \upbar{\rho}_x \>\!\>
\;-\;  \frac{\varepsilon}4 
\big\langle\!\big\langle  u_x - \upbar{\rho}_x \, ,\, 
\sigma(u) \, [u_x - \upbar{\rho}_x] \big\rangle\!\big\rangle
\\
&\qquad
\le\; I_{[-T,0]}(u) \;+\; \frac 14 \, 
\big|u(-T) - \upbar{\rho} \big|_{L^2}^2 
\;+\; \frac{1}2 
\<\!\< f(u) - f(\upbar{\rho})\,,\, u_x - \upbar{\rho}_x \>\!\>
\end{split}
\end{equation*}
Since $\sigma(u)\le 1/4$, 
it remains to apply Schwarz inequality to bound the last term and to
recall that $f$ is Lipschitz with Lipschitz constant one.
\end{proof}

\begin{lemma}
\label{s06}
Fix a path $u$ in $\mc U(\upbar{\rho})$ such that
$I(u) < \infty$. Then
\begin{equation*}
\lim_{t\to -\infty}  \big| u(t) - \upbar{\rho} \big|_{L^2} 
\;=\; 0 \quad\text{and}\quad
\lim_{n\to\infty} \: \int_{-(n+1)}^{-n} 
\big| u(t) - \upbar{\rho} \big|_{\mc H_0^1}^2 \,dt \;=\; 0\;.
\end{equation*}
\end{lemma}

\begin{proof}
We first show that 
\begin{equation}
\label{f01}
\lim_{n\to\infty} \int_{-(n+1)}^{-n} 
\big\< u(t) - \upbar{\rho}\,,\,  u(t) - \upbar{\rho} \big\> \,dt \;=\; 0\;.
\end{equation}

Let $\{e_k , \,k\ge 1\}$ be the complete orthonormal system of
$L_2([0,1])$ given by $e_k(x) = \sqrt{2} \sin (k\pi x)$  and set 
$\theta_k(t) = \< u(t) - \upbar{\rho} , e_k\>$, $k\ge 1$. 
The integral in the previous formula can be written as
\begin{equation*}
\sum_{k\ge 1} \int_{-(n+1)}^{-n}  \theta_k(t)^2 \,dt
\;\le\; \sum_{k= 1}^{k_0} \int_{-(n+1)}^{-n}  \theta_k(t)^2 \,dt
\;+\; \frac 1{k_0^2} \sum_{k\ge  1} \int_{-(n+1)}^{-n}  
k^2 \theta_k(t)^2 \,dt 
\end{equation*}
for any $k_0\ge 1$.  Since $e_k$ is an eigenfunction of the Laplacian
with Dirichlet boundary conditions, the second term is equal to the
time integral of $(\pi k_0) ^{-2} \big| u(t) - \upbar{\rho}\big|_{\mc H_0^1}^2$.
Therefore, by Lemma \ref{s04}, the last expression is less than or
equal to
\begin{equation*}
\begin{split}
& \sum_{k= 1}^{k_0} \int_{-(n+1)}^{-n} 
\big\< u(t) - \upbar{\rho} , e_k\big\>^2  \,dt
\;+\; \frac 2 {\pi^2 k_0^2 \, \varepsilon} 
\, \big| u(-(n+1)) - \upbar{\rho} \big|_{L^2}^2\\
& \qquad +\; \frac 1{\pi^2 k_0^2} \Big\{
\frac 8 \varepsilon \, I_{[-(n+1),-n]} (u) \;+\; 
\frac 4{\varepsilon^2}\, \int_{-(n+1)}^{-n} 
\big| u(t) - \upbar{\rho} \big|_{L^2}^2  \,dt \Big\} \;.
\end{split}
\end{equation*}
Since $u(t) - \upbar{\rho}$ is absolutely bounded by $1$ and since 
$I_{[-(n+1),-n]} (u) \le I (u)$, the previous
expression is bounded by
\begin{equation*}
\sum_{k= 1}^{k_0} \int_{-(n+1)}^{-n} \big\< u(t) - 
\upbar{\rho} , e_k\big\>^2  \,dt
\;+\; \frac 1 {\pi^2 k_0^2} 
\, \Big\{ \frac 2 \varepsilon \;+\; \frac 4 {\varepsilon^2}
\;+\; \frac 8 \varepsilon \, I (u) \Big\}\;.
\end{equation*}
Since $u(t)$ converges weakly in $L^2([0,1])$ to $\upbar{\rho}$, as
$t\to -\infty$, to conclude the proof of the claim \eqref{f01} it
remains to let $n\to\infty$ and then $k_0\to\infty$.

Since $u$ belongs to $C((-\infty, 0], M)$, the function $t\mapsto 
| u(t) - \upbar{\rho}|_{L^2}$ is lower semicontinuous. 
In particular, there exists
$s_n \in [-n , -(n-1)]$ such that
\begin{equation*}
\big| u(s_n) - \upbar{\rho} \big|_{L^2}^2 \;=\; 
\min_{-n \le t \le -(n-1)} \big| u(t) - \upbar{\rho}\big|_{L^2}^2 \;.
\end{equation*}
By applying Lemma \ref{s04} in the time interval $[s_{n+1},t]$ with
$t\in [-n,-(n-1)]$ we deduce  
\begin{equation*}
\sup_{-n \le t \le -(n-1)} \big| u(t) - \upbar{\rho} \big|_{L^2}^2
\;\le\; 
4 \, I_{[-(n+1),-(n-1)]} (u) \;+\; \Big( 1 + \frac
2\varepsilon \Big) \int_{-(n+1)}^{-(n-1)} 
\big| u(t) -  \upbar{\rho} \big|_{L^2}^2 \,dt \;,
\end{equation*}
where we bounded $|u(s_{n+1}) - \upbar{\rho}|_{L^2}^2$ by the time
integral of $|u(t) - \upbar{\rho}|_{L^2}^2$ over the interval 
$[-(n+1), - n]$. In view of the hypothesis $I(u)<\infty$ and
\eqref{f01}, the first statement of the lemma follows from the previous
estimate. 
The second one is a direct consequence of the first and Lemma \ref{s04}.
\end{proof}

Fix a path $u$ in $\mc U(\upbar{\rho}_\varepsilon)$ such that
$I(u) < \infty$.  It follows from the previous lemma that
there exists a sequence $t_n \to -\infty$ such that
\begin{equation}
\label{f05}
\lim_{n}  \big| u(t_n) - \upbar{\rho} \big|_{\mc H_0^1} \;=\; 0\;.
\end{equation}

\begin{lemma}
\label{s14}
Suppose that $\rho$ is a function in $M$ such that
\begin{equation*}
\big| \rho - \upbar{\rho} \big|_{\mc H_0^1}
\;\le\; \frac 12  \min \{\rho_0, 1 - \rho_1\}\;.
\end{equation*}
Then, there exists a constant $C_0$, depending only on $\rho_0$
and $\rho_1$, such that for any $T>0$
\begin{equation*}
\inf \, \big\{ I_{[-T,0]}(u) \,,\:
u(-T)=\upbar\rho,\,u(0)=\rho\big\} 
\;\le\; C_0 \Big(T + \frac 1T \Big)\, 
\Big( \varepsilon + \frac 1 \varepsilon \Big) \,
\big| \rho - \upbar{\rho} \big|_{\mc H_0^1}^2 \; .
\end{equation*}
\end{lemma}

\begin{proof}
We have to exhibit a path whose action can be estimated by the right
hand side of the inequality appearing in the statement of the lemma.
We claim that the straight one $u(t) = [1+(t/T)] \rho - (t/T)
\upbar{\rho}$, $t\in[-T,0]$, fulfills the requirements.

Since $\rho_0\le \upbar{\rho}(x) \le \rho_1$, it follows from the
assumption of the lemma and from the elementary estimate between the
$L^\infty$ norm and the $\mc H^1_0$ norm, namely, from the estimate 
$|h|^2_{L^\infty} \le \int_0^1 h_x^2 dx$, that $\rho_0/2 \le u \le
\rho_1 + (1-\rho_1)/2$. In particular, $\sigma(u) \ge c_0>0$ for some
constant $c_0$ which depends only on $\rho_0$ and $\rho_1$.

Recall the definition \eqref{e:2.6} of the action functional
$I_{[-T,0]}$. Fix a function $H$ in $C^\infty_{0}([-T,0]\times
[0,1])$, we need to estimate three terms to get a bound on
$I_{[-T,0]}(u)$. The first one is $\<\!\< u_t, H\>\!\>$.  Since
$u_t=(\rho-\upbar\rho)/T$, by using Schwarz inequality and $\<h,h\>
\le \<h_x,h_x\>$, we deduce
\begin{equation*}
\<\!\< u_t, H \>\!\>
\;\le\; \frac {\varepsilon c_0}2 \, \langle\< H_x,H_x\>\!\> 
\;+\; \frac 1 {2\varepsilon c_0 T} \, 
\big| \rho - \upbar{\rho} \big|_{\mc H_0^1}^2 \;.
\end{equation*}

The second and third terms are estimated as follows. Since
$f(\upbar{\rho})_x + \varepsilon \upbar{\rho}_{xx} =0$, 
\begin{equation*}
\<\!\< f(u)_x - \varepsilon u_{xx} \,,\, H\>\!\>
\; =\;  - \<\!\< f(u) -f(\upbar\rho) \,,\, H_x \>\!\>
\;+\; \epsilon \<\!\< u_x -\upbar\rho_x \,,\, H_x \>\!\>
\;.
\end{equation*}
Since $u-\upbar{\rho} = [1+(t/T)] [\rho - \upbar{\rho}]$, 
again by Schwarz inequality
\begin{equation*}
\<\!\< f(u)_x - \varepsilon u_{xx} \,,\, H\>\!\>
\;\le \; 
\frac {\varepsilon c_0}2 \, \<\!\< H_x ,H_x\>\!\>
\;+\; \frac {T}{c_0} \Big( \varepsilon + \frac 1 \varepsilon \Big)  
\, \big| \rho - \upbar{\rho} \big|_{\mc H_0^1}^2
\;,
\end{equation*}
where we used the fact that $f'$ is absolutely bounded by $1$ and
again the estimate $\<h,h\> \le \<h_x,h_x\>$. Since $\sigma(u) \ge
c_0$, the lemma follows from the previous bounds and \eqref{e:2.6}.
\end{proof}

We are now ready to prove the equivalence between the variational
problems \eqref{qp-1} and \eqref{qp}.

\begin{proof}[Proof of Theorem~\ref{t:S=V}: the identity 
  $\hat V_\epsilon = V_\epsilon$.]
  Any path $u$ in $C([-T,0];M)$ such that $u(-T)=\upbar{\rho}$ may be
  extended to a path in $\mc U (\upbar\rho)$ by setting $u(t) =
  \upbar{\rho}$, for $t\in (-\infty,-T)$. The inequality $\hat
  V_\varepsilon \le V_\varepsilon $ follows trivially.

  To prove the reverse inequality, fix $\rho$ in $M$ such that $\hat
  V_\varepsilon (\rho) < \infty$. Fix $0<\delta \le (1/2) \min
  \{\rho_0, 1 - \rho_1\}$ and let $u^\delta \in \mc U(\upbar{\rho})$
  be such that $u^\delta(0) = \rho$ and $I(u^\delta) \le \hat
  V_\varepsilon (\rho) + \delta$.  
  By \eqref{f05}, there exists $T>0$ such that
  \begin{equation*}
    \big| u^\delta (-T) - \upbar{\rho} \big|_{\mc H_0^1}  
    \; \le \; \delta \;.
\end{equation*}
By Lemma \ref{s14}, there exists a path $v$ in $C([-1,0], M)$
such that $v(-1)=\upbar{\rho}$, $v(0)=u^\delta (-T)$, and 
$I_{[-1,0]} (v) \le 2 C_0 (\varepsilon + \varepsilon^{-1})
\delta^2$. Consider the path $w$ in $C([-T-1,0], M)$ defined by $w(t) =
v(t+T)$, $t\in[-(T+1),-T]$, $w(t) = u^\delta(t)$, $t\in (-T,0]$. 
Clearly, $w(-T-1) = \upbar{\rho}$, $w(0) = \rho$ and
\begin{equation*}
I_{[-(T+1),0]}(w) \;=\; I_{[-1,0]}(v) \;+\;
I_{{[-T,0]}}(u^\delta) \;\le\; \hat V_\varepsilon (\rho) 
\;+\; 
\delta 
\;+\; 
2 \, C_0 \Big(\varepsilon + \frac 1\varepsilon \Big) \, \delta^2\; ,
\end{equation*}
which, by the arbitrariness of $\delta$, yields 
$V_\varepsilon(\rho) \le \hat V_\varepsilon(\rho)$ and concludes the
proof.  
\end{proof}

\subsection*{Upper bound for quasi-potential}

Recalling \eqref{Mo} and \eqref{Sig}, the first two lemmata below state that
if $\rho$ belongs to $M^o$ and $(\varphi,\rho)\in \Sigma$ then we can
construct a solution to the canonical equations \eqref{nhf} converging
to $(s'(\upbar{\rho}_\epsilon),\upbar{\rho}_\epsilon)$ as
$t\to-\infty$.

Fix a continuous function $\varphi : [0,1]\to [\varphi_0, \varphi_1]$
satisfying the boundary conditions $\varphi(0)=\varphi_0$,
$\varphi(1)=\varphi_1$, and consider the parabolic equation
\begin{equation}
\label{f41}
\begin{cases}
{\displaystyle 
\psi_t = \varepsilon \psi_{xx} \, + \, \frac{e^{\psi}-1}{e^{\psi}+1} 
\, \psi_x \, \big( 1- \varepsilon \psi_x \big) } \;, \\
\vphantom{\Big\{} \psi(t,0) = \varphi_0\;, \;\;
\psi(t,1) = \varphi_1\;, \\
\psi(0, \cdot) = \varphi (\cdot)\;.
\end{cases}
\end{equation}
A classical solution $\psi$ of \eqref{f41} is a function 
$\psi: [0,\infty)\times [0,1]\to\bb R$ such that 
$\psi$ is continuous in $[0,\infty)\times [0,1]$, 
$\psi_t$, $\psi_x$ and $\psi_{xx}$ are continuous
in $(0,\infty)\times (0,1)$, and $\psi$ satisfies the
identities \eqref{f41}. 
Let $F: [0,\infty)\times [0,1]\to [0,1]$ be a classical solution to
the viscous Burgers equation \eqref{eq:1}. A simple computation shows
that $\psi =s'(F)= \log[F/(1-F)]$ is a classical solution to
\eqref{f41}.

\begin{lemma}
\label{s10}
Fix an absolutely continuous function $\varphi : [0,1]\to
[\varphi_0,\varphi_1]$ satisfying the boundary conditions
$\varphi(0)=\varphi_0$, $\varphi(1)=\varphi_1$ and $0\le \epsilon
\varphi_x \le 1$.  There exists a unique classical solution $\psi$ to
\eqref{f41}.  Moreover, this solution satisfies $0<\varepsilon\psi_x
<1$ for $(t,x)\in (0,\infty)\times [0,1]$.  Finally, as $t\to+\infty$
the function $\psi(t)$ converges to $s'(\upbar\rho_\varepsilon)=
\log[\upbar\rho_\varepsilon / (1-\upbar\rho_\varepsilon)]$ in the
$C^3([0,1])$ topology uniformly in $\varphi$.
\end{lemma}

\begin{proof}
We start with existence. Set $F_0 := e^{\varphi}/ [1+e^{\varphi}]$. By
assumption, the function $F_0$ is continuous, bounded below by $\rho_0$,
bounded above by $\rho_1$ and satisfies $F_0(0)=\rho_0$, $F_0(1)=\rho_1$.
By Theorem 4.4 in Chapter 6 of \cite{lady}, there exists a unique
classical solution, denoted by $F=F(t,x)$, 
to the viscous Burgers \eqref{eq:1} with initial condition $F_0$. 
By maximum principle, $\rho_0 \le F\le \rho_1$. 
The function $\psi: [0,\infty) \times [0,1] \to \bb R$ given by $\psi
:= s'(F) = \log [ F / (1-F) ]$ is therefore well defined. An
elementary computation that we omit, relying on the special form of
the flux given by $f(u)=u(1-u)$, shows that
$\psi$ is a classical solution to \eqref{f41}.
Uniqueness follows from a similar argument. Indeed, if $\psi$ is a classical
solution to \eqref{f41}, then $F = e^{\psi}/ [1+e^{\psi}]$ is a classical
solution to \eqref{eq:1}. Hence, uniqueness of \eqref{f41} follows
from uniqueness of \eqref{eq:1}. 

A recursive argument, relying on differentiability properties of
solutions of linear parabolic equations, shows that $F(t)$, $F_t(t)$,
and therefore $\psi (t)$, $\psi_t (t)$, belong to $C^k([0,1])$ for any
$k\ge 1$. To prove that $0<\varepsilon\psi_x <1$ for $(t,x)\in
(0,\infty)\times[0,1]$, let $\gamma :=
\psi_x$ and observe that $\gamma$ solves the nonlinear equation with
mixed boundary conditions
\begin{equation*}
\begin{cases}
{\displaystyle 
\gamma_t = \varepsilon \gamma_{xx} + \, \big[ g
\, \gamma \, ( 1- \varepsilon \gamma ) \big]_x } \;, \\
\vphantom{\Big\{} \varepsilon \gamma_x (t,0) = - g (t,0)
\, \gamma (t,0) \, ( 1- \varepsilon \gamma (t,0) ) \;,  \\
\vphantom{\Big\{} \varepsilon \gamma_x (t,1) = - g (t,1)
\, \gamma (t,1) \, ( 1- \varepsilon \gamma (t,1) ) \;, \\ 
\gamma (0, \cdot) = \varphi_x (\cdot)\;,
\end{cases}
\end{equation*}
where $g = [e^{\psi}-1]/[e^{\psi}+1]$.  Since $g_x$ is bounded on
compact subsets of $(0,\infty) \times [0,1]$, and $\varphi_x$ is 
neither identically equal to $0$ nor to $\varepsilon^{-1}$, Theorem 3.7 in
\cite{pw} and the remark (ii) following it, imply the result.

Finally, by \cite[Thm.~4.9]{GK}, as $t\to\infty$, the function $F(t)$
converges to $\upbar\rho_\varepsilon$ in the $C^1([0,1])$ topology,
uniformly over $F_0:[0,1]\to [\rho_0, \rho_1]$. By the methods there
developed, it is straightforward to prove this statement in the
$C^3([0,1])$ topology.  
Since $F(t)$ converges to $\upbar\rho_\varepsilon$ in the $C^3([0,1])$
topology uniformly in $F_0$, $\psi (t)$ converges to
$\log[\upbar\rho_\varepsilon/ (1+\upbar\rho_\varepsilon)] =
\upbar\varphi_\varepsilon$ in the $C^3([0,1])$ topology, uniformly in
$\varphi$.
\end{proof}

Recalling \eqref{Mo}, fix $\rho$ in $M^o$ and $\varphi$ in $\ms P
(\rho)$.  Let $\psi$ be the solution to \eqref{f41} and define $v:
[0,\infty) \times [0,1] \to \bb R$ by
\begin{equation}
\label{r*}
v := \frac{1}{1+ e^{\psi}} -
\frac{\varepsilon \psi_{xx}}{\psi_x (1-\varepsilon \psi_x)}\;\cdot 
\end{equation}
Since $\varphi$ is smooth, $\psi(t)$ and $v(t)$ belong to $C^k([0,1])$
for all $t\ge 0$, $k\ge 1$.

\begin{lemma}
\label{s11} 
The function $v$ defined by \eqref{r*} is smooth, solves
\begin{equation}
\label{adhr*}
\begin{cases}
{\displaystyle
v_t - f(v)_x
= \varepsilon v_{xx} -  2 \varepsilon \big[ \sigma(v) \, \psi_x \big]_x } \\
{\displaystyle \vphantom{\Big\{}
v(t, 0) = \rho_0\;, \;\; v(t, 1) = \rho_1 } \\
{\displaystyle
v(0, \cdot) = \rho (\cdot)}
\end{cases}
\end{equation} 
and satisfies $0<v<1$, $(t,x)\in[0,\infty)\times[0,1]$. 
Moreover, as $t\to+\infty$ the function $v(t)$ converges to
$\upbar\rho_\varepsilon$ in the $C^1([0,1])$ topology, uniformly for
$\rho$ in $M^o$.
\end{lemma}

\begin{proof}
  By using the differential equation in \eqref{f41} and the identity
  $f(u)=\sigma(u)=u(1-u)$, a tedious computation shows that
  $v$ solves the differential equation in \eqref{adhr*}.  The boundary
  conditions in \eqref{adhr*} follow directly from \eqref{f41} and the
  definition of $v$.  The initial condition in \eqref{adhr*} holds, in
  view of \eqref{r*}, because $\varphi$ is a fixed point of $K_{\rho}$
  and therefore solves the Euler-Lagrange equation \eqref{Deq}.  Since
  the boundary conditions are bounded away from $0$ and $1$, and since
  $\psi_{xx}$ is bounded on compact subsets of $[0,\infty) \times
  [0,1]$, by Theorem 3.7 in \cite{pw} and the remark (ii) following
  it, $0<v<1$, $(t,x)\in[0,\infty)\times[0,1]$.

  By Lemma \ref{s10}, $\psi(t)$ converges in the $C^3([0,1])$ topology
  to $\log [\upbar\rho_\varepsilon / (1-\upbar\rho_\varepsilon)]$,
  uniformly in $\varphi$ . Therefore, by \eqref{r*}, $v(t)$ converges
  in the $C^1([0,1])$ topology to $\upbar{\rho}_\varepsilon$ uniformly
  for $\rho$ in $M^o$.
\end{proof}

\begin{lemma}
\label{t:dotG}
Fix a time interval $[T_1,T_2]$, a smooth path 
$u \in C\big([T_1,T_2]; M^o \big)$, 
and a smooth path $\varphi\in C \big([T_1,T_2];\mc F\big)$ 
with $0<\varepsilon \varphi_x < 1$.  
Then,
\begin{equation}
\label{i1f}
\begin{split}
& \mc G_\varepsilon( u(T_2),\varphi(T_2) )  -
\mc G_\varepsilon( u(T_1),\varphi(T_1) )  \\
& \qquad =\;  \int_{T_1}^{T_2}  
\Big\{ \langle s'(u) -\varphi,u_t \rangle \;-\;
\Big\langle \frac{\varepsilon \varphi_{xx}}
{\varphi_x (1-\varepsilon\varphi_x)}
- \frac 1{1+ e^{\varphi}}  + u , \varphi_t \Big\rangle\Big\}
\, dt \; .
\end{split}
\end{equation}
\end{lemma}

The proof relies on a simple computation and it is omitted.  We remark
that if the paths $u$ and $\varphi$ are chosen so that
$(\varphi(t),u(t))\in \Sigma$, $t\in[T_1,T_2]$, then the second term
on the right hand side of \eqref{i1f} vanishes.  Therefore, the
previous lemma provides an explicit expression for the integral of the
symplectic one-form along any path that lies in $\Sigma$, showing in
particular that the result depends only on the endpoints of the path,
namely that $\Sigma$ is Lagrangian.

\medskip
In view of Lemma \ref{s11}, the time reversal of the function $v$
defined in \eqref{r*} can be chosen as a test path for the variational
problem \eqref{qp}. In the next lemma we compute the action of such
a path.

\begin{lemma}
\label{s13}
Fix a function $\rho$ in $M^o$ and $\varphi$ in $\ms P (\rho)$. 
Let $v$ be the path \eqref{r*} and let 
$u \in \mc U(\upbar\rho_\epsilon)$ be defined by $u(t) := v(-t)$. 
Then, $I (u) = \mc G_\varepsilon (\rho, \varphi) \;-\;
S^o_\varepsilon(\upbar{\rho}_\varepsilon)$. 
\end{lemma}

\begin{proof}
By Lemma \ref{s11}, $u$ is a smooth path in $C((-\infty,0];M^o)$,
satisfies the final condition $u(0,\cdot) = \rho(\cdot)$,
and solves
\begin{equation*}
u_t +  f(u)_x = - \varepsilon u_{xx}
+ 2 \varepsilon \big[ \sigma(u) \psi^*_x \big]_x \;,
\end{equation*}
where $\psi^*(t) = \psi(-t)$, $t\le 0$. Equivalently, 
by setting $K = s'(u) - \psi^*$, $u$ solves
\begin{equation}
\label{f06}
u_t +  f(u)_x = \varepsilon u_{xx}
- 2 \varepsilon \big[ \sigma(u) K_x \big]_x \;.
\end{equation}
Fix $T>0$, recall that $u$ satisfies the boundary conditions
$u(t,0)=\rho_0$, $u(t,1)=\rho_1$, $t\in[-T,0]$, 
and observe that $K$ is smooth and
satisfies $K(t,0)=K(t,1)=0$, $t\in[-T,0]$.  
In particular, $K\in \mf H_0^1(\sigma(u))$ and 
therefore, by Theorem~\ref{s01},
\begin{equation*}
I_{[-T,0]}(u) \;=\; \varepsilon   \| K\|_{1,\sigma(u)}^2 
 \;=\;
\varepsilon  
\langle\!\langle K_x, \sigma(u) K_x\>\!\>
\end{equation*}
Recall the definition \eqref{He} of the Hamiltonian $\bb H$.
Multiplying both sides of \eqref{f06} by $K$ and integrating, we get
that
\begin{equation*}
\<K, u_t\> \;-\; \bb H (u,K) \;=\; \varepsilon  \<K_x, \sigma(u) K_x \>
\end{equation*}
so that
\begin{equation*}
I_{[-T,0]}(u) \;=\;  \int_{-T}^0 \Big\{
\< s'(u) - \psi^*, u_t \> \;-\; \bb H \big(u ,s'(u) - \psi^*\big) 
\Big\} \, dt\;.
\end{equation*}
Since $v$ is defined by \eqref{r*}, $u$ is given by the same equation
with $\psi^*$ replacing $\psi$. In particular, by item (iii) 
in Proposition \ref{t:fp}, 
$\bb H \big(u,s'(u) - \psi^*\big)$ vanishes.
Hence, by Lemma~\ref{t:dotG},
\begin{equation}
\label{f25}
I_{[-T,0]}(u) \;=\; \int_{-T}^0 
\<s'(u) - \psi^*, u_t \> \, dt \;=\; \mc G_\varepsilon (\rho, \varphi)
\;-\; \mc G_\varepsilon (v(T), \psi(T)) \;.
\end{equation}
By Lemma \ref{s10}, $\psi(T)$ converges to
$s'(\upbar{\rho}_\varepsilon)$ in the $C^3([0,1])$ topology and by
Lemma \ref{s11}, $v(T)$ converges to $\upbar{\rho}_\varepsilon$ in the
$C^1([0,1])$ topology. Therefore, $\mc G_\varepsilon (v(T), \psi(T))$
converges to $\mc G_\varepsilon (\upbar{\rho}_\varepsilon,
s'(\upbar{\rho}_\varepsilon))$ which is equal to
$S^o_\varepsilon(\upbar{\rho}_\varepsilon)$ by \eqref{f09}. In
conclusion,
\begin{equation*}
I (u) \;=\; \lim_{T\to\infty} I_{[-T,0]}(u) 
\;=\; \mc G_\varepsilon (\rho, \varphi) \;-\; 
S^o_\varepsilon(\upbar{\rho}_\varepsilon)\; ,
\end{equation*}
which proves the lemma.
\end{proof}

It follows from the previous lemma that if $\varphi$ is chosen in
${\ms F}(\rho)$, namely $\varphi$ is a minimizer for $\mc
G_\epsilon(\rho,\cdot)$, then $I(u) = S^o_\epsilon(\rho) -
S^o_\varepsilon(\upbar{\rho}_\varepsilon) = S_\varepsilon (\rho)$.  In
particular, this proves the inequality $\hat V_\varepsilon (\rho)\le
S_\varepsilon (\rho)$ for smooth functions $\rho$ in $M^o$,
By a density argument, we next show this bound holds for any $\rho\in M$.

\begin{proof}[Proof of Theorem~\ref{t:S=V}: the bound $\hat V_\epsilon
  \le S_\epsilon$.]
Fix $\rho\in M$ and denote by $\varphi\in {\ms F}(\rho)$,
$\{\rho^n\}\subset M^o$, and $\{\varphi^n \}\subset \mc F$ the
function and the sequences provided by Lemma \ref{s16}. Denote by
$\psi$ the classical solution to \eqref{f41}, and define $u
:(-\infty,0]\times [0,1]\to [0,1]$ by
\begin{equation*}
u(t) =
\begin{cases}
\rho & \textrm{if $t=0$}\;, \\
v(-t) & \textrm{if $t<0$}\;,
\end{cases}
\end{equation*}
where $v$ is the function defined in \eqref{r*}.  Since $\varphi \in
C^1([0,1])$, $\psi(t) \to \varphi$ in the $C^1$ topology as
$t\downarrow 0$. Since by Lemma \ref{t:int} we have $0< \epsilon \varphi_x <
1$, it is simple to check that the function $u(t)$ converges to $\rho$
in $M$ as $t\uparrow 0$.  Hence, by the convergence of $\psi(t)$ as
$t\to +\infty$ stated in Lemma \ref{s10}, the path $u$ belongs to the
set $\mc U (\upbar\rho_\epsilon)$ introduced in \eqref{Ue}.

Let $\psi^n$ be the solution to \eqref{f41} with $\varphi$ replaced by
$\varphi^n$ and $v^n$ be as defined in \eqref{r*} with $\psi$ replaced
by $\psi^n$. Let finally $u^n: (-\infty,0]\times [0,1] \to [0,1]$ be
defined by $u^n(t)=v^n(-t)$.  In view of the continuity with respect
to the initial condition of the solution to the viscous Burgers
equation \eqref{eq:1} and the uniformity of the convergence as $t\to +
\infty$ stated in Lemma \ref{s11}, the sequence $\{u^n\}$ converges to
$u$ in $\mc U(\upbar\rho_\epsilon)$.  The lower semicontinuity of the
functional $I: \mc U(\upbar\rho_\epsilon) \to [0,+\infty]$, together
with Lemmata \ref{s13} and \ref{s16}, now imply
\begin{equation*}
I(u) \le  \liminf_n I(u^n) =  
\liminf_n \big[ \mc G_\epsilon (\rho^n,\varphi^n)  
- S^o_\epsilon (\upbar\rho_\epsilon)  \big] 
=   \mc G_\epsilon (\rho,\varphi) 
- S^o_\epsilon (\upbar\rho_\epsilon)
=   S_\epsilon(\rho)\;,
\end{equation*}
whence $\hat V_\epsilon(\rho)\le S_\epsilon(\rho)$ and the proof is
concluded.
\end{proof}

\subsection*{Lower bound for the quasi-potential}

Before carrying out the details, we explain the main idea and the novel
difficulty here encountered. 
Fix $\rho\in M$, $T>0$ and a path $u\in C([-T,0];M)$ such that
$u(-T)=\upbar\rho_\epsilon$ and $u(0)=\rho$. We need to show
$I_{[-T,0]}(u)\ge S_\epsilon(\rho)$. Assume that the path $u$ is
smooth, bounded away from zero and one, and satisfies the boundary
conditions $\rho_0$, $\rho_1$ at the endpoints of $[0,1]$.  
By the variational definition \eqref{e:2.6} 
of the action functional and the definition \eqref{He} of the
Hamiltonian $\bb H$, for each smooth function 
$H : [-T,0]\times [0,1]\to \bb R$ vanishing at the boundary, we have 
\begin{equation}
\label{f04}
\begin{split}
I_{[-T,0]}(u) & \; \ge \; 
\<\!\< u_t +f(u)_x -\epsilon u_{xx}, H\>\!\> 
- \epsilon \<\!\< H_x,\sigma(u) H_x\>\!\>
\\
& \;= \; \int_{-T}^0 \big[ \langle H, u_t\rangle  - \bb H (u,H) \big]
\, dt \;. 
\end{split}
\end{equation}
Assume now that for each $t\in [-T,0]$ there exists a unique solution
$\varphi(t)$ to the Euler-Lagrange equation \eqref{Deq} with $\rho$
replaced by $u(t)$. Assume furthermore that $\varphi(t)$ is smooth; in
this case we may choose above $H =s'(u ) -\varphi $. In view of item
(iii) in Proposition \ref{t:fp} and Lemma \ref{t:dotG} we then
conclude $I_{[-T,0]}(u) \ge S_\epsilon(\rho)$.  If for each $\rho\in
M$ the functional $\mc G(\rho,\cdot)$ has a unique critical point, it
is not difficult to turn the previous argument into a proof
\cite{bdgjl3,bgl1}.  On the other hand, in the case here discussed we
have to face the lack of uniqueness of \eqref{Deq}: if we choose the
``wrong'' $\varphi$, the bound we would get by the previous argument
will not be sharp. We shall overcome this problem by discretizing the
time interval $[-T,0]$ and choosing a piecewise constant path
$\varphi(t)$.  As we show, the previous argument gives then the sharp
bound provided we choose the optimal $\varphi$ at the endpoint of each
time step.

We start by recalling the following density result, which is proven 
in Theorem 5.1 of \cite{blm1}.

\begin{lemma}
\label{dId} 
Fix $T>0$ and a path $u$ in $C([-T,0]; M)$ such that $u(-T) =
\upbar\rho_\varepsilon$ and $I_{[-T,0]}(u) < \infty$.  There exists a
sequence $\{u^n\}\subset C([-T,0];M)$ of smooth functions $u^n :[-T,0]
\times [0,1] \to (0,1)$ converging to $u$ in $C([-T,0]; M)$ such that
$u^n (-T,\cdot) =\upbar\rho_\varepsilon(\cdot)$, $u^n(t,0) = \rho_0$,
$u^n(t, 1) = \rho_1$, $t\in [-T,0]$, and $I_{[-T,0]}(u^n)$ converges
to $I_{[-T,0]}(u)$.
\end{lemma}

\begin{proof}[Proof of Theorem~\ref{t:S=V}: the bound $S_\epsilon 
  \le V_\epsilon$.]
Fix a function $\rho$ in $M$, $T>0$ and a path $u$ in $C([-T,0]; M)$
such that $u(-T) = \upbar\rho_\varepsilon$, $u(0)=\rho$. We need to
show that $I_{[-T,0]}(u)\ge S_\varepsilon(\rho)$.

Assume firstly that $\rho$ belongs to $M^o$ and that $u:[-T,0]\times
[0,1]\to [0,1]$ is a smooth path bounded away from zero and one and
satisfies $u(t,0)=\rho_0$, $u(t,1)=\rho_1$, $t\in [-T,0]$. In this
case, as we have seen above, \eqref{f04} holds for any $H$ in $\mf
H_0^1(\sigma(u))$.  Consider a partition $[-T,0) = \bigcup_{k=1}^n
[-T_{k},-T_{k-1})$ with $T_0=0$ and $T_{n}=T$.  For $k=1,\cdots,n$,
choose $\varphi^k \in \ms F(u(-T_k))$ namely, $\varphi^k$ is a
minimizer for $\mc G_\epsilon( u(-T_k) ,\cdot)$.  In view of Lemma
\ref{s02}, $\varphi^k$ belongs also to $\ms P (u(-T_k))$.  Define the
path $\varphi$ piecewise constant in $[-T,0]$ by $\varphi (t)
=\varphi^{k}$ for $t\in [-T_{k},-T_{k-1})$, $k=1,\cdots,n$ and
$\varphi(0) = \varphi^1$.

Since the path $u$ is smooth, by the definition of $\varphi$,
$\varphi(t,0)=\varphi_0$, $\varphi(t,1)=\varphi_1$, $t\in[-T,0]$, and
$s'(u)-\varphi$ is a smooth function in space, piecewise smooth in
time which vanishes at the endpoints of $[0,1]$. In particular,
$s'(u)-\varphi$ belongs to $\mf H_0^1(\sigma(u))$.  By choosing
$H=s'(u)-\varphi$ in \eqref{f04} we then get
\begin{equation*}
I_{[-T,0]}(u) \; \ge \;
\sum_{k=1}^n \int_{-T_{k}}^{-T_{k-1}}
\Big[ \langle s'( u ) - \varphi^k \,,\, u_t  \rangle  
- \bb H \big( u ,  s'(u ) - \varphi^k\big)
\Big] \, dt\; .
\end{equation*}
By the choice of $\varphi^k$ and item (iii) in Proposition \ref{t:fp},
$\bb H \big(u(t) , s'(u(t)) - \varphi^k\big)$ vanishes for $t=-T_{k}$.
A simple computation, based on the bounds stated in item (ii) of
Proposition \ref{t:fp}, shows that the map $[-T_k,-T_{k-1}) \ni t
\mapsto \bb H \big( u(t) , s'(u(t)) - \varphi^k\big)$ is Lipschitz
with a Lipschitz constant depending on $u$ but independent of
$\varphi^k$. Hence,
\begin{equation*}
\lim_{n\to\infty} \sum_{k=1}^n \int_{-T_{k}}^{-T_{k-1}}
\bb H \big( u ,  s'(u) - \varphi^k\big)\, dt\; =\; 0
\end{equation*}
provided the mesh of the partition vanishes as $n\to\infty$.

On the other hand, since $\varphi(t)$ is constant in the interval
$[-T_k, -T_{k-1})$, Lemma~\ref{t:dotG} yields
\begin{equation*}
\begin{split}
\sum_{k=1}^n \int_{-T_{k}}^{-T_{k-1}} \langle s'( u(t) ) -
\varphi^k , u_t \rangle \, dt \;=\; \sum_{k=1}^n \Big[ \mc
G_\varepsilon (u(-T_{k-1}), \varphi^k) - \mc G_\varepsilon
(u({-T_k}), \varphi^k) \Big] \;.
\end{split}
\end{equation*}
Since $\varphi^{k-1} \in \ms F(u({-T_{k-1}}))$, it follows $\mc
G_\varepsilon (u({-T_{k-1}}), \varphi^k) \ge \mc G_\varepsilon
(u(-T_{k-1}), \varphi^{k-1})$.  The previous expression is thus
bounded below by the telescopic sum
\begin{equation*}
\sum_{k=1}^n \Big[ \mc G_\varepsilon (u(-T_{k-1}), \varphi^{k-1}) 
-  \mc G_\varepsilon (u({-T_k}), \varphi^k) \Big]
\;=\;  \mc G_\varepsilon (u(0), \varphi^1) - 
\mc G_\varepsilon (u({-T}), \varphi^n) \;.
\end{equation*}
Since $u({-T}) = \bar\rho_\varepsilon$, by the choice of $\varphi^n$,
we have $\mc G_\varepsilon (\bar\rho_\varepsilon, \varphi^n) =
\inf_\psi \mc G_\varepsilon (\bar\rho_\varepsilon, \psi) =
S^o_\varepsilon(\bar\rho_\varepsilon)$.  On the other hand, since
$u(0) = \rho$, we have $\mc G_\varepsilon (u(0), \varphi^1) = \mc
G_\varepsilon (\rho, \varphi^1) \ge S^o_\epsilon(\rho)$.  By taking
the limit $n\to\infty$, the previous bounds imply
\begin{equation*}
I_{[-T,0]}(u) \;\ge\;
S^o_\varepsilon(\rho)  \;-\; S^o_\varepsilon(\upbar{\rho}_\varepsilon) 
\;=\; S_\varepsilon ( \rho )\;.
\end{equation*}

Let now $\rho\in M$ be arbitrary and consider an arbitrary path $u \in
C([-T,0];M)$ such that $u(0)=\rho$ and $u(-T) =
\upbar{\rho}_\varepsilon$.  Since we can assume $I_{[-T,0]}(u)
<\infty$, by Lemma~\ref{dId}, there exists a sequence of smooth paths
$u^n$ bounded away from zero and one which converges to $u$ in
$C([-T,0];M)$ and such that $\lim_{n} I_{[-T,0]}(u^n) =
I_{[-T,0]}(u)$. The lower semicontinuity of $S_\varepsilon$ on $M$,
see Theorem \ref{t:1-1}, and the result for smooth paths yield
\begin{equation*}
I_{[-T,0]}(u) \;=\; \lim_{n} I_{[-T,0]}(u^n)  \;\ge\;
\liminf_{n} S_\varepsilon ( u^n(0) ) 
\;\ge\; S_\varepsilon ( u(0) )  \;=\;  S_\varepsilon (\rho ) \;,
\end{equation*}
which concludes the proof.
\end{proof}

\section{The inviscid limit}
\label{sec6}

In this section we discuss the inviscid limit $\epsilon\downarrow 0$.
We first discuss the variational convergence of the functional $\mc
G_\epsilon$ to $\mc G$.  By analyzing the limiting variational problem
\eqref{s-in} and using a perturbation argument we then show, provided
$\epsilon$ is small enough, that there exists $\rho\in M$ such that $\mc
G_\epsilon(\rho,\cdot)$ admits at least two minimizers.

\subsection*{Variational convergence of $\mc G_\epsilon$}

By standard properties of $\Gamma$--convergence, Theorems \ref{t:gconv} 
and \ref{t:convmin} are corollaries of the following result. 

\begin{theorem}
\label{t:gcg}
Let $\mc G_\epsilon : M\times \mc F\to (-\infty,+\infty]$ be the
functional defined in \eqref{Grf}. As $\epsilon\downarrow 0$, the
family $\{\mc G_\epsilon\}_{\epsilon>0}$ $\Gamma$--converges to the
functional $\mc G$ defined in \eqref{Grf-in}.
\end{theorem}

\begin{proof}
We start by showing the $\Gliminf$ inequality.  Fix a sequence
$\epsilon_n\downarrow 0$, $(\rho,\varphi)\in M\times \mc F$, and a
sequence $\{(\rho^n,\varphi^n)\} \subset M\times \mc F$ converging to
$(\rho,\varphi)$. We need to show $\liminf_n \mc
G_{\epsilon_n}(\rho^n, \varphi^n)\ge \mc G (\rho,\varphi)$.  The
convexity of the real function $s$ trivially implies $\liminf_{n}
\int_{0}^{1} s(\rho^n) dx \ge \int_{0}^{1} s(\rho) \, dx$.  Likewise,
by Jensen inequality, $\int_{0}^{1} s(\varepsilon_n \varphi_x^n) \,dx
\ge s( \varepsilon_n [\varphi_1-\varphi_0])$ which vanishes as
$\varepsilon_n\downarrow 0$.  To conclude the proof, it remains to
observe that the convergence of $\rho^n$ to $\rho$ in $M$ and the one
of $\varphi^n$ to $\varphi$ in $\mc F$ implies the convergence of
$\int_{0}^{1} \big[ (1-\rho^n)\, \varphi^n - \log (1+e^{\varphi^n})
\big] \, dx$.

We next show the $\Glimsup$ inequality.  Fix a sequence
$\epsilon_n\downarrow 0$ and $(\rho,\varphi)\in M\times \mc F$. We
need to exhibit a sequence $\{(\rho^n,\varphi^n)\} \subset M\times \mc
F$ converging to $(\rho,\varphi)$ such that $\limsup_n \mc
G_{\epsilon_n}(\rho^n,\varphi^n) \le \mc G (\rho,\varphi)$.  Consider
first the case in which $\varphi\in \mc F$ is smooth, say $C^1$, and
satisfies $\varphi(0)=\varphi_0$, $\varphi(1)=\varphi_1$.  We then
claim a \emph{recovering sequence} is simply given by the constant
sequence $(\rho^n,\varphi^n)=(\rho,\varphi)$. Indeed, by the
smoothness of $\varphi$, for such a sequence we have $\int_{0}^{1}
s(\varepsilon_n \varphi_x^n) \,dx \to 0$.  The proof is now completed
by a density argument.  More precisely, it is enough to observe that,
given $\varphi\in\mc F$, there exists a sequence $\{\varphi^k\}\subset
\mc F$ converging to $\varphi$ such that $\varphi^k\in C^1([0,1])$,
$\varphi^k(0)=\varphi_0$, $\varphi^k(1)=\varphi_1$, and $\lim_k \mc
G(\rho,\varphi^k)=\mc G(\rho,\varphi)$.
\end{proof}

Next result is a straightforward consequence of the the previous proof
because the sequence used in the $\Glimsup$ inequality is constant in
$\rho$.

\begin{corollary}
\label{s18}
For every $\rho$ in $M$, $\mc G_\epsilon(\rho, \cdot)$
$\Gamma$--converges to $\mc G(\rho, \cdot)$ as $\epsilon\downarrow 0$.
\end{corollary}

\begin{proof}[Proof of Theorem~\ref{t:convmin}]
  Since $\mc F$ is compact, Theorem \ref{t:convmin} follows from
  Theorem \ref{t:gcg} and \cite[Thm.~1.21]{Braides}.
\end{proof}

To deduce Theorem~\ref{t:gconv}, we only need to ``project''
Theorem~\ref{t:gcg} to the first variable.  For completeness, we
detail below the proof.

\begin{proof}[Proof of Theorem~\ref{t:gconv}]
We first show that, as $\epsilon\downarrow 0$, the family of
functionals $\{S^o_\epsilon\}_{\epsilon>0}$ $\Gamma$--converges to
$S^o$. We start by proving the $\Gliminf$ inequality.  Fix a sequence
$\epsilon_n\downarrow 0$, $\rho\in M$ and a sequence
$\{\rho^n\}\subset M$ converging to $\rho$.  In view of the
compactness of $\mc F$ and the lower semicontinuity of $\mc
G_{\epsilon_n}(\rho^n,\,\cdot\,)$, there exists a sequence
$\{\varphi^n\}\subset \mc F$ such that $S^o_{\epsilon_n} (\rho^n) =\mc
G_{\epsilon_n}(\rho^n,\varphi^n)$.  Again by the compactness of $\mc
F$, by taking if necessary a subsequence, there exists $\varphi\in\mc
F$ such that $\varphi^n\to\varphi$. By Theorem~\ref{t:gcg} we then
deduce
\begin{equation*}
\liminf_n S^o_{\epsilon_n} (\rho^n) =
\liminf_n \mc G_{\epsilon_n}(\rho^n,\varphi^n)
\ge \mc G (\rho,\varphi) \ge S^o(\rho) \; . 
\end{equation*}
We next prove the $\Glimsup$ inequality.  Fix a sequence
$\epsilon_n\downarrow 0$, $\rho\in M$ and choose the constant sequence
$\rho^n=\rho$.  By Corollary~\ref{s18}, the compactness of $\mc F$,
and \cite[Theorem.~1.21]{Braides}
\begin{equation*}
\lim_n S^o_{\epsilon_n} (\rho) =
\lim_n \: \inf_{\varphi} \: \mc G_{\epsilon_n}(\rho,\varphi)
= \inf_{\varphi} \: \mc G (\rho,\varphi)
= S^o(\rho) \; . 
\end{equation*}

To complete the proof we need to show that
$S^o_\epsilon(\upbar{\rho}_\varepsilon) \to S^o(\upbar{\rho})$ as
$\epsilon\downarrow 0$. While this statement can be proven by using
\eqref{f09} and the explicit expression of $\mc G_\epsilon$ and $\mc
G$, we next give an argument again based on Theorem~\ref{t:gcg} and
the fact that $\upbar{\rho}_\epsilon$ converges to $\upbar{\rho}$ in
$L^1([0,1])$, pointed out just before \eqref{Grf-in}.  In the case
$\rho_0+\rho_1=1$, by the latter statement we mean that
$\upbar{\rho}_\epsilon$ converges to the stationary entropic solution
with a shock placed at $x=1/2$.  Since $\upbar{\rho}_\varepsilon\to
\upbar{\rho}$ in $M$, the $\Gliminf$ inequality proven above yields
$\liminf_\epsilon S^o_\epsilon(\upbar{\rho}_\varepsilon)\ge
S^o(\upbar{\rho})$. To prove the other inequality, fix a smooth
$\varphi\in\mc F$ so that $\int_0^1s(\epsilon\varphi_x) \, dx$
vanishes as $\epsilon\downarrow 0$.  Since $\upbar{\rho}_\epsilon$
converges to $\upbar{\rho}$ in $L^1([0,1])$, $\int_0^1
s(\upbar{\rho}_\epsilon)\, dx$ converges to $\int_0^1
s(\upbar{\rho})\, dx$. Therefore
\begin{equation*}
\limsup_\epsilon S^o_\epsilon(\upbar{\rho}_\varepsilon)
\le  \limsup_\epsilon
\mc G_\epsilon(\upbar{\rho}_\varepsilon,\varphi) 
=  \mc G(\upbar{\rho},\varphi) \;.     
\end{equation*}
By optimizing on $\varphi$ we then deduce $\limsup_\epsilon
S^o_\epsilon(\upbar{\rho}_\varepsilon)\le S^o(\upbar{\rho})$, which
concludes the proof.
\end{proof}

\medskip The variational problem appearing in \eqref{s-in} is much
simpler than the one in \eqref{ssg}. In fact, the former can be
reduced to a one-dimensional minimum problem. More precisely, we can
restrict the infimum in \eqref{s-in} to step functions $\varphi$.
Denote by $\tilde{\mc{G}} (\rho,y)$ the functional ${\mc{G}}
(\rho,\varphi)$ defined in \eqref{Grf-in} evaluated at
$\varphi=\varphi^{(y)}$, $y\in[0,1]$, where $\varphi^{(y)}(x):=
\varphi_0 \mb 1_{[0,y)}(x) + \varphi_1 \mb 1_{[y,1]}(x)$.  In other
words, let $\tilde{\mc G} : M \times [0,1]\to \bb R$ be the functional
defined by
\begin{equation*}
\begin{split}
\tilde{\mc G} (\rho,y) \; :=\; 
& \int_{0}^1 s(\rho) \, dx \; +\;
\varphi_0 \int_{0}^y  (1-\rho)\, dx
\;+\; \varphi_1 \int_{y}^1 (1-\rho)\, dx \\
& \; -\; y  \log\big(1+e^{\varphi_0}\big)
\;-\; (1-y) \log\big(1+e^{\varphi_1}\big)\;.
\end{split}
\end{equation*}

\begin{proposition}
\label{t:gin=tgin}
Fix $\rho\in M$. Then, 
\begin{equation}
\label{dav}
\inf\,  \big\{ \mc G (\rho,\varphi) \,,\: \varphi \in \mc F \big\}
\;=\;
\inf\, \big\{ \tilde{\mc G} (\rho,y)\,,\: y\in[0,1] \big\}  \;.
\end{equation}
\end{proposition}

\begin{proof}
Since the functional $\mc G(\rho,\cdot)$ is concave and $\mc F$ is a
compact convex set, the infimum on the left hand side of \eqref{dav}
is achieved when $\varphi$ belongs to the extremal elements of $\mc
F$. We thus have to show that the extremal elements of $\mc F$ are
$\{\varphi^{(y)},\,y\in[0,1]\}$. This is easily proven recalling
\eqref{Fa} and noticing that, given $m>0$, the extremal elements of
$\mc P_m([0,1])$ are $\{m \,\delta_y,\, y\in[0,1]\}$, where $\delta_y$
is the Dirac measure at $y$.
\end{proof}
\subsection*{Non uniqueness of the minimizers}  

The statements of Theorems~\ref{t:1-1} and \ref{t:degmine} imply that
for $\epsilon$ small enough there exists $\rho\in M$ such that $\ms
P_\epsilon(\rho)$ is a not a singleton. This result can however be
directly proven by a simpler argument, presented below, which also
shows that there exists points in $\ms P_\epsilon(\rho)$ which are not
in $\ms F_\epsilon(\rho)$.
  
\begin{lemma}
\label{s07}
Fix $\varphi_0<\varphi_1$. There exists
$\varepsilon^*\in(0,\epsilon_0)$ and a smooth function $\rho$ in $M$
such that $K_{\rho,\varepsilon}$ has at least two fixed points
for $\varepsilon<\varepsilon^*$.
\end{lemma}

\begin{proof}
Let $\varphi\in \mc F$ be the affine function $\varphi (x) = \varphi_0
\,(1-x)+ \varphi_1\, x$, and let $\rho\in M$ be given by $\rho =
1/(1+e^{\varphi})$.  Clearly, $\varphi$ is a fixed point of
$K_{\rho,\varepsilon}$.  To conclude the proof, it is enough to show
that, provided $\epsilon$ is small enough, $\varphi$ is not a
minimizer of $\mc G_\varepsilon (\rho, \cdot)$.  Recall the second
variation of $\mc G_\varepsilon (\rho, \cdot)$ evaluated at $\varphi$
is the quadratic form \eqref{f08}.  For the above choices of $\rho$
and $\varphi$, it is easy to check that this quadratic form is not
positive semi-definite for $\varepsilon$ small enough, which proves
the lemma.
\end{proof}

In order to prove Theorem~\ref{t:degmine}, we first analyze the
limiting variational problem \eqref{s-in} and characterize, for
suitable functions $\rho$, its minimizers.  Fix $\varphi_0 <
\varphi_1$, set $\upbar{\varphi}:= (\varphi_0+\varphi_1)/2$, and let
\begin{equation*}
\begin{split}    
& A:= 1 - \frac{ \log (1+e^{\varphi_1}) -\log (1+e^{\varphi_0})}
     {\varphi_1-\varphi_0}\;, \quad
A_+:= 1 - \frac{\log (1+e^{\upbar{\varphi}}) -\log (1+e^{\varphi_0})}
     {\upbar{\varphi}-\varphi_0}\;, \\ 
& \qquad\qquad\qquad\qquad
A_-:= 1 - \frac{\log (1+e^{\varphi_1})-\log (1+e^{\upbar{\varphi}})}
     {\varphi_1-\upbar{\varphi}}\;\cdot
\end{split}
\end{equation*}
Since the real function $x\mapsto \log(1+e^x)$ is strictly increasing,
strictly convex, and Lipschitz with Lipschitz constant one, we have
$0<A_-<A<A_+ <1$.  Fix a continuous function $\rho:[0,1]\to [0,1]$
satisfying the following condition. 
There exist three reals $0\le y_-<y_0<y_+\le 1$ such that:
$\rho(y_0)=\rho(y_\pm)=A$, $\rho(x)< A$ for $x\in
[0,y_-)\cup(y_0,y_+)$, $\rho(x)> A$ for $x\in(y_-,y_0)\cup(y_+,1)$,
and $ A_-< \rho(x)<A_+$ for $x\in [y_-,y_+]$.  Recalling $\mc
P_m([0,1])$, $m>0$, denotes the set of positive Borel measure on
$[0,1]$ with mass $m$ endowed with the topology of weak convergence,
let
\begin{equation*}
\begin{split}    
& \mc P_m^-([0,1]):= \big\{ \mu\in \mc P_m([0,1])\,:\; \mu([0,y_0])
\ge m/2 \big\} \;, \\
&\quad
\mc P_m^+([0,1]):= \big\{ \mu\in \mc P_m([0,1])\,:\; \mu([y_0,1])
\ge m/2 \big\}\; .
\end{split}
\end{equation*}
Note that $\mc P_m^\pm([0,1])$ is a closed convex subset of $\mc
P_m([0,1])$. Recalling \eqref{Fa}, let accordingly 
\begin{equation*}
\mc F^\pm := \big\{ \varphi\in \mc F \,:\: \varphi(x) = \varphi_0 + \mu ([0,x]) 
\textrm{ for some } \mu \in \mc P_{\varphi_1-\varphi_0}^\pm([0,1])
\big\}
\end{equation*}
and observe that $\mc F= \mc F^-\cup \mc F^+$.

\begin{lemma}
\label{t:mina}
Fix $\rho$ as above and set $\varphi^*_{\pm}:= \varphi_0 \mb
1_{[0,y_\pm)} + \varphi_1 \mb 1_{[y_\pm,1]}$. Then
\begin{equation*}
\mathrm{arg \,inf} \big\{ \mc G(\rho,\varphi)\,,\: \varphi\in \mc
  F^\pm\big\} = \big\{ \varphi^*_\pm \big\}
\end{equation*}
namely, the infimum of $\mc G(\rho,\cdot)$ over $\mc F^\pm$ is uniquely
achieved at $\varphi^*_\pm$.
If furthermore $\rho$ satisfies $\int_{y_-}^{y_+}\rho\,dx = A
(y_+-y_-)$ then $\inf_{\mc F^-} \mc G(\rho,\cdot) =\inf_{\mc F^+} \mc
G(\rho,\cdot)$.
\end{lemma}

This result implies that if the function $\rho$ satisfies all
conditions of the previous lemma then
\begin{equation*}
  \mathrm{arg \,inf} \big\{ \mc G(\rho,\varphi)\,,\: \varphi\in \mc
  F \big\} = \big\{\varphi^*_-,\varphi^*_+\big\}
\end{equation*}
namely, the functional $\mc G(\rho,\cdot)$ on $\mc F$ has exactly two
minimizers.  In fact, by using Proposition~\ref{t:gin=tgin}, this can
be easily proven even if the condition $A_-< \rho(x)<A_+$, $x\in
[y_-,y_+]$, is dropped. We also mention that, as observed in
\cite{dls3}, if $\rho$ is constant and equal to $A$, then there exists
a one parameter family of minimizers for $\mc G(\rho,\cdot)$ which is
exactly the collection of the extremal elements
$\{\varphi^{(y)},\,y\in[0,1]\}$ of $\mc F$.

\begin{proof}[Proof of Lemma \ref{t:mina}.]
Since $\mc G (\rho,\cdot)$ is a strictly concave functional and $\mc
F^\pm$ is a compact convex set, the infimum can only be achieved at
the extremal elements of $\mc F^\pm$. These elements can be easily
characterized. Indeed, it is simple to check that the extremal points
of $\mc P^-_m([0,1])$ are given by $E^-_1\cup E^-_2$ where
\begin{equation*}
E^-_1 = \big\{ m \, \delta_x \,,\: x \in [0,y_0] \big\}
\,,\qquad
E^-_2 = \Big\{ \frac m2 \big( \delta_{x} + \delta_{x'} \big)\,,
\: x\in [0,y_0]\,,\,x'\in [y_0,1]\Big\} 
\end{equation*}
We denote by $\mc F^-_{\mathrm e,1}$ and $\mc F^-_{\mathrm e,2}$ the
corresponding subsets of $\mc F^-$ with $m=\varphi_1-\varphi_0$. Note
that $\varphi \in \mc F^-_{\mathrm e,1}$ iff $\varphi$ jumps from
$\varphi_0$ to $\varphi_1$ at some point $x\in [0,y_0]$ while $\varphi
\in \mc F^-_{\mathrm e,2}$ iff $\varphi$ jumps from $\varphi_0$ to
$\upbar{\varphi}$ at some point $x\in [0,y_0]$ and from
$\upbar{\varphi}$ to $\varphi_1$ at some point $x'\in[y_0,1]$.

We have thus reduced the original problem of the minimum over $\mc
F^-$ to a minimum problem in one and two real variables.  By using
that $\rho(x)< A_+$ for $x\in [0,y_0]$ and $\rho(x)> A_-$ for $x\in
[y_0,1]$, elementary computations show that the infimum of $\mc
G(\rho,\cdot)$ over $\mc F^-_{\mathrm e,2}$ is achieved when $\varphi$
has a single jump at $y_0$. Note that such a $\varphi$ belongs both to
$\mc F^-_{\mathrm e,2}$ and $\mc F^-_{\mathrm e,1}$.  Likewise, by
using that $\rho(x) < A$ for $x\in [0,y_-)$ and $\rho(x)>A$ for $x\in
(y_-,y_0)$, it is readily seen that the infimum over $\mc F^-_{\mathrm
  e,1}$ is uniquely achieved at $\varphi^*_-$.
  
The argument for $\mc F^+$ is the same and the last statement follows
from a direct computation.
\end{proof}

By using the variational convergence of $\mc G_\epsilon$ to $\mc G$
and the above lemma, we finally show, provided $\epsilon$ is small
enough and $\rho$ is suitably chosen, the functional  
$\mc G_\epsilon(\rho,\cdot)$ admits more than a single minimizer.

\begin{proof}[Proof of Theorem~\ref{t:degmine}]
  Fix a continuous function $\rho\in M$ bounded away from zero and one
  satisfying the condition stated above Lemma~\ref{t:mina} as well as
  the condition in the last statement of that lemma. Pick now a
  continuous positive function $\lambda:[0,1]\to \bb R_+$ such that
  $\mathrm{supp}\,\lambda \subset (y_-,y_0)$ and $\int_0^1\!dx \,
  \lambda =1$. Choose finally $\delta>0$ so small that for each
  $\alpha\in [-\delta,\delta]$ the function $\rho^{(\alpha)} := \rho +
  \alpha\lambda$ still satisfies the condition stated above
  Lemma~\ref{t:mina} with $y_0$ and $y_\pm$ independent on $\alpha$.
  Note however that for $\alpha\neq 0$ the functional $\mc
  G(\rho^{(\alpha)},\cdot)$ has a unique minimizer.
  The proof of the theorem will be accomplished by considering the one
  parameter family of functions $\{\rho^{(\alpha)}, |\alpha|\le
  \delta\}$ and showing that, for $\varepsilon$ small enough, there
  exists $\alpha_0$ for which $\mc
  G_\varepsilon(\rho^{(\alpha_0)},\cdot)$ has at least two distinct
  minimizers.

  Given $\epsilon>0$, let $g_\varepsilon : [-\delta,\delta]\to \bb R$
  be defined by
  \begin{equation*}
    g_\varepsilon(\alpha) \; := \; 
    \inf \big\{{\mc G}_\epsilon(\rho^{(\alpha)},\varphi)\,,\:
     {\varphi\in \mc F^+}  \big\} 
    - 
    \inf\big\{{\mc G}_\epsilon(\rho^{(\alpha)},\varphi)\,,\:
     {\varphi\in \mc F^-}  \big\} 
  \end{equation*}
  and observe that, in view of Lemma~\ref{t:cont}, 
  the function $g_\epsilon$ is continuous.
  Let $\mathrm{int\,} \mc F^\pm$ be the \emph{interior} of $\mc
  F^\pm$. By Theorem~\ref{t:gcg} and standard properties of
  $\Gamma$--convergence, see e.g.\ \cite[Prop.~1.18]{Braides}, 
  for each $\alpha\in [-\delta,\delta]$ 
  \begin{equation*}
    \inf_{\mc F^\pm} \mc G(\rho^{(\alpha)},\cdot) 
    \le \liminf_{\epsilon\downarrow 0} 
        \inf_{\mc F^\pm} {\mc G}_\epsilon(\rho^{(\alpha)},\cdot)
    \le \limsup_{\epsilon\downarrow  0} 
        \inf_{\mathrm{int\,} \mc F^\pm} {\mc G}_\epsilon(\rho^{(\alpha)},\cdot)
    \le \inf_{\mathrm{int\,} \mc F^\pm} \mc G(\rho^{(\alpha)},\cdot)  
  \end{equation*}
  Whence, using $\varphi^*_\pm \in \mathrm{int\,} \mc F^\pm$  and
  Lemma~\ref{t:mina},
  \begin{equation*}
    \lim_{\epsilon\downarrow  0} 
    \inf_{\varphi\in\mc F^\pm} 
    {\mc G}_\epsilon(\rho^{(\alpha)},\varphi) 
    = \mc G(\rho^{(\alpha)},\varphi^*_\pm)
  \end{equation*}
  so that
  \begin{equation*}
    \begin{split}
    \lim_{\epsilon\downarrow  0} g_\epsilon(\delta) 
    &= 
    \mc G(\rho^{(\delta)},\varphi^*_+) -  \mc G(\rho^{(\delta)},\varphi^*_-) 
    \\
    &=  \mc G (\rho,\varphi^*_+) -  \mc G(\rho,\varphi^*_-) 
    -\delta \int_0^1 \lambda \, (\varphi^*_+ -\varphi^*_-)\,dx
    =\delta (\varphi_1-\varphi_0)
    \end{split}
  \end{equation*}
  In particular, there exists $\epsilon_1\in (0,\epsilon_0)$ such that
  for any $\epsilon\in(0,\epsilon_1)$ we have $g_\epsilon(\delta)>0$
  and, by the same argument, $g_\epsilon(-\delta)<0$.

  Applying the theorem on the existence of zeros for a continuous
  function of a real variable, we deduce that for each 
  $\varepsilon \in (0,\epsilon_1)$ 
  there exists $\alpha_0\in (-\delta,\delta)$ such that
  $g_\epsilon(\alpha_0)=0$. This implies the existence of at least two
  distinct minimizers for $\mc G_\epsilon(\rho^{(\alpha_0)},\cdot)$.
  Note indeed that although the sets $\mc F^+$ and $\mc F^-$ are not
  disjoint, the minimizer of ${\mc G}_\epsilon(\rho^{(\alpha_0)},\cdot)$   
  over $\mc F^-$ cannot coincide with the one over $\mc F^+$ since
  they respectively converge to $\varphi^*_+$ and $\varphi^*_-$ as
  $\epsilon\downarrow 0$.  The last statement follows again from
  Theorem~\ref{t:gcg} and standard properties of $\Gamma$--convergence,
  see e.g.\ \cite[Thm.~1.21]{Braides}.
\end{proof}

\section{Hamilton--Jacobi equation}
\label{sec7}

In the context of diffusion processes in $\bb R^n$, the
quasi-potential is connected to a Hamilton-Jacobi equation. More
precisely, let $I$ be the action functional corresponding to the
Lagrangian $\bb L (x,\dot x)$ and denote by $\bb H (x,p)$ the
associated Hamiltonian.  Under suitable conditions, the
quasi-potential, as defined in \eqref{qpfw}, is then a
\emph{viscosity} solution to the Hamilton-Jacobi equation $\bb H(x,D
v) =0$ \cite{D,Pert}.  In an infinite dimensional setting, the theory
of Hamilton-Jacobi equation, in particular of stationary
Hamilton-Jacobi equation, is much less developed.  As
Theorem~\ref{t:S=V} gives a somewhat explicit expression for the
quasi-potential, the problem here discussed may reveal itself to be
a good example for the development of the theory of infinite
dimensional Hamilton-Jacobi equations. In this section, we present
some possible formulations of the Hamilton-Jacobi equation which seem
apt for the variational problems \eqref{qp-1} or \eqref{qp}.

It is not clear how to introduce a differentiable structure on the set
$M$ in such a way that the Hamiltonian $\bb H$ in \eqref{He} becomes a
well defined function on the cotangent bundle of $M$.  The
formalization of the Hamilton-Jacobi equation is thus a non trivial
issue.  In the following we consider two possibilities to circumvent
this problem.  In the first we simply consider the subset of $M$ given
by the smooth functions $\rho: [0,1]\to [0,1]$ which are bounded away
from zero and one and satisfy the boundary conditions.  In the second
one we exploit the symplectic transformation \eqref{scv} and deduce
the functional $\Lambda^*$, recall \eqref{dL*e} and \eqref{S=S-L},
solves the Hamilton-Jacobi equation with the Hamiltonian $\tilde{\bb  H}$
introduced in \eqref{tHam}.

Recall the definitions \eqref{ssg}, \eqref{Mo}, and \eqref{He}.  We
would like to claim that the functional $S^o_\epsilon$ solves the
Hamilton-Jacobi equation $\bb H (\rho, D U) =0$ in $M^o$.  In this
setting, we formulate the notion of viscosity solutions to such
Hamilton-Jacobi equation in terms of G\^ateaux sub-differentials,
recall \eqref{gsd}.

\begin{theorem}
  \label{t:gHJ}
  For each $\rho\in M^o$ we have 
  $D^\pm_\mathrm{G} S^o_\epsilon (\rho) \subset \mc H^1_0$. 
  Moreover, $S^o_\epsilon$ is a G\^ateaux viscosity solution to 
  $\bb H (\rho, D U) =0$ in $M^o$ namely, 
  the two following inequalities hold for any $\rho\in M^o$ 
  \begin{equation*}
    \begin{aligned}
      & \bb H (\rho, h) \le 0\;,
        \qquad
        \forall\: h \in D^+_\mathrm{G} S^o_\epsilon (\rho)\;,
      \\
      & \bb H (\rho, h) \ge 0\;, 
      \qquad  \forall\:h \in D^-_\mathrm{G} S^o_\epsilon (\rho)
            \;.
    \end{aligned}
  \end{equation*}
\end{theorem}

\begin{proof}
  Recall the decomposition \eqref{S=S-L} and fix $\rho\in M^o$.  It is
  straightforward to check that $\ms S$ is Gauteaux differentiable at
  $\rho$ and $D_\mathrm{G}\ms S(\rho) = s'(\rho)$. In view of
  Proposition~\ref{t:DL*} we deduce that $D^+_\mathrm{G} S^o_\epsilon (\rho)
  = s'(\rho) - \mathrm{co\:} {\ms F}(\rho)$ and
  \begin{equation*}
     D^-_\mathrm{G} S^o_\epsilon (\rho) =
     \begin{cases}
       s'(\rho) - \varphi 
       &\textrm{if ${\ms F}(\rho)=\{\varphi\}$ for some $\varphi\in\mc
         F$}
       \\
       \varnothing
       &\textrm{otherwise}
     \end{cases}
  \end{equation*}
  By Lemma~\ref{s02} and item (ii) in Proposition~\ref{t:fp}, we then
  deduce $D^\pm_\mathrm{G} S^o_\epsilon (\rho) \subset \mc H_0^1$.

  By item (iii) in Proposition~\ref{t:fp}, $\bb
  H(\rho,s'(\rho)-\varphi)=0$ for any $\rho\in M^o$ and any
  $\varphi\in \ms P(\rho)$. In particular, again by Lemma~\ref{s02},
  this equation holds for any $\rho\in M^o$ and $\varphi \in {\ms F}(\rho)$.  
  The explicit expression of the G\^ateaux sub-differentials above and
  the convexity of $\bb H(\rho,\cdot)$ on $\mc H_0^1$ now easily yield the
  statements.
\end{proof}

Since in our case $\upbar{\rho}_\epsilon$ is the unique, globally
attractive, fixed point of the flow defined by \eqref{eq:1}, we can
try to characterize the quasi-potential $V_\epsilon$, as defined in
\eqref{qp-1}, as the \emph{maximal} viscosity sub-solution of the
Hamilton-Jacobi equation $\bb H (\rho, DU) = 0$ satisfying
$U(\upbar{\rho}_\epsilon)=0$.  Our next result, which does not depend
on the special form of the flux $f$ and mobility $\sigma$, goes in
this direction.

\begin{theorem}
  \label{t:mHJ}
  Let $U_\epsilon: M \to \bb R$ be a lower semicontinuous functional
  such that $U_\epsilon(\upbar{\rho}_\epsilon)=0$ and satisfying the
  following condition.
  For each $T>0$ and each smooth path $u:[-T,0]\times[0,1]\to (0,1)$
  such that $u(t,0)=\rho_0$, $u(t,1)=\rho_1$, $t\in[-T,0]$, consider a
  partition  $[-T,0) = \bigcup_{k=1}^n [-T_{k},-T_{k-1})$ with 
  $T_0=0$ and $T_{n}=T$. Then there are elements $h^k\in \mc H^1_0$,
  $k=1,\ldots,n$ which are uniformly bounded in $\mc H^1_0$ and such
  that 
  \begin{equation}
    \label{subs}
    \bb H\big(u(-T_k),h^k \big) \le 0\;, \qquad k=1,\ldots,n
  \end{equation}
  and 
  \begin{equation}
    \label{condbrutta}
    \limsup_{\max |T_{k-1}-T_k| \to 0 } \; 
    \sum_{k=1}^n \Big[ 
    U_\epsilon ( -T_{k-1}) -U_\epsilon ( -T_{k}) 
    -\int_{-T_k}^{-T_{k-1}} \langle h^k,u_t\rangle \,dt \Big]  \le 0 \;.
\end{equation}
  Then $V_\epsilon\ge U_\epsilon$.
\end{theorem}

Note that condition \eqref{condbrutta} basically requires that $h^k$
belongs to the super-dif\-fe\-ren\-tial of $U_\epsilon$ at the point
$u(-T_k)$ with a uniform control over the smooth path $u(t)$, $t\in
[-T,0]$. Equation \eqref{subs} is thus a viscosity formulation of $\bb
H (\rho, DU_\epsilon) \le 0$.  As a matter of fact, the proof of the
above theorem is quite similar to the one of the lower bound in
Theorem~\ref{t:S=V} proven in Section~\ref{sec4} and we only sketch
the argument.

\begin{proof}
Let $U_\epsilon: M \to \bb R$ be as in the statement of the theorem. 
In view of Lemma \ref{dId} and the lower semicontinuity of $U_\epsilon$, it is
enough to prove the following statement. Fix $\rho\in M^o$, $T>0$, and
a smooth path $u\in C([-T,0];M^o)$ such that
$u(-T)=\upbar\rho_\epsilon$, $u(0)=\rho$; then $I_{[-T,0]}(u)\ge
U_\epsilon(\rho)$.

Consider a partition $[-T,0) = \bigcup_{k=1}^n [-T_{k},-T_{k-1})$ with
$T_0=0$ and $T_{n}=T$.  For $k=1,\cdots,n$, let $h^k \in \mc H^1_0$ as
in the statement of the theorem and choose in \eqref{f04} the
piecewise constant path $H(t)= h^{k}$ for $t\in [-T_{k},-T_{k-1})$,
$k=1,\cdots,n$ and $H(0) = h^1$.  By assumption, $\bb H(u(-T_k),h^k)
\le 0$. Whence, by the smoothness of the path $u$ and the assumption
that $|h^k|_{\mc H_0^1}$ is uniformly bounded,
\begin{equation*}
\limsup_{n\to\infty} \sum_{k=1}^n \int_{-T_{k}}^{-T_{k-1}}
\bb H \big( u(t) , h^k\big)\, dt\; \le \; 0
\end{equation*}
provided the mesh of the partition vanishes as $n\to\infty$.

Again by assumption,
\begin{equation*}
\liminf_{n\to\infty}    
\sum_{k=1}^n \Big\{ \int_{-T_{k}}^{-T_{k-1}} \langle h^k, u_t \rangle
\, dt - \big[ U_\epsilon(u(-T_{k-1})) - U_\epsilon(u(-T_{k})) \big] \Big\}
\;\ge \; 0
\end{equation*}
provided the mesh of the partition vanishes as $n\to\infty$.
We then deduce that
\begin{equation*}
I_{[-T,0]}(u) \ge U(u(0)) - U(u(-T)) = U(\rho)\;,
\end{equation*}
which concludes the proof. 
\end{proof}

We next discuss a formulation of the Hamilton-Jacobi equation in a
strong topology and in the whole set $M$. Recalling \eqref{S=S-L}, the
basic idea is to look for an equation for $\Lambda^*$, equivalently to
exploit the symplectic transformation \eqref{scv}.

Consider the space $L^1([0,1])$ equipped with the strong topology and
denote by $\hat M$ its (closed) subset given by $\hat M := \{ \rho\in
L^1([0,1]) :\, 0\le \rho\le 1\}$.  There is nothing really peculiar
with the choice of $L^1([0,1])$, any $L^p([0,1])$ with $p\in
[1,\infty)$ would lead to the same results; $L^\infty([0,1])$ might
look more natural, but its lack of separability (with respect to the
strong topology) prevents its use.  Let
\begin{equation*}
\mc W  :=\big\{ w\in W^{2,1}([0,1])\,:\: 
w(0)= \varphi_0 \,,\, w(1)= \varphi_0 \big\}\;,
\end{equation*}
where $W^{2,1}([0,1])$ is the Sobolev space of the functions
whose second (weak) derivative belongs to $L^1([0,1])$.  
Let $\hat{\bb H} : \hat M \times\mc W \to \bb R$ 
be the Hamiltonian defined by  
\begin{equation}
\label{dhH}
\begin{split}
\hat{\bb H} (\rho,w) \; &:=\; 
\;-\; \epsilon\, \langle w_x, \sigma(\rho) w_x\rangle 
\;-\; \epsilon\, \langle \rho, w_{xx} \rangle \\
& \phantom{:=\;}\; +\; \langle  \sigma(\rho), w_x \rangle
\;-\; \rho_1 \big[ 1-\epsilon w_x(1) \big] 
\;+\; \rho_0 \big[ 1-\epsilon w_x(0) \big] \;. 
\end{split}
\end{equation}
The relationship of $\hat{\bb H}$ to the original Hamiltonian $\bb H$
is the following.  Consider the anti-symplectic transformation
$(\rho,h)\mapsto (\rho,w)$ where $w=s'(\rho)-h$. The associated
Hamiltonian $\hat{\bb H}(\rho,w) = - \bb H(\rho,s'(\rho)-w)$ is then
the one defined above. We remark that while the identity 
$\hat{\bb H}(\rho,w) = - \bb H(\rho,s'(\rho)-w)$ holds only for
$\rho\in M^o$, recall \eqref{Mo}, the Hamiltonian $\hat{\bb H}$ is defined for all $\rho$
in $\hat M$.

We claim that the functional $\hat \Lambda$ in \eqref{dhLe} is a
viscosity solution of the Hamilton-Jacobi equation $\hat{\bb H} (\rho,
DU) =0$, $\rho\in \hat M$.  In order to state precisely this result,
we recall the notion of Fr\'echet sub-differentials.  Let $B$ be a
Banach space and denote by $B^*$ its dual. The norms in $B$ and $B^*$
are respectively denoted by $|\cdot|_B$ and $|\cdot|_{B^*}$.  A
function $f:B\to \bb R$ is \emph{Fr\'echet differentiable} at $x\in B$
iff there exists $\ell\in B^*$ such that
\begin{equation*}
  \lim_{y\to x} \frac{ f(y)- f(x) - \langle\ell, y-x\rangle}{|y-x|_B} =0  
\end{equation*}
in this case we denote $\ell$ by $D_\mathrm{F} f(x)$.  In general, the
\emph{Fr\'echet sub-differential} $D^-_\mathrm{F} f(x)$ and
\emph{Fr\'echet super-differential} $D^+_\mathrm{F} f(x)$ of $f$ at
the point $x$ are defined as the, possibly empty, convex subsets of
$B^*$
\begin{equation*}
\begin{aligned}
& D^+_\mathrm{F}f(x)
:=
\Big\{\ell\in B^*\,:\;
\limsup_{y\to x} \: \frac{f(y)-f(x) - \langle\ell, y-x \rangle  }{|y-x|_B}
\leq 0 \Big\}\,, \\
& D^-_\mathrm{F} f(x)
:=\Big\{ \ell\in B^*\,:\;
\liminf_{y\to x}\: \frac{f(y)-f(x) - \langle\ell, y-x \rangle  }{|y-x|_B}
\ge 0 \Big\}\;.
\end{aligned}
\end{equation*}

\begin{theorem}
  \label{t:fHJ}
  The functional  $\hat{\Lambda}
  :\big(L^1([0,1]),\mathrm{strong}\big) \to \bb R$ is convex and
  continuous, in particular locally Lipschitz.  Moreover, for each
  $\rho\in L^1([0,1])$ we have $D^\pm_\mathrm{F}
  \hat{\Lambda} (\rho) \subset \mc W$.  Finally,
  $\hat{\Lambda}$ is a Fr\'echet viscosity solution of
  $\hat{\bb H} (\rho, D U) =0$ in $\hat M$ namely, the two
  following inequalities hold for any $\rho\in \hat M$
  \begin{equation*}
    \begin{aligned}
      & \hat{\bb H} (\rho, w) \le 0\;,
        \qquad
        \forall\: w \in D^+_\mathrm{F} \hat{\Lambda} (\rho)\;,
      \\
      & \hat{\bb H} (\rho, w) \ge 0 \;,
      \qquad  \forall\:w \in D^-_\mathrm{F} \hat{\Lambda} (\rho)
            \;.
    \end{aligned}
  \end{equation*}
\end{theorem}

\begin{proof}
  Clearly $\hat\Lambda$ is convex and, by the argument used in
  Lemma \ref{t:cont}, continuous.  Fix $\rho\in \hat M$.
  Proposition~\ref{t:DL*} and the argument presented below
  Theorem~\ref{t:lr} show that $D^-_\mathrm{F}\hat\Lambda(\rho) =
  \mathrm{co\:} {\ms F}(\rho)$ and
  $D^+_\mathrm{F}\hat\Lambda(\rho)=\{\varphi\}$ when
  ${\ms F}(\rho)=\{\varphi\}$,
  $D^+_\mathrm{F}\hat\Lambda(\rho)=\varnothing$ otherwise.  In
  particular, by the inclusion ${\ms F}(\rho)\subset \ms P(\rho)$ and
  item (ii) in Proposition~\ref{t:fp},
  $D^\pm_\mathrm{F}\hat\Lambda(\rho)\subset \mc W$.
  
  As in the proof of Lemma~\ref{t:int}, we consider the integral
  operator $K_\rho$ defined in \eqref{g03} also for $\rho\in
  L^1([0,1])$ and denote by $\hat{\ms P}(\rho)\subset \mc F$ the set
  of fixed points of $K_\rho$. As in Proposition~\ref{t:fp}, if
  $\varphi\in\hat{\ms P}(\rho)$ then $\varphi\in C^1([0,1])$ and there
  exist $\delta\in (0,1)$ and $C<\infty$, depending only on
  $|\rho|_{L^1}$, such that $\delta\le \epsilon \varphi_x \le 1-\delta$ and
  $|\varphi_{xx}|_{L^1}\le C$. In particular the Euler-Lagrange
  equation \eqref{Deq} holds a.e.

  We now prove the sub-solution statement, that is $\hat{\bb
    H}(\rho,w) \le 0$ for any $\rho\in \hat M$ and $w\in
  D^+_\mathrm{F} \hat\Lambda (\rho)$. Since $D^+_\mathrm{F}
  \hat\Lambda(\rho) =\varnothing$ if ${\ms F}(\rho)$ is not a singleton, we
  need only to consider the case ${\ms F}(\rho)=\{\varphi\}$ for some
  $\varphi\in\mc F$. Hence, as shown in the proof of
  Lemma~\ref{t:int}, $\varphi\in\hat{\ms P}(\rho)$. 
  It is therefore enough to show that 
  for any $\rho\in\hat M$ and $\varphi\in \hat{\ms P}(\rho)$ 
  we have $\hat{\bb H}(\rho,\varphi)=0$. 
  This is essentially the same computation as
  the one used in the proof of item (iii) in Proposition~\ref{t:fp}.
  Recall $\varphi_i=\log[\rho_i/(1-\rho_i)]$, $i=0,1$ and observe that,
  in view of the bounds stated above, if $\varphi\in \hat{\ms
    P}(\rho)$ then 
  \begin{equation*}
    \begin{aligned}
    &\rho_1\big[1-\epsilon\varphi_x(1)]-\rho_0\big[1-\epsilon\varphi_x(0)\big]
     = \Big\langle \Big(\frac{e^\varphi}{1+e^\varphi}\Big)_{\! x} \,,\,
     1-\epsilon\varphi_x \Big\rangle 
     - \Big\langle \frac{e^\varphi}{1+e^\varphi}\,,\, 
     \epsilon \varphi_{xx}\Big\rangle
     \\
     & \qquad\qquad
     = \Big\langle \frac{e^\varphi}{\big( 1+e^\varphi \big)^2 }\,,\,
     \varphi_x (1-\epsilon\varphi_x) \Big\rangle 
     - \Big\langle \frac{e^\varphi}{1+e^\varphi}\,,\, 
     \epsilon\varphi_{xx}\Big\rangle
    \end{aligned}
  \end{equation*}
  so that, recalling \eqref{dhH}, 
  \begin{equation*}
     \hat{\bb H} (\rho,\varphi) = 
     \Big\langle \sigma(\rho) - 
     \frac{e^\varphi}{\big( 1+e^\varphi \big)^2} \,,\,
       \varphi_x (1-\epsilon\varphi_x) \Big\rangle 
       +\Big\langle \frac{e^\varphi}{1+e^\varphi} - \rho \,,\, \epsilon
       \varphi_{xx} \Big\rangle
  \end{equation*}
  At this point we use the special form of $\sigma$, i.e.\
  $\sigma(\rho)=\rho(1-\rho)$, which implies 
  \begin{equation*}
    \sigma(\rho) - \frac{e^\varphi}{\big( 1+e^\varphi \big)^2}
      = \Big( \frac{e^\varphi}{1+e^\varphi} - \rho \Big)
      \, \Big(\rho - \frac{1}{1+e^\varphi}\Big)
  \end{equation*}
  We then deduce 
  \begin{equation*}
    \hat{\bb H} (\rho,\varphi) = 
    \Big\langle  \frac{e^\varphi}{1+e^\varphi} - \rho \,,\,
    \epsilon \varphi_{xx} + \varphi_x(1-\epsilon\varphi_x) 
    \Big(\rho - \frac{1}{1+e^\varphi}\Big) \Big\rangle = 0
  \end{equation*}
  since $\varphi$ satisfies \eqref{Deq} a.e.

  It remains to prove the super-solution statement, that is $\hat{\bb
    H}(\rho,\varphi) \ge 0$ for any $\rho\in \hat M$ and $\varphi\in
  D^-_\mathrm{F} \hat\Lambda (\rho)$. To this end, consider first the
  case $\varphi\in \partial_{\mathrm{app}} \hat\Lambda (\rho)$. By
  definition, there exists a sequence $\{\rho^n\}\subset L^1([0,1])$
  converging to $\rho\in\hat M$ strongly in $L^1([0,1])$ such that 
  $\partial \hat\Lambda (\rho^n)=\{\varphi^n\}$ and $\varphi^n\to
  \varphi$ weakly* in $L^\infty([0,1])$. In particular, 
  $\varphi^n\in \hat{\ms P}(\rho^n)$ and therefore, by the bounds
  stated at the beginning of this proof,  
  $\varphi\in \hat{\ms P}(\rho)$. By the computation presented
  above, we then deduce $\hat{\bb H}(\rho,\varphi)=0$.
  We now consider the general case 
  $\varphi\in D^-_\mathrm{F}\hat\Lambda(\rho)
  =\partial\hat\Lambda(\rho)$. The concavity of 
  $\hat{\bb H}(\rho,\cdot)$ implies $\hat{\bb H}(\rho,\varphi)\ge 0$ for any
  $\varphi\in \mathrm{co\:} \partial_{\mathrm{app}} \hat\Lambda
  (\rho)$. In view of Theorem~\ref{t:lr}, it is now enough to take the
  weak* closure in $L^\infty([0,1])$. This is easily accomplished
  noticing there exist $\delta\in (0,1)$ and $C<\infty$ depending only
  on $\rho$ such that any $\varphi\in \mathrm{co\:}
  \partial_{\mathrm{app}} \hat\Lambda (\rho)$ satisfies $\delta\le
  \epsilon\varphi_x\le 1-\delta$ and $|\varphi_{xx}|_{L^\infty}\le C$.  
\end{proof}

\bigskip
\subsection*{Acknowledgements}
L.\ Bertini acknowledges the very kind hospitality at IMPA.
D.\ Gabrielli acknowledges the financial support of PRIN 
20078XYHVYS.

\end{document}